\documentclass[10pt]{amsart}
\usepackage{amssymb,amsmath,amsopn,textcomp, thmtools}
\usepackage[foot]{amsaddr}
\usepackage{aliascnt}
\usepackage{mathrsfs} 

\theoremstyle{plain}
\newtheorem{theorem}{Theorem}[section]
\newtheorem{defi}[theorem]{Definition}

\newtheorem{lemma}[theorem]{Lemma}

\newtheorem{corollary}[theorem]{Corollary}
\newtheorem{proposition}[theorem]{Proposition}
\theoremstyle{definition}
\newtheorem{remark}[theorem]{Remark}
\newtheorem{example}[theorem]{Example}

\newcommand{\iR}{\mathbb{R}}
\newcommand{\iN}{\mathbb{N}}
\newcommand{\iZ}{\mathbb{Z}}

\newcommand{\cH}{\mathcal{H}}
\newcommand{\cP}{\mathcal{P}}
\newcommand{\cI}{\mathcal{I}}
\newcommand{\cW}{\mathcal{W}}

\usepackage{bbm}
\newcommand{\R}{\mathbb{R}}
\newcommand{\N}{\mathbb{N}}
\newcommand{\Z}{\mathbb{Z}}

\newcommand{\cA}{\mathcal{A}}
\newcommand{\cB}{\mathcal{B}}

\usepackage{extarrows}
\usepackage{dsfont}
\usepackage{tikz,subcaption} 

\usepackage{mathabx} 

\allowdisplaybreaks

\begin{document}

\title{The entropy method under curvature-dimension conditions in the spirit of Bakry-\'Emery in the discrete setting of Markov chains}
\date{\today}
\author{Frederic Weber}
\email{frederic.weber@uni-ulm.de}
\author{Rico Zacher$^*$}
\thanks{$^*$Corresponding author. F.\ W.\ is supported by a PhD-scholarship of the ``Studienstiftung des
deutschen Volkes'', Germany. R.\ Z.\ is supported by the DFG (project number 355354916, GZ ZA 547/4-2).}
\email[Corresponding author:]{rico.zacher@uni-ulm.de}
\address[Frederic Weber, Rico Zacher]{Institut f\"ur Angewandte Analysis, Universit\"at Ulm, Helmholtzstra\ss{}e 18, 89081 Ulm, Germany.}


\begin{abstract}
We consider continuous-time (not necessarily finite) Markov chains on discrete spaces and identify a curvature-dimension inequality, the condition $CD_\Upsilon(\kappa,\infty)$, which serves as a natural
analogue of the classical Bakry-\'Emery condition $CD(\kappa,\infty)$ in several respects. In particular, it is tailor-made
to the classical approach of proofing the modified logarithmic Sobolev inequality via computing and estimating the
second time derivative of the entropy along the heat flow generated by the generator of the Markov chain. We prove that
curvature bounds in the sense of $CD_\Upsilon$ are preserved under tensorization, discuss links to other notions of discrete curvature and consider a variety of examples including complete graphs, the hypercube and birth-death processes. 
We further consider power type entropies and determine, in the same spirit, a natural CD condition which leads to 
Beckner inequalities. The $CD_\Upsilon$ condition is also shown to be compatible with the diffusive setting, in the sense that
corresponding hybrid processes enjoy a tensorization property.
\end{abstract}

\maketitle

 

\section{Introduction and main results}
\subsection{The  $\Gamma$-calculus by Bakry and \'Emery}
Curvature-dimension (CD) conditions play a central role in the study of Markov semigroups
and operators and related functional inequalities like, e.g., the spectral gap or logarithmic Sobolev inequality (cf.\ \cite{BGL}). They encode important information about the geometric
properties (curvature and dimension) of the underlying state space $X$, e.g.\ a Riemannian manifold, and thus
also constitute an important tool in geometric analysis (\cite{Li}).

The classical CD condition
has been introduced by Bakry and \'Emery in the pioneering work \cite{BaEm} and is formulated 
by means of the so-called $\Gamma$-calculus. Denoting by $L$ the infinitesimal generator of a Markov semigroup,
the carr\'e du champ operator $\Gamma$ and
the iterated carr\'e du champ operator $\Gamma_2$
are defined on a suitable algebra of functions as
\begin{align*}
\Gamma(f,g) & =\frac{1}{2}\big(L(fg)-fLg-gLf\big),\label{def:Gamma(f,g)} \\
\Gamma_2(f,g) & =\frac{1}{2}\big(L\Gamma(f,g)-\Gamma(f,Lg)-\Gamma(Lf,g)\big),\nonumber
\end{align*}
and one sets $\Gamma(f)=\Gamma(f,f)$ and $\Gamma_2(f)=\Gamma_2(f,f)$. The operator $L$ is said to
satisfy the curvature-dimension inequality $CD(\kappa,d)$, for $\kappa\in \iR$ and $d\in [1,\infty]$,
if for every  $f$ in a sufficiently rich class of functions $\cA$,
\begin{equation} \label{BEDEF}
\Gamma_2(f)\ge \frac{1}{d}\,\big(Lf)^2+\kappa \Gamma(f),\quad \mu-\mbox{a.e.},
\end{equation}
where $\mu$ is a fixed invariant and reversible measure on $X$ for the semigroup $e^{Lt}$. 

It is well known that various fundamental results for Markov semigroups can be obtained 
under a certain CD condition in the Bakry-\'Emery sense. For example, positive curvature, i.e.\
$CD(\kappa,\infty)$ with $\kappa>0$, implies a logarithmic Sobolev inequality, which is
equivalent to both exponential decay in entropy and hypercontractivity of the Markov semigroup
and can be also used to show the Gaussian measure concentration property for the measure $\mu$, see
 \cite[Chapter 5]{BGL}. Furthermore, finite dimension and nonnegative curvature, i.e.\ $CD(0,d)$
 with $d<\infty$, imply the Li-Yau inequality, which in turn leads to the parabolic Harnack inequality,
 cf.\ \cite[Chapter 6]{BGL}.
 
A key assumption in this rather general and powerful theory, for which the monograph \cite{BGL}
is an excellent source, is the so-called {\em diffusion property}, which means that the generator $L$
satisfies the chain rule 
\begin{equation} \label{Lchainrule}
L H(f)=H'(f)Lf+H''(f)\Gamma(f)
\end{equation}
for every $H\in C^2(\iR)$ and every sufficiently smooth function $f$. The diffusion property can be also expressed equivalently via a certain chain rule of the corresponding carr\'e du champ operator. 
The typical example is a second order differential operator on $\iR^d$ (or on a manifold) without zeroth order term.
For example, if $L=\Delta_g$ is the Laplace-Beltrami operator on a Riemannian manifold $X=(M,g)$ and 
$\mu_g$ the canonical Riemannian measure, then using the Bochner-Lichnerowicz formula one can show
that $CD(\kappa,d)$ is equivalent to Ric$_g(x)\ge \kappa g(x)$ for almost every $x\in M$ and dim$\,M\le d$.

If one wants to extend the theory from \cite{BGL} to the discrete setting of Markov chains,
which is highly desirable, then
the {\em failure} of the diffusion property creates a major difficulty.  The same problem occurs in case
of non-local operators like, e.g., the fractional Laplacian on $\iR^d$, where the setting is continuous but the
chain rule is violated. Moreover, even though it is still possible to define the operators $\Gamma$ and $\Gamma_2$
and to deduce positive results (like e.g.\ eigenvalue or diameter estimates) by means of the $\Gamma$-calculus (see e.g.\ 
\cite{LinYau,LMP,LiPe}),
the Bakry-\'Emery condition $CD(\kappa,d)$, in general, does not have such strong implications or is not applicable as in the diffusion setting, see e.g.\ \cite{Harvard} and \cite[Chapter 5]{Jn} concerning the Li-Yau inequality on graphs and discrete entropy methods, respectively.
These facts suggest that the classical Bakry-\'Emery CD inequality is not the natural condition in the
discrete setting, in particular in situations where the {\em logarithm} plays a crucial role like in Li-Yau inequalities, entropy or the logarithmic Sobolev inequality.

The problem of finding an appropriate substitute for the 
classical $\Gamma$-calculus in the discrete setting has stimulated a lot of research in the last decade and a half.   
Several notions of generalized Ricci curvature (lower bounds) for metric measure spaces including discrete spaces have been proposed and studied, see e.g.\ the recent book \cite{ModCur}. 

\subsection{Notions of Ricci curvature via optimal transport}
One important approach is based on geodesic convexity of the entropy. It relies on the crucial observation made in \cite{vRS} that on a Riemannian manifold $M$, the Ricci curvature is bounded from below by $\kappa\in \iR$ if and only if the Boltzmann
entropy $\cH(\rho)=\int_M \rho \log \rho\,d\mu_g$ is $\kappa$-convex along geodesics in the Wasserstein space
$P_2(M)$; we also refer to the earlier important contributions \cite{CEMcCS,McCann,OttoVill}. This characterization opened up
the synthetic theory of Ricci curvature bounds in geodesic metric measure spaces established independently by
Lott-Villani \cite{LottVill} and Sturm \cite{Stu1,Stu2}, see also the monograph \cite{Vil} and \cite{AGSInv,AGS14,AGS15,
EKS}.

However, this theory does not apply directly to discrete spaces,
since the $L^2$-Wasserstein space over a discrete space does not contain geodesics. This problem was solved for finite 
Markov chains by
Erbar and Maas \cite{EM12} and Mielke \cite{Mie} by replacing the $L^2$-Wasserstein metric  by a different
metric $\cW$, which had been introduced before by Maas \cite{Maa1} (see also Mielke \cite{Mie}). 
The metric $\cW$
is constructed in such a way that the heat flow associated with the Markov kernel of the chain is the gradient flow of 
the entropy with respect to the metric $\cW$, thereby obtaining a discrete analogue of the celebrated
Wasserstein gradient flow interpretation of the heat flow in $\iR^n$ by Jordan, Kinderlehrer and Otto \cite{JKO}.
It is shown that, together with the metric $\cW$, the space of probability measures on the finite state space $X$ becomes a geodesic space. Using this property, Erbar and Maas \cite{EM12} (see also Maas \cite{Maa1}) define lower Ricci curvature bounds by means of geodesic convexity of the entropy and prove a variety of discrete analogues of important results
from the continuous setting including that positive curvature implies the (modified) logarithmic Sobolev and Talagrand inequality 
as well as the Poincar\'e inequality. They also show that Ricci curvature bounds enjoy a tensorization principle and obtain,
as a special case, sharp Ricci curvature lower bounds for the discrete hypercube. Recently, various functional inequalities have also been proved under nonnegative discrete Ricci curvature and some suitable additional conditions (\cite{ErFa}).
By means of different methods,
discrete entropic Ricci curvature bounds have been derived for several examples, e.g.\ one-dimensional birth-death processes,
Bernoulli-Laplace models, zero-range processes and random transposition models, see \cite{EMT,FaMa,Mie}.
We also refer to the survey \cite{MaaSurvey} and the references given therein.

A different approach to adapt the Lott-Sturm-Villani theory to the discrete setting has been proposed by Bonciocat
and Sturm in \cite{BonStu} and is based on the notion of approximate midpoints w.r.t.\ the $L^2$-Wasserstein metric.
This notion is also studied in \cite{OllVill} in connection with a Brunn-Minkowski inequality on the discrete hypercube.

Another important notion of Ricci curvature including the discrete case has been introduced by Ollivier \cite{Oll}. This approach also relies on optimal transport and makes use of the $L^1$-Wasserstein metric. Positive Ollivier Ricci curvature also implies
certain functional inequalities. In contrast to the entropic bounds in \cite{EM12}, Ollivier's condition does not require
reversibility. As shown in \cite{Oll}, his criterion can be checked without much effort in many examples.
It does not seem to be known whether Ollivier's notion is connected to the entropic one.
For further studies of Ollivier's Ricci curvature in the graphical setting we refer to \cite{HJL,JoLiu,LinYau,MünWo}.

\subsection{CD inequalities in the context of discrete Li-Yau gradient estimates} 
Recently, motivated mainly by the problem of finding a discrete version of the celebrated Li-Yau inequality (\cite{LiYau}), various modifications of the classical Bakry-\'Emery CD condition have been introduced (\cite{Harvard,DKZ,Mn1}). 
We recall that positive solutions $u$ of the heat equation $\partial_t u-\Delta u=0$ on $(0,\infty)\times \iR^d$
satisfy the (sharp) Li-Yau inequality $-\Delta \log u\le \frac{d}{2t}$, which is equivalent to the {\em differential
Harnack inequality} 
\begin{equation} \label{CLY}
\partial_t(\log u)\ge |\nabla (\log u)|^2-\frac{d}{2t}\quad \mbox{in}\;(0,\infty)\times \iR^d.
\end{equation}
By integration along a suitable path in space-time, the gradient estimate \eqref{CLY} yields a sharp parabolic Harnack estimate.
Corresponding results hold true for the heat equation on a manifold with nonnegative Ricci curvature, see \cite{Li, LiYau}. 

First results concerning Li-Yau inequalities on graphs were obtained by Bauer, Horn, Lin, Lippner, Mangoubi, and Yau \cite{Harvard}.
Their key idea to circumvent the problem of the failure of the chain rule for the logarithm is to consider the {\em square root} instead of the logarithm and to use the identity
\begin{equation} \label{sqrt}
L \sqrt{f}=\frac{L f}{2\sqrt{f}}-\frac{\Gamma(\sqrt{f})}{\sqrt{f}},
\end{equation}
which still holds true for the graph Laplacian $L$. Furthermore, the authors of \cite{Harvard} introduce a new notion of CD-inequality, the so-called {\em exponential
curvature-dimension inequality} $CDE(\kappa,d)$, which means that for any vertex $x$ and any
positive function $f$ defined on the set of vertices $X$ satisfying $(L f)(x)<0$ one has
\begin{equation} \label{CDE}
\tilde{\Gamma}_2(f)(x):=\Gamma_2(f)(x)-\Gamma\Big(f,\frac{\Gamma(f)}{f}\Big)(x)\ge \frac{1}{d} (L f)(x)^2+\kappa\, \Gamma(f)(x).
\end{equation}
By means of \eqref{sqrt} and their new notion of CD-inequality the authors of \cite{Harvard},
using maximum principle arguments, obtain various Li-Yau type results for graphs. They also introduce
a different type of exponential CD-inequality, the condition $CDE'(\kappa,d)$, which holds if for all positive 
functions $f$ on $X$, 
\[
\tilde{\Gamma}_2(f)\ge \frac{1}{d} f^2(L \log f)^2+\kappa\, \Gamma(f).
\] 
In \cite{HLLY}, a family of Li-Yau type estimates similar to the ones in \cite{Harvard} 
(involving again the square root) are derived under the condition $CDE'(0,d)$, now using 
semigroup techniques as in \cite{BGL}. Moreover, volume doubling, the Poincar\'e inequality and Gaussian heat kernel estimates are proved for non-negatively curved graphs.

A few years later, M\"unch \cite{Mn1,Mn2} introduced a new calculus, called $\Gamma^\psi$-calculus, where $\psi$ is a concave function. Assuming that a {\em finite} graph satisfies the so-called $CD\psi(d,0)$-condition, M\"unch proves a Li-Yau estimate which for 
$\psi=\sqrt{}$ becomes one of those obtained in \cite{Harvard} and which in case $\psi=\log$
leads to the more natural logarithmic estimate
$-L \log u \le \frac{d}{2t}$, for positive solutions $u$ of the heat equation on the graph. M\"unch shows that 
in some examples (e.g.\ Ricci-flat graphs) the latter choice gives better estimates than the square-root approach.

Very recently, Dier, Kassmann and Zacher \cite{DKZ} proposed another CD inequality in the graphical setting, 
the condition $CD(F;0)$, which allows 
to significantly improve the Li-Yau estimates in \cite{Harvard,Mn1}. For example, they obtain a sharp logarithmic
Li-Yau estimate for the unweighted two-vertex graph and prove a Li-Yau inequality for the lattice $\tau \iZ$ with grid size $\tau$
which for $\tau\to 0$ leads to the classical sharp Li-Yau inequality on $\iR$.

One of the new key ideas in \cite{DKZ} is to replace the square $(L f)^2$ in the 
CD-inequality by a more general term of the form $F(-Lf)$, where $F$ is a so-called {\em CD-function}.
Their Li-Yau inequality for positive solutions $u$ of the heat equation on $[0,\infty)\times X$ (in case of a finite graph
satisfying $CD(F;0)$) takes the form 
$-L \log u \le \varphi(t)$, where the {\em relaxation function} $\varphi$ is the
unique strictly positive solution of the ODE 
$\dot{\varphi}(t)+F(\varphi(t))=0$
which has $(0,\infty)$ as its maximal interval of existence. In all examples studied in \cite{DKZ}, $\varphi$ exhibits 
a logarithmic behaviour near $t=0$ and thus is integrable at $t=0$ (in contrast to the $\frac{c}{t}$ estimates in
\cite{Harvard,Mn1}), which makes
it possible to deduce Harnack inequalities without time gap at $t=0$. 

\subsection{A discrete analogue of the $CD(\kappa,\infty)$ condition that fits to the entropy method}
The main purpose of the present paper is to identify a CD inequality in the discrete setting of Markov chains which
serves as a natural analogue of the classical condition $CD(\kappa,\infty)$ in several respects,
most notably with regard to the Bakry-\'Emery strategy to prove decay of the entropy (and the 
(modified) logarithmic Sobolev inequality) via computing and estimating the second time derivative of the entropy along the 
heat flow generated by the generator $L$ of the Markov chain.   

Before summarizing the main results, we describe the setting and some notation used in this paper. We consider a time-homogeneous Markov process $(Z_t)_{t\ge 0}$ on a finite or countable state space $X$, equipped 
with the maximal $\sigma$-algebra. The generator $L$ of the Markov process, defined for a suitable class of functions $f:\,X\to \iR$, is given by
\begin{equation}\label{def:generator}
Lf(x)=\sum_{y\in X}k(x,y)f(y)=\sum_{y\in X}k(x,y)\big(f(y)-f(x)\big),\quad x\in X,
\end{equation}
where $k(x,y)\ge 0$ is the transition rate for jumping from $x$ to $y$ ($\neq x$) and $\sum_{y\in X}k(x,y)=0$
for all $x\in X$, that is
$k(x,x)=-\sum_{y\in X\setminus \{x\}}k(x,y)<0$, $x\in X$. The Markov chain is called irreducible if for any $x,y \in X$ we find   $z_1,...,z_m \in X$, $m \geq 2$, with $x=z_1$ and $y=z_m$ such that $k(z_i,z_{i+1})>0$ for any $i \in \{1,...,m-1\}$. We denote the associated (Sub-)Markov semigroup on the space of bounded functions on $X$ by
${\bf P}=(P_t)_{t\ge 0}$, which is given by
\begin{equation}\label{semigroupconditionalExp}
P_t f (x) = \mathbb{E} \big( f(Z_t) | Z_0=x \big).
\end{equation}
 We recall that a $\sigma$-finite measure $\mu$ on $X$ is said to be invariant for the semigroup ${\bf P}$ 
if for any bounded function $f:\,X\to [0,\infty)$,
$\int_X P_t f\,d\mu=\int_X f\,d\mu$, $t\ge 0$.
In the recurrent case this is known to be equivalent to $\sum_{x \in X}\mu (\{x\})k(x,y)=0$. Furthermore, a $\sigma$-finite measure $\mu$ on $X$ is said to be reversible if the so-called {\em detailed balance condition}
\begin{equation}\label{detailedbalancecondition}
\mu(\{x\}) k(x,y)=\mu(\{y\})k(y,x),\quad x,y\in X,
\end{equation}
is satisfied. It is known that reversibility implies invariance. In particular, under reversibility $\int_X L f d\mu =0$ holds true on a suitable class of functions (e.g. on bounded functions if $\mu$ is finite).
For more details on the general theory of continuous-time Markov chains we refer to \cite{And} and \cite{Nor}.

The starting point in our search for suitable substitutes for $\Gamma$ and $\Gamma_2$ is the 
identity
\begin{align} \label{logchain}
L(\log f)=\frac{1}{f}\,L f-\Psi_\Upsilon (\log f)
\end{align}
for positive functions $f$, which has been observed in \cite{DKZ}. Here the operator $\Psi_\Upsilon$ is defined by 
\[
\Psi_\Upsilon(f)(x)=\sum_{y\in X}k(x,y) \Upsilon\big(f(y)-f(x)\big),\quad x\in X,
\]
with the function $\Upsilon(r)=e^r-1-r$, $r\in \iR$, see also Lemma \ref{PalmeFI} below.
Comparing \eqref{logchain} with the corresponding identity in the diffusion setting, one sees that 
$\Psi_\Upsilon (\log f)$ replaces $\Gamma(\log f)$ in the discrete setting. In other words, the quadratic
function $\frac{1}{2}r^2$, which represents the (discrete) $\Gamma$-operator, is replaced with the
function $\Upsilon(r)$, which interestingly for $r\to 0$ behaves like $\frac{1}{2}r^2$. 
The same phenomenon can be observed with the Fisher information $\cI(\rho)$ of a probability density $\rho$
w.r.t.\ to a fixed invariant and reversible probability measure $\mu$ on $X$. In fact, in the discrete case we have
$\cI(\rho)=\int_X \rho \Psi_\Upsilon(\log \rho)\,d\mu$, cf.\ Section 3 below.

Following the Bakry-\'Emery strategy, i.e.\ computing the second time derivative of the entropy, it turns out, as we show, that one can proceed further and, by means of detailed balance and certain
{\em algebraic identities for differences}, find a natural analogue of $\Gamma_2$, too, namely
\[
\Psi_{2,\Upsilon}(f)=\,\frac{1}{2}\,\big( L\Psi_\Upsilon(f)-B_{\Upsilon'}(f,Lf)\big),
\]
where 
\[
B_{\Upsilon'}(f,g)(x)=\sum_{y\in X}k(x,y) \Upsilon'\big(f(y)-f(x)\big)\big(g(y)-g(x)\big).
\]

Motivated by these findings, we introduce the CD condition $CD_\Upsilon(\kappa,\infty)$ with $\kappa\in \iR$ via the
inequality $\Psi_{2,\Upsilon}(f)\ge \kappa \Psi_\Upsilon(f)$ on $X$, $f \in \ell^\infty(X)$. That this notion is indeed a suitable discrete version of the classical $CD(\kappa,\infty)$
condition will be demonstrated in this paper by proving a variety of important results that parallel those in the diffusion setting.
In particular, we will show the following.
\begin{itemize}
\item $CD_\Upsilon(\kappa,\infty)$ with $\kappa>0$ implies the {\em modified logarithmic Sobolev inequality} with constant
$\kappa$, see Corollary \ref{mLSIoutofCDPalme}.
\item $CD_\Upsilon(\kappa,\infty)$ is characterized by the {\em gradient bounds} $\Psi_\Upsilon(P_t f) \leq e^{-2 \kappa t} P_t(\Psi_\Upsilon(f))$, cf.\ Remark \ref{gradbounds}.
\item Curvature bounds are preserved under {\em tensorization}, more precisely $CD(\kappa_i,\infty)$ for 
two chains ($i=1,2$) implies $CD(\kappa,\infty)$ with $\kappa=\min \kappa_i$ for the product chain
(Theorem \ref{theo:tensorproperty}).
\end{itemize}
See also the beginning of Section \ref{sec:basicsandCD} for the underlying basic assumptions.

Moreover, by a scaling argument, one can also see that $CD_\Upsilon(\kappa,\infty)$ implies $CD(\kappa,\infty)$
(Proposition \ref{CDPalmeandCD}). If the state space $X$ is {\em finite} and $CD(\kappa_0,\infty)$ and 
$CD_\Upsilon(\kappa,\infty)$ hold true with $\kappa_0>\kappa>0$, then for any $\varepsilon>0$ we will prove
that the entropy $\cH(P_t f)$ is bounded above by $C\exp(-2(\kappa_0-\varepsilon)t)$ for all $t\ge 0$, where the
(explicitly known) constant $C$ only depends on $\varepsilon$, $\kappa_0$, $\kappa$, the kernel $k$, the invariant measure
and the entropy of $f$, cf.\ Theorem \ref{improvedDecay}. So the rate of exponential convergence is 'almost' that given
by the Bakry-\'Emery constant $\kappa_0$. 

In Section \ref{sec:examplesection} we discuss several examples. We show positive curvature bounds with respect to the $CD_\Upsilon$ condition, e.g. for the two-point space, the  complete graph, the hypercube and some birth-death processes on $\N_0$. In case of the hypercube, we will further see that the $CD_\Upsilon$ condition yields the optimal constant in the resulting modified logarithmic Sobolev inequality. Moreover, we compare the $CD_\Upsilon$ condition with  Bakry-\'Emery curvature along many examples. While, for instance, both curvature notions coincide for the hypercube or complete bipartite graphs, the Poisson case of the birth-death process on $\N_0$ serves as an example with positive Bakry-\'Emery curvature for which $CD_\Upsilon(0,\infty)$ is best possible. We give a quite applicable condition to show $CD_\Upsilon(\kappa,\infty)$, with $\kappa>0$, for birth-death processes  which is on one hand stronger than the assumption of \cite{CaDP}, but improves on the other hand the constant in the corresponding modified logarithmic Sobolev inequality by a factor $2$. Further, we give a quite general criterion that ensures that  no (negative) lower curvature bound exists regarding the $CD_\Upsilon$ condition for a large class of examples among which are, for instance, unweighted trees that are no (finite or infinite) paths and have at least 5 vertices. This is in huge contrast to the Bakry-\'Emery case. The latter criterion allows to show that the $CD_\Upsilon$ condition is not robust with regard to small perturbations of the transition rates and, moreover, to characterize unweighted graphs with girth at least $5$ satisfying $CD_\Upsilon(0,\infty)$.

Another striking feature of the $CD_\Upsilon$ condition is its compatibility with the diffusion setting. To be more precise,
we will establish a {\em tensorization principle} for {\em hybrid processes}, that is, tensor products of a process in the diffusion setting
and a (discrete) Markov chain. If $CD(\kappa_c,\infty)$ holds for the continuous part and $CD_\Upsilon(\kappa_d,\infty)$
for the Markov chain, then for the product process we will derive the lower curvature bound $\kappa=\min\{\kappa_d,\kappa_c\}$
meaning that the inequality 
\[
(\Gamma \oplus \Psi_\Upsilon \big)_2(f) \ge (\Gamma \oplus \Psi_\Upsilon \big)(f)
\]
is satisfied for all $f$ in a sufficiently rich class of functions defined on the product space. Here 
$\Gamma \oplus \Psi_\Upsilon$ and $(\Gamma \oplus \Psi_\Upsilon \big)_2$ are natural analogues of $\Gamma$ and
$\Gamma_2$, respectively, in the hybrid continuous-discrete setting, see Section \ref{tensorhybrid}.
We will further show that positive curvature in the hybrid sense implies a corresponding modified logarithmic
Sobolev inequality and exponential decay of the entropy. As a possible application we discuss systems of linear
reaction-diffusion systems.

The logarithmic Sobolev and the Poincar\'e inequality are in some sense extreme cases of a more general family of
{\em convex Sobolev inequalities}, see \cite{AMTU, Jn}. Whereas in the diffusive setting, both inequalities follow
from $CD(\kappa,\infty)$ with $\kappa>0$, this does not seem to be known in the discrete setting for the
(modified) logarithmic Sobolev inequality; but it is true for the Poincar\'e inequality (see e.g. \cite[Section 3]{KRT}). Among those convex Sobolev inequalities, Beckner inequalities are also of high significance. Here the corresponding entropy is generated by
a power type function with exponent between $1$ and $2$. Having identified a natural CD condition 
which is sufficient for the modified logarithmic Sobolev inequality we will also determine such a CD condition 
for power type entropies and derive the corresponding Beckner inequalities, cf.\ Section \ref{BecknerSec}. In order to succeed, we follow again the Bakry-\'Emery strategy and derive a suitable representation for the second time derivative of the power type entropy along the heat flow.
 
As our approach to discrete CD inequalities is merely based on certain algebraic rules for differences, it can be naturally extended
to a large class of {\em non-local operators} in a continuous setting. We will illustrate this in Section \ref{NonlocalSec} for operators of the form  
\[
\mathcal{L}f(x)=  \int_{\Omega}\big (f(y)-f(x)\big) k(x,dy),\quad x\in \Omega,
\]
where $\Omega\subset \iR^d$ is a domain and the integral may possibly have to be understood in the principal value sense.
A prominent example for such operators is given by the fractional Laplacian on 
$\iR^d$, which satisfies the condition $CD_\Upsilon(0,\infty)$. 

\subsection{Links to other notions of (discrete) curvature bounds}
It turns out that there are quite a few close relations of the $CD_\Upsilon(\kappa,\infty)$ condition to other 
notions of (discrete) curvature and CD inequalities. Firstly, as we will prove, $CD_\Upsilon(\kappa,\infty)$ is equivalent to M\"unch's condition 
$CD\psi(\infty,\kappa)$ with $\psi=\log$, which is introduced in a side remark in \cite[Remark 3.12]{Mn1}
in the special case of the discrete Laplacian on a finite unweighted graph and which seemingly has not been used at all neither in \cite{Mn1} nor elsewhere
in the literature. Actually, at the same place, M\"unch also defines the condition $CD\psi(d,\kappa)$ with finite dimension parameter $d >0$, which for $\psi=\log$, is equivalent to the inequality
\begin{equation} \label{CDUpsgen}
\Psi_{2,\Upsilon}(f)\ge \frac{1}{d}\,(Lf)^2+\kappa \Psi_\Upsilon(f),
\end{equation}
which we will refer to as condition $CD_\Upsilon(\kappa,d)$, cf.\ Remark \ref{commentsCD}.
We emphasize that the equivalence of \eqref{CDUpsgen} and $CD\log(\infty,\kappa)$ is by no means easy to see
as the involved quantities in M\"unch's condition are defined in a highly implicit fashion, see
Remark \ref{commentsCD}(i),(ii) and Section \ref{gemischtes} for more details.

Further, condition $CD_\Upsilon(0,d)$ is also closely related to the condition $CD(F;0)$ by Dier, Kassmann and Zacher
\cite{DKZ} with a quadratic CD-function $F$. Instead of $\Psi_{2,\Upsilon}(f)(x)$, the authors of \cite{DKZ} use a sum of the latter term and an additional term which is nonnegative if one restricts the class of functions $f$ to those where $Lf$ has a 'local' minimum
at the vertex $x$, see Remark \ref{commentsCD}(iii) for the precise formulation. 

Note that these references show that the condition $CD_\Upsilon(0,d)$ (and more general versions with a CD-function
$F$ in the dimension term) is suitable to derive {\em Li-Yau and Harnack inequalities}. Thus, like the classical
Bakry-\'Emery $\Gamma$-calculus in the diffusive setting, the $CD_\Upsilon$-calculus is powerful enough to
obtain {\em both} modified logarithmic Sobolev inequalities and entropy decay as well as Li-Yau and Harnack inequalities.
To our knowledge, this is not known
for other approaches to (discrete) CD conditions in particular for those that rely on optimal transport. 

There is also an interesting connection to the generalized $\Gamma$-calculus which has been recently introduced
by Monmarch\'e in the context of {\em degenerate Markov processes} like e.g.\ {\em kinetic} Fokker-Planck diffusion in $\iR^{2d}$ and chains of interacting particles  (\cite[Section 2]{Mon}, see also \cite{Mon2}).
It is known that for such processes the classical methods based on the $\Gamma$-calculus do not work in the usual way. To overcome this problem, Villani \cite{VilH} proposed to study certain {\em distorted quantities} when the classical
quantities like entropy or variance do not behave well under the action of the semigroup, see also \cite{Bau}.
Monmarch\'e's basic strategy also follows the Bakry-\'Emery approach, but now allowing for a wider class of functionals
and operators, respectively. His motivation is to provide a general framework where the computations can be done
more easily. We will show that (in the discrete setting) for a special choice of Monmarch\'e's 
$\Phi$-operator, his general CD inequality amounts to the condition $CD_\Upsilon(\kappa,\infty)$, see Remark \ref{commentsCD}(iv). 
However, operators like $\Psi_\Upsilon$, $\Psi_{2,\Upsilon}$ or the discrete Boltzmann entropy
are apparently not the subject of the investigations in \cite{Mon, Mon2} as the author's focus lies on degenerate processes
in a continuous setting.

In the case of finite Markov chains, it is not clear whether there is a direct implication between the condition $CD_\Upsilon(\kappa,\infty)$ and the
notion of lower Ricci curvature bounds by Erbar and Maas \cite{EM12} and Mielke \cite{Mie} defined via geodesic convexity. But what we can say is that there is no $\alpha>0$ and $\beta >- \frac{\alpha}{2}$ such that the entropic Ricci curvature of \cite{EM12} with constant $\kappa>0$ does imply $CD_\Upsilon(\alpha \kappa + \beta,\infty)$, see Remark \ref{ErbarMaascompletegraph}.
Further, we know that both conditions, those of \cite{EM12} with constant $\kappa$ and $CD_\Upsilon(\kappa,\infty)$, imply the classical Bakry-\'Emery condition $CD(\kappa,\infty)$. Moreover, there is an interesting link
between both notions. Erbar and Maas introduce two functionals $\cA(\rho,\psi)$ and $\cB(\rho,\psi)$ for probability densities
$\rho$ on the state space $X$ and functions $\psi\in \iR^X$, cf.\ also Section \ref{gemischtes}. They show that the chain has curvature bounded from below by
$\kappa$ (in their sense) if and only if the inequality $\cB(\rho,\psi)\ge \kappa \cA(\rho,\psi)$ holds for all admissible
$\rho$ and $\psi$, see \cite[Theorem 4.5]{EM12}. It turns out that
\begin{equation*}
\cA(\rho,\log \rho)=\int_X \rho \Psi_{\Upsilon}(\log \rho)\,d\mu
\end{equation*}
and
\begin{equation*}
\cB(\rho,\log \rho)=\int_X \rho \Psi_{2,\Upsilon}(\log \rho)\,d\mu,
\end{equation*}
see Section \ref{gemischtes}. This shows that on one hand, the entropic condition is stronger with regard to the general
function $\psi$ instead of $\log \rho$. On the other hand, it is weaker than $CD_\Upsilon(\kappa,\infty)$ in the sense that
it is of integral form, whereas $CD_\Upsilon(\kappa,\infty)$ is a {\em pointwise condition}, which has to hold at any $x\in X$.

The paper is organized as follows. In Section \ref{sec:basicsandCD} we motivate and introduce the curvature-dimension condition $CD_\Upsilon(\kappa,\infty)$. We demonstrate that this notion fits well to the discrete entropy method and deduce modified logarithmic Sobolev inequalities or, equivalently, an entropy decay of exponential rate in Section \ref{sec:entrodecayandmLSI}. Afterwards, we show in Section \ref{tensandgrad} that a tensorization property holds for $CD_\Upsilon(\kappa,\infty)$. Section \ref{sec:examplesection} then illustrates the content with several examples. The remaining part of the paper is mainly concerned with related situations to what we have studied before. In Section \ref{BecknerSec}, we replace the Boltzmann entropy by power type entropies, introduce corresponding curvature-dimension conditions and deduce Beckner inequalities.  After that, in Section \ref{tensorhybrid}, we study hybrid processes, that is, products of processes satisfying the diffusion property and those with discrete state space. Section \ref{NonlocalSec} is devoted to an outlook how
the $CD_\Upsilon$-calculus can be extended to non-local operators in the space-continuous setting. Finally, we conclude with Section \ref{gemischtes}, where we provide more details on some links between the $CD_\Upsilon(\kappa,\infty)$ condition and other existing approaches in the literature.
\section{Setting, basic identities and CD inequalities}\label{sec:basicsandCD}
Let us begin this section by introducing our setting. We consider a time-homogeneous irreducible continuous-time Markov chain on a countable state space $X$ with unique (up to positive multiples) invariant measure $\mu$ whose generator is given by \eqref{def:generator}. Let $\pi: X \to (0,\infty)$ denote the density for $\mu$ with respect to the counting measure, i.e. $d\mu=\pi d \#$. We assume that the detailed balance condition \eqref{detailedbalancecondition} holds true.

Let $\R^X$ denote the space of real-valued functions on $X$. We say that $f \in \ell^p(\mu)$, $1\leq p < \infty$,  if $f \in \R^X$ is $p$-summable with respect to $\mu$ and $f \in \ell^\infty(X)$ if  $f \in \R^X$ is bounded. Further, $(P_t)_{t \geq 0}$ denotes the (Sub-)Markov semigroup given by \eqref{semigroupconditionalExp}.

Throughout this paper we assume that for any $x \in X $ 
\begin{equation}\label{ASSsumfinite}
M_1(x):=\sum_{y \in X \setminus \{x\}} k(x,y) <\infty
\end{equation}
and
\begin{equation}\label{ASSsecsumfinite}
M_2(x):=\sum_{y \in X \setminus \{x\}} k(x,y) \sum_{z \in X \setminus \{y\}} k(y,z) <\infty
\end{equation}
holds true.

In Section \ref{sec:entrodecayandmLSI}, \ref{BecknerSec} and \ref{tensorhybrid} we will assume that the Markov chain is positive recurrent, i.e.\ we assume in addition that $\mu$ is a probability measure and that stochastic completeness holds (that is $P_t \mathds{1} = \mathds{1}$) or in other words that the Markov chain is non-explosive. It is  known that the latter holds automatically whenever $\sup_{x \in X} M_1(x) < \infty$, see e.g.\ \cite{Nor}.

Due to the detailed balance condition, the generator of the Dirichlet form given by
\begin{equation}\label{Dirichletform}
\mathcal{E}(f,g) = \frac{1}{2}\sum_{x \in X}\sum_{y \in X} k(x,y) \big(f(y)-f(x)\big)\big(g(y)-g(x)\big) \pi(x)
\end{equation}
for $f,g$ lying in the form domain, coincides with $L$ given by \eqref{def:generator} on bounded functions which are contained in the domain of the form generator, see \cite{KeLe}. Therefore, we will also denote the $\ell^2(\mu)$ operator with $L$ in the sequel. The corresponding $\ell^2(\mu)$-semigroup generated by $L$ is an extension of the semigroup given by \eqref{semigroupconditionalExp} restricted to $\ell^\infty(X) \cap \ell^2(\mu)$. We will also use the notation $(P_t)_{t \geq 0}$ for the corresponding $\ell^2(\mu)$-semigroup. 

$L$ determines in a natural way a graph structure with vertex set $X$ and edge weights  given by $k(x,y)$ for $x,y \in X$, $x\neq y$. We say that $x,y$ are adjacent if $k(x,y)>0$. Note that irreducibility and the detailed balance condition together imply that $k(x,y)>0$ if and only if $k(y,x)>0$. We will refer to the described graph structure as the underlying or associated graph to $L$. If $k(x,y) \in \{0,1\}$ for any $x,y \in X$ with $x \neq y$, then the underlying graph to $L$ is an unweighted graph.

Now we recall a basic identity, which can be viewed as a kind 
of chain rule for the discrete operator $L$ and we refer to as the {\em (first)
fundamental identity}, cf.\ Section 2 in \cite{DKZ}. 
\begin{lemma} \label{lem:firstFI}
Let $\Omega\subset \iR$ be an open set and $f:\,X\to \iR$ such that the range of 
$f$ is contained in $\Omega$.
Let further $H\in C^1(\Omega;\iR)$ and  $f, H(f) \in \ell^1\big(k(x,\cdot)\big)$ for  $x \in X$. Then there holds
\begin{align} \label{FI1}
L\big( H(f)\big) (x)=H'(f(x)) L f(x)+\sum_{y\in X} 
k(x,y) \Lambda_H\big(f(y),f(x)\big)
\end{align}
where
\begin{align}\label{eq:def_Lambda}
\Lambda_H(w,z):=H(w)-H(z)-H'(z)(w-z),\quad w,z\in \iR.
\end{align}
\end{lemma}
This follows from a straightforward computation, cf.\ \cite[Lemma 2.1]{DKZ}. In the case of a strictly convex function $H$, the quantity $\Lambda_H(w,z)$ is called {\em Bregman distance} associated with $H$ for the points $w$ and $z$. 
Note that identity \eqref{FI1} can be regarded
as an analogue in the discrete setting of the chain rule for Markov diffusion operators which reads
\begin{equation} \label{chainL}
LH(f)=H'(f)Lf+H''(f)\Gamma(f).
\end{equation}

It turns out that for several important examples of the function $H$, the second term on the right-hand side of \eqref{FI1} can be rewritten in a form that provides more insights into the structure of the term. To this purpose it is useful to define, for a
given function $H:\iR\to \iR$, the operator  
\begin{equation} \label{def:Psi}
\Psi_H (f)(x)=\,\sum_{y\in X}k(x,y) H\big(f(y)-f(x)\big),\quad 
x\in X,
\end{equation}
where $f \in \R^X$ is such that the latter is well defined.
Clearly, $\Psi_H(f)=Lf$ if $H$ is the identity. Observe as well that
in case of $\Omega=\iR$ and $H(y)=\frac{1}{2}y^2$ we have $\Psi_H( 
f)=\Gamma(f)$, and thus \eqref{FI1} leads to the well known identity
\begin{equation} \label{squareID}
\frac{1}{2}\,L(f^2) =f L f+\Gamma(f).
\end{equation}
The following identity plays a central role in our approach. It is a discrete version of the identity
\begin{equation} \label{logChain}
L(\log f)=\frac{1}{f}\,L f-\Gamma (\log f)
\end{equation}
from the diffusion setting and can be found in \cite[Example 2.4]{DKZ}.
\begin{lemma} \label{PalmeFI}
For positive functions $f\in \R^X$ such that $f, \log f \in \ell^1\big(k(x, \cdot)\big)$ for any $x \in X$, there holds
\begin{equation} \label{keyPalme}
L(\log f)=\frac{1}{f}\,L f-\Psi_\Upsilon (\log f),
\end{equation}
where
\[
\Upsilon(y):=e^y-1-y, \quad y\in \iR.
\]
\end{lemma}
\begin{proof}
For the reader's convenience we repeat the short calculation, in which we apply Lemma \ref{lem:firstFI} with
$\Omega=(0,\infty)$ and $H(y)=\log y$. In this case we have
\[
\Lambda_H(w,z)=\log w-\log z-\frac{1}{z}\,(w-z)
= -\Upsilon(\log w-\log z),
\] 
and thus \eqref{keyPalme} follows from \eqref{FI1}.
\end{proof}
Given a  function $H:\iR\to \iR$, we next define the operator $B_H$ by  
\begin{equation} \label{def:BH}
B_H (f,g)(x)=\,\sum_{y\in X}k(x,y) H\big(f(y)-f(x)\big)\big(g(y)-g(x)\big),\quad 
x\in X,
\end{equation}
where $f,g \in \R^X$ are such that the latter is well defined. Provided that $H$ is continuous this is the case if $f,g \in \ell^\infty(X)$ by \eqref{ASSsumfinite}, as well as for the quite important  choice of $g=Lf$ with $f \in \ell^\infty(X)$ due to \eqref{ASSsecsumfinite}.
Observe further that $B_H(f,g)=2\Gamma(f,g)$ in case of $H$ being the identity. 

The next result provides a fundamental identity for the state space $X=\iZ^d$ in case of translation invariant kernels. 
\begin{lemma} \label{secondFI}
Let $d\in \iN$, $H\in C^1(\iR)$ and suppose that the kernel $k:\iZ^d\times \iZ^d\to \iR$ is of the form
\[
k(x,y)=k_*(y-x),\quad x,y\in \iZ^d,\,x\neq y,\quad k(x,x)=-\sum_{y\in \iZ^d\setminus\{0\}}k_*(y),\quad x\in \iZ^d,
\]
where $k_*$ is nonnegative on $\iZ^d\setminus\{0\}$ and $k_*\in \ell^1(\iZ^d\setminus\{0\})$. Let $L$ be the Markov generator associated with the kernel $k$,
the state space being $X=\iZ^d$, and let 
$f \in \ell^\infty(\iZ^d)$.
Then for all $x\in \iZ^d$,
\begin{align}
L\big(\Psi_H(f)&\big)(x)  =B_{H'}(f,Lf)(x) \label{FI2}\\
&  +\!\!\!\sum_{h,\sigma \in \iZ^d\setminus\{0\}} k_*(h) k_*(\sigma)
\Lambda_H\big(f(x+h+\sigma)-f(x+h),f(x+\sigma)-f(x)\big).\nonumber
\end{align}
In particular, if $H$ is convex,
\begin{equation} \label{FI2a}
L\big(\Psi_H(f)\big)(x)-B_{H'}(f,Lf)(x)\ge 0,\quad x\in \iZ^d. 
\end{equation}
\end{lemma}
\begin{proof}
Using the special form of the kernel $k$ we have
\begin{align*}
L\big(\Psi_H(f)&\big)(x)=\sum_{h\in \iZ^d\setminus\{0\}}k_*(h)\big(\Psi_H(f)(x+h)-\Psi_H(f)(x)\big)\\
& =\sum_{h,\sigma\in \iZ^d\setminus\{0\}}k_*(h)k_*(\sigma)\big[H\big(f(x+h+\sigma)-f(x+h)\big)
-H\big( f(x+\sigma)-f(x)\big)\big]
\end{align*}
and
\begin{align*}
& B_{H'}(f,Lf)(x)=\sum_{h\in \iZ^d\setminus\{0\}}k_*(h) H'\big(f(x+h)-f(x)\big)\big(Lf(x+h)-Lf(x)\big)\\
&=	 \sum_{h,\sigma\in \iZ^d\setminus\{0\}}k_*(h)k_*(\sigma) H'\big(f(x+h)-f(x)\big)\big(f(x+h+\sigma)-f(x+h)-f(x+\sigma)+f(x)\big).
\end{align*}
Recalling the definition of $\Lambda_H$ (see \eqref{eq:def_Lambda}), identity \eqref{FI2} follows 
immediately from the last two relations. The second assertion is evident by positivity of $k_*$ and the definition of
$\Lambda_H$.
\end{proof}

It is interesting to note that identity \eqref{FI2} is structurally very similar to the fundamental identity
\eqref{FI1}. In contrast to \eqref{FI1}, the $\Lambda_H$-term in \eqref{FI2} involves a difference of second order. 
We also remark that \eqref{FI2} can be viewed as a generalized discrete version of Bochner's identity. In fact, taking
$H(y)=\frac{1}{2}y^2$ we have
\[ \Lambda_H(w,z)=\frac{1}{2}(w-z)^2, \quad \Psi_H(f)=\Gamma(f),\quad  
B_{H'}(f,g)=2\Gamma(f,g),
\]
and thus obtain 
\begin{align*}
& L \Gamma(f)(x) =L \Psi_H(f)(x)\\
& = B_{H'}(f, Lf)(x)+\!\!\!\sum_{h,\sigma \in \iZ^d\setminus\{0\}}\!\! k_*(h) k_*(\sigma)
\Lambda_H\big(f(x+h+\sigma)-f(x+h),f(x+\sigma)-f(x)\big)\\
& = 2\Gamma(f,Lf)(x)+\frac{1}{2}
\!\sum_{h,\sigma \in \iZ^d\setminus\{0\}}\!\! k_*(h) k_*(\sigma)
\big(f(x+h+\sigma)-f(x+h)-f(x+\sigma)+f(x)\big)^2.
\end{align*}
This is a discrete variant of Bochner's identity for the Laplacian $\Delta$ on $\iR^d$, which states that
\[
\Delta |\nabla f|^2=2 \nabla f\cdot \nabla\Delta f+2\sum_{i,j=1}^d (\partial_i\partial_j f)^2.
\] 
Since identity \eqref{FI2} plays an important role with regard to curvature-dimension inequalities for non-local
operators in the discrete setting we will also refer to it as {\em second fundamental identity}.    

The following identity is of crucial importance in the proof of Theorem \ref{ThmSecTDEnt} on the second time derivative
of the entropy.  
\begin{lemma} \label{BPalmeID}
There holds
\begin{equation} \label{BPalmeIDFormel}
\frac{1}{2}\,\int_X e^g B_{\Upsilon'}(g,h)\,d\mu=-\int_X e^g Lh\,d\mu.
\end{equation}
for any functions $g,h \in \R^X$ such that the expressions in \eqref{BPalmeIDFormel} are well defined.
\end{lemma}
\begin{proof}
Inserting $\Upsilon'(z)=e^z-1$ into the definition of the operator $B_{\Upsilon'}$ we obtain
\begin{align*}
\int_X e^g B_{\Upsilon'}(g,h)\,d\mu & = \sum_{x,y\in X}k(x,y) e^{g(x)}
\big(e^{g(y)-g(x)}-1\big) \big(h(y)-h(x)\big)\pi(x)\\
& = \sum_{x,y\in X}k(x,y) e^{g(y)} \big(h(y)-h(x)\big)\pi(x)-\sum_{x\in X} e^{g(x)} Lh(x)\pi(x).
\end{align*}
By detailed balance,
\begin{align*}
\sum_{x,y\in X}k(x,y) e^{g(y)} \big(h(y)-h(x)\big)\pi(x) & = \sum_{x,y\in X}k(y,x) e^{g(y)} \big(h(y)-h(x)\big)\pi(y)\\
& = \sum_{y\in X} e^{g(y)}\big(-Lh(y)\big)\pi(y),
\end{align*}
which proves the assertion.
\end{proof}
\begin{remark} \label{GamExpRem}
Identity \eqref{BPalmeIDFormel} can be viewed as discrete analogue of the formula
\begin{equation} \label{GammaExp}
\int_X e^g \Gamma\big(g,h\big)\,d\mu=-\int_X e^g Lh\,d\mu,
\end{equation}
which is valid in the diffusion setting (with $\mu$ being an invariant and reversible measure on $X$) and follows from the chain rule $\Gamma(H(g),h)=H'(g)\Gamma(g,h)$ for the carr\'e du champ
operator and the identity
\[
\int_X \Gamma(f,g)\,d\mu=-\int_X f Lg\,d\mu,
\]
see \cite[Section 3.4.1]{BGL}.
\end{remark}

Lemma \ref{PalmeFI} suggests to regard $\Psi_\Upsilon(f)$ as a suitable substitute for the carr\'e du champ operator $\Gamma(f)$.
Further, Lemma \ref{BPalmeID} shows that identity \eqref{GammaExp} has a discrete version where $\Gamma$ is replaced by
the operator $2B_{\Upsilon'}$. These two key observations are also reflected in the next definition, where 
we introduce the operator which will serve as appropriate substitute for the gamma deux $\Gamma_2(f)$ in our curvature-dimension condition.
\begin{defi} \label{Psi2Definition}
For any $f \in \ell^\infty(X)$, we define the operator
\begin{equation} \label{Psi2Def}
\Psi_{2,\Upsilon}(f)=\,\frac{1}{2}\,\big( L\Psi_\Upsilon(f)-B_{\Upsilon'}(f,Lf)\big).
\end{equation}
\end{defi}
We want to emphasize here the great structural similarity to the definition of $\Gamma_2(f)$, which is given by
\[
\Gamma_2(f)=\,\frac{1}{2}\,\big( L\Gamma(f)-2\Gamma(f,Lf)\big).
\]
To better understand the connection between both objects, let us introduce the more general operator
\begin{equation} \label{Psi2HDef}
\Psi_{2,H}(f)=\,\frac{1}{2}\,\big( L\Psi_H(f)-B_{H'}(f,Lf)\big),
\end{equation}
where $H\in C^1(\iR)$ is an arbitrary function and $f \in \ell^\infty(X)$. Now observe that the choice $H(y)=\frac{1}{2}y^2$ leads to the 
classical object, the gamma deux operator $\Gamma_2(f)$, whereas $H=\Upsilon$ gives the new object from
Definition \ref{Psi2Definition}. So the quadratic function is replaced with $H(y)=\Upsilon(y)=e^y-1-y$!  
\begin{defi}
The Markov generator $L$ is said to satisfy the curvature-dimension inequality 
$CD_\Upsilon(\kappa,\infty)$ at $x \in X$ for $\kappa\in \iR$, if for every function $f \in l^\infty(X)$  there holds
\begin{equation} \label{CDUpsDef}
\Psi_{2,\Upsilon}(f)(x)\ge \kappa \Psi_\Upsilon(f)(x).
\end{equation}
If $L$ satisfies $CD_\Upsilon(\kappa,\infty)$ at any $x \in X$, then we say that $L$ satisfies $CD_\Upsilon(\kappa,\infty)$.
\end{defi}
\begin{remark}\label{CDPalmefunctionsspace}(i) The choice of the suitable class of functions is a quite subtle question in the classical Bakry-\'Emery theory. In our case, the assumptions \eqref{ASSsumfinite} and \eqref{ASSsecsumfinite} ensure that the operators $\Psi_\Upsilon$ and $\Psi_{2,\Upsilon}$ are both well defined on $\ell^\infty(X)$. This is enough for the applications we have in mind, as we will see below. However, one may  choose alternative function spaces that possibly depend on the kernel.

(ii) If the associated graph to $L$ is locally finite, \eqref{CDUpsDef} holds true for any $f \in \R^X$ if and only it holds on $\ell^\infty(X)$. This is a consequence of the fact that for any $x \in X$ the expressions $\Psi_{2,\Upsilon} (f) (x)$ and $\Psi_\Upsilon(f)(x)$ only depend on finitely many values of $f$. 
\end{remark}
\begin{remark} \label{commentsCD}
 (i) In \cite[Remark 3.12]{Mn1}, M\"unch introduces the CD condition $CD\psi(d,K)$ for the discrete Laplacian on a finite
graph without weights, where $\psi$ is a concave function on $(0,\infty)$, $d>0$ is the dimension parameter and $K\in \iR$
the lower curvature bound. It turns out that for $\psi=\log$ the condition $CD\psi(\infty,\kappa)$ is equivalent to
$CD_\Upsilon(\kappa,\infty)$, see the proof in Section \ref{gemischtes}. This observation is by no means obvious, as the quantity $\Gamma_2^\psi$ appearing in $CD\psi(d,K)$ is defined in a rather implicit way. We point out that in \cite{Mn1},
the condition $CD\psi(d,K)$ with $K>0$ only crops up in Remark 3.12., in form of a definition. It is not used at all in
\cite{Mn1} as this paper focuses on Li-Yau inequalities, for which, assuming nonnegative curvature, the term involving the dimension parameter plays the decisive role. 

(ii) We can also define a more general $CD_\Upsilon(\kappa,d)$ condition with $\kappa\in \iR$ and a dimension 
parameter $d\in [1,\infty]$ via the inequality
\begin{equation} \label{CDUpsDefDim}
\Psi_{2,\Upsilon}(f)\ge \frac{1}{d}\,(Lf)^2+\kappa \Psi_\Upsilon(f).
\end{equation}
This condition is equivalent to $CD\psi(d,\kappa)$ from \cite{Mn1} with $\psi=\log$, cf.\ Section \ref{gemischtes} and the comments in  \cite{Mn1} concerning the case $\psi=\log$. The condition $CD_\Upsilon(\kappa,d)$ with $d<\infty$ is not so much in the focus
of this paper. It will appear in Remark \ref{limiting} on the relation to the classical Bakry-\'Emery condition and in
Remark \ref{RemarkTensorDim} in connection with a generalization of our main tensorization result.

(iii) The condition $CD_\Upsilon(0,d)$ is also closely linked to the condition $CD(F;0)$ with a quadratic CD-function $F$. These
notions were introduced in the very recent work \cite[Definitions 3.1 and 3.8]{DKZ}, which is also concerned with discrete
versions of the Li-Yau gradient estimate. The operator $\Psi_\Upsilon$ is also already used in \cite{DKZ}. Translating the situation
from \cite{DKZ}, where the authors consider a generalized graph Laplacian with vertex weights and symmetric edge weights on 
a locally finite graph, the condition $CD(F;0)$ at $x\in X$ with CD-function $F$ says that
\begin{equation} \label{CDFcond}
L\Psi_{\Upsilon'}(f)(x)\ge F\big(-Lf(x)\big)
\end{equation}
holds for any function $f:\,X\to \iR$ satisfying the conditions
\begin{equation} \label{CDFmaxcond}
Lf(x)<0\quad \mbox{and}\quad Lf(x)\le Lf(y)\quad \mbox{if}\;\;k(x,y)>0.
\end{equation}
 Using $\Upsilon'(z)=\Upsilon(z)+z$, $e^z=\Upsilon'(z)+1$, $\Psi_{Id}(f)=Lf$ and denoting the function with constant value $1$ by ${\bf 1}$ we have
\begin{align*}
L\Psi_{\Upsilon'}(f)(x)& = L\Psi_{\Upsilon}(f)(x)+L^2 f(x)\\
& = 2\Psi_{2,\Upsilon}(f)(x)+B_{\Upsilon'}(f, Lf)(x)+B_{{\bf 1}}(f,Lf)(x)\\
& =  2\Psi_{2,\Upsilon}(f)(x)+B_{\exp}(f, Lf)(x).
\end{align*}
This shows that under \eqref{CDFmaxcond} (which suffices for the Li-Yau inequality when using the maximum principle), it is advantageous to work with $L\Psi_{\Upsilon'}(f)(x)$ instead of 
$2\Psi_{2,\Upsilon}(f)(x)$, since one has the additional nonnegative term $B_{\exp}(f,Lf)(x)$. Employing this term and allowing 
also for nonquadratic CD-functions may lead to
significantly better (even sharp) Li-Yau estimates (\cite{DKZ}).

(iv) The condition $CD_\Upsilon(\kappa,\infty)$ is further closely connected to the generalized $\Gamma$-calculus described in Monmarch\'e \cite[Section 2]{Mon}. Here the author extends the classical $\Gamma$-calculus to study functional inequalities,
hypoellipticity and hypocoercivity for degenerate Markov processes like, e.g., generalized Ornstein-Uhlenbeck processes,
kinetic Fokker-Planck diffusion in $\iR^{2d}$ or chains of interacting particles.
We also refer to 
\cite[Section 4]{Mon2}, where Monmarch\'e already formulates the basic ingredients of his calculus
for Markov semigroups on the Euclidean space $\iR^d$. 

Assuming that the domain $D(L)$ of the generator $L$ contains a
core $\cA$ (which is also an algebra) and setting $\cA_+:=\{f\in \cA: \,f\ge 0\}$ he considers maps $\Phi:\,\cA_+\to \cA$ differentiable in the sense
that $\frac{1}{\tau}(\Phi(f+\tau g)-\Phi(f))(x)$ converges as $\tau\to 0$ to $(d\Phi(f).g)(x)$ (this is the 
notation used in \cite{Mon}) for all $x\in X$, $f,g\in \cA_+$,
where $d\Phi$ is the corresponding derivation operator. For $f\in \cA_+$ Monmarch\'e then defines the operator
\[
\Gamma_{L,\Phi}(f)=\frac{1}{2}\,\big(L\Phi(f)-d\Phi(f).Lf\big).
\]
For $\Phi(f)=f^2$ one obtains $\Gamma_{L,\Phi}(f)=\Gamma(f)$. With the choice $\Phi(f)=\Gamma(f)$ one recovers the $\Gamma_2$-operator. In this general framework, Monmarch\'e introduces the curvature condition $\Gamma_{L,\Phi}(f)\ge \kappa\Phi(f)$ with some $\kappa\in \iR$.

If we now take $\Phi(f)=\Psi_\Upsilon(f)$ in the discrete setting, we obtain that
\[
\frac{1}{\tau}(\Phi(f+\tau g)-\Phi(f))(x)\to B_{\Upsilon'}(f,g)(x) \quad \mbox{as}\;\tau\to 0,
\] 
under suitable conditions on $f,g$ and the kernel $k$, and consequently, $\Gamma_{L,\Phi}(f)=\Psi_{2,\Upsilon}(f)$ so that
Monmarch\'e's curvature condition is formally equivalent to $CD_\Upsilon(\kappa,\infty)$. 
Note that here, there is no reason to restrict $\Phi$ to nonnegative functions
as in \cite{Mon}.

We point out that objects like $\Psi_\Upsilon$, $\Psi_{2,\Upsilon}$ or the discrete Boltzmann entropy (see 
Section \ref{sec:entrodecayandmLSI} below) do not seem to appear at all in \cite{Mon}, \cite{Mon2}, as the focus of
the author does not lie on the pure discrete setting but mainly on degenerate problems in a continuous setting.     

(v) As already mentioned in the introduction, the operators $\Psi_\Upsilon$ and $\Psi_{2,\Upsilon}$ also appear implicitly  in the context of entropic Ricci curvature of Erbar and Maas in \cite{EM12}. More precisely, in Section \ref{gemischtes} we will demonstrate that
\begin{equation} \label{AFormel}
\cA(\rho,\log \rho)=\int_X \rho \Psi_{\Upsilon}(\log \rho)\,d\mu
\end{equation}
and
\begin{equation} \label{BFormel}
\cB(\rho,\log \rho)=\int_X \rho \Psi_{2,\Upsilon}(\log \rho)\,d\mu,
\end{equation}
hold respectively, where the functionals $\mathcal{A}$ and $\mathcal{B}$ play a key role in \cite{EM12}, see also Section \ref{gemischtes} for a definition.
\end{remark}
On suitable functions the Bakry-\'Emery condition is known to be equivalent to respective gradient bounds in the diffusive setting of \cite{BGL} as well as in the discrete setting, for instance, of \cite{LiLi} and \cite{HuLi}. The following remark indicates that we can characterize the $CD_\Upsilon$ condition in a similar spirit.
\begin{remark}\label{gradbounds}
Due to the work of \cite{Mon}, which we have discussed in Remark \ref{commentsCD}(iv),  \eqref{CDUpsDef} is equivalent to the gradient bound 
\begin{equation}\label{gradboundformula}
\Psi_\Upsilon(P_t f) \leq e^{-2 \kappa t} P_t(\Psi_\Upsilon(f)), \quad t>0,
\end{equation}
in a class of sufficiently good functions, cf. Lemma 1 in \cite{Mon}. We revisit the instructive argument on a formal level for the reader's convenience.

Let $t>0$ and $\varphi(s):= e^{-2\kappa s} P_s \big( \Psi_\Upsilon (P_{t-s}f)\big)$, $s \in [0,t]$. Then we have
\begin{align*}
\varphi'(s)= e^{-2 \kappa s} \Big( -2 \kappa P_s (\Psi_\Upsilon(P_{t-s}f)) + L P_s (\Psi_\Upsilon(P_{t-s}f)) + P_s \big(\frac{d}{d _s}\Psi_\Upsilon(P_{t-s}f)\big)\Big).
\end{align*}
Now, for any $\tau>0$
\begin{equation}\label{AbleitungPsiPalmealogheatflow}
\frac{d}{d \tau}\Psi_\Upsilon(P_{\tau}f)) =  B_{\Upsilon'}(P_{\tau}f,L P_{\tau}f)
\end{equation}
holds true and thus
\begin{align*}
\varphi'(s) = 2 e^{-2\kappa s} P_s \big( \Psi_{2,\Upsilon}(P_{t-s}f) - \kappa \Psi_\Upsilon(P_{t-s}f)\big).
\end{align*}
Thus, we infer from $CD_\Upsilon(\kappa,\infty)$ and the positivity of the semigroup that $\varphi$ is increasing and consequently we observe \eqref{gradboundformula} by
\begin{align*}
e^{-2\kappa t} P_t (\Psi_\Upsilon(f)) - \Psi_\Upsilon (P_t f) = \varphi(t) - \varphi(0)  \geq 0.
\end{align*}
Conversely, \eqref{gradboundformula} implies for $t>0$
\begin{align*}
\frac{G(t)-G(0)}{t}\geq 0,
\end{align*}
where $G(t)= e^{-2\kappa t} P_t (\Psi_\Upsilon(f)) -\Psi_\Upsilon(P_t f)$. Using \eqref{AbleitungPsiPalmealogheatflow}, we have further
\begin{equation*}
\frac{d}{d t} G(t) = e^{-2\kappa t} \big(-2\kappa P_t (\Psi_\Upsilon (f)) + L P_t (\Psi_\Upsilon (f))\big) - B_{\Upsilon'}(P_t f, L P_t f ).
\end{equation*}
In particular, we deduce from $G'(0) \geq 0$ the estimate 
\begin{equation*}
0 \leq -2 \kappa \Psi_\Upsilon (f)  +  L \Psi_\Upsilon (f) - B_{\Upsilon'}(f,Lf),
\end{equation*}
which yields $CD_\Upsilon(\kappa,\infty)$.

The above calculation is rigorous for any $f \in \R^X$ if $X$ is finite. In the general case, regularity assumptions on $f$ are needed, particularly to guarantee that the generator and the semigroup do commute in the above lines.
\end{remark}
\begin{proposition}\label{CDPalmeandCD}
If the Markov generator $L$ satisfies $CD_\Upsilon(\kappa,\infty)$, then $L$ also satisfies $CD(\kappa,\infty)$.
\end{proposition}
\begin{remark}
We remark that the statement of Proposition \ref{CDPalmeandCD} is implicitly contained in \cite{Mn1} in the special case of finite graphs without weights, cf.\ Remark \ref{commentsCD}(i) above. Since our setting is more general, we present a self-contained proof.
\end{remark}
\begin{proof}
Let $f \in \ell^\infty(X)$ and $x \in X$. In order to calculate $\lim\limits_{\lambda \to 0}\frac{\Psi_\Upsilon(\lambda f)(x)}{\lambda^2}$ and $\lim\limits_{\lambda \to 0}\frac{\Psi_{2,\Upsilon}(\lambda f) (x)}{\lambda^2}$ we apply l' Hospital's rule, respectively. We obtain
\begin{align*}
\frac{d}{d\lambda}\Psi_\Upsilon (\lambda f)(x) = \sum_{y \in X \setminus \{x\}} k(x,y) \Upsilon'\big(\lambda(f(y)-f(x))\big)(f(y)-f(x))
\end{align*}
and
\begin{equation}\label{lambdaderivativPsiPalme}
\frac{d^2}{d\lambda^2}\Psi_\Upsilon (\lambda f)(x) = \sum_{y \in X \setminus \{x\}} k(x,y) e^{\lambda (f(y)-f(x))}\big( f(y) - f(x) \big)^2.
\end{equation}
Hence, we have
\begin{align*}
\lim\limits_{\lambda \to 0}\frac{\Psi_\Upsilon(\lambda f)(x)}{\lambda^2} = \frac{1}{2} \sum_{y \in X \setminus \{x\}} k(x,y) \big( f(y) - f(x) \big)^2 = \Gamma(f)(x).
\end{align*}
Above, we were able to interchange the order of differentation and summation by \eqref{ASSsumfinite}. In the subsequent lines we use \eqref{ASSsecsumfinite} for that purpose. We obtain
\begin{align*}
\frac{d}{d\lambda} L \Psi_\Upsilon(\lambda f)(x) = \sum_{y \in X \setminus \{x \}} k(x,y) \frac{d}{d \lambda} \big(\Psi_\Upsilon(\lambda f) (y) - \Psi_\Upsilon(\lambda f)(x) \big)
\end{align*}
and by \eqref{lambdaderivativPsiPalme},
\begin{align*}
\frac{d^2}{d\lambda^2} L \Psi_\Upsilon(\lambda f)(x)
&= \sum_{y \in X \setminus \{x \}} k(x,y) \Big(\sum_{z \in X \setminus \{y \}} k(y,z) e^{\lambda(f(z)-f(y))} \big( f(z)-f(y)\big)^2 \\
&\qquad\qquad\qquad\qquad - \sum_{z \in X \setminus \{x \}}k(x,z) e^{\lambda(f(z)-f(x))} \big( f(z)-f(x)\big)^2 \Big).
\end{align*}
Consequently,
\begin{equation}\label{lambdaLPsiPalmelimit}
\lim\limits_{\lambda \to 0} \frac{L \Psi_\Upsilon(\lambda f)(x)}{\lambda^2} = \sum_{y \in X \setminus \{x\}} k(x,y) (\Gamma(f)(y)-\Gamma(f)(x)) = L \Gamma(f)(x).
\end{equation}
Further, by the product rule
\begin{align*}
\frac{d}{d\lambda}B_{\Upsilon'}(\lambda f, L(\lambda f))(x) &= \sum_{y \in X\setminus\{x\}}k(x,y) \Big( \Upsilon'\big(\lambda(f(y)-f(x))\big) \big(L f(y) - L f(x)\big) \\
&\qquad\qquad\qquad\qquad+ \lambda e^{\lambda(f(y)-f(x))}\big( L f(y) - L f(x) \big) \big( f(y) - f(x) \big) \Big)
\end{align*}
and thus
\begin{align*}
\frac{d^2}{d\lambda^2}B_{\Upsilon'}(\lambda f, L(\lambda f))(x) &= \sum_{y \in X \setminus \{x\}}k(x,y) \Big( 2 e^{\lambda(f(y)-f(x))} \big( L f(y)- L f(x) \big) \big( f(y) - f(x) \big) \\
&\qquad\qquad\qquad\qquad+ \lambda e^{\lambda(f(y)-f(x))} \big( L f(y)- L f(x) \big) \big( f(y) - f(x) \big)^2 \Big).
\end{align*}
We deduce
\begin{equation}\label{lambdaBPalmelimit}
\lim\limits_{\lambda \to 0} \frac{B_{\Upsilon'}(\lambda f, L (\lambda f))(x)}{\lambda^2} = 2 \Gamma(f,Lf)(x).
\end{equation}
Combining \eqref{lambdaLPsiPalmelimit} and \eqref{lambdaBPalmelimit}, yields
\begin{equation*}
\lim\limits_{\lambda \to 0} \frac{\Psi_{2,\Upsilon}(\lambda f) (x)}{\lambda^2} = \Gamma_2(f)(x).
\end{equation*}
Now, $CD_\Upsilon(\kappa,\infty)$ implies
\begin{align*}
\Gamma_2(f)(x) &= \lim\limits_{\lambda \to 0} \frac{\Psi_{2,\Upsilon}(\lambda f) (x)}{\lambda^2}\\
&\geq \kappa \lim\limits_{\lambda \to 0} \frac{\Psi_{\Upsilon}(\lambda f) (x)}{\lambda^2} = \kappa \Gamma(f)(x).
\end{align*}
\end{proof}
\begin{remark} \label{limiting}
(i) The proof of Proposition \ref{CDPalmeandCD} shows that the classical Bakry-\'Emery condition $CD(\kappa,\infty)$
is obtained from $CD_\Upsilon(\kappa,\infty)$ by taking the scaling limit at zero, the crucial property being that $\Upsilon(y)$
behaves like $\frac{1}{2}y^2$ as $y\to 0$. 

(ii) Since clearly $(L(\lambda f))^2=\lambda^2 (Lf)^2$, the statement of Proposition \ref{CDPalmeandCD} (and its proof)
extends to the finite dimension case, that is, $CD_\Upsilon(\kappa,d)$ (as defined in Remark \ref{commentsCD}(ii))
implies $CD(\kappa,d)$ for any $\kappa\in \iR$ and $d\in [1,\infty]$.
\end{remark}
We next want to derive estimates which allow us to compare $\Psi_\Upsilon(f)$ and $\Gamma(f)$ as well as
$\Psi_{2,\Upsilon}(f)$ and $\Gamma_2(f)$, respectively, in a quantitative way provided that $f$ has small values.
We begin with a useful representation of the operator $\Psi_{2,H}(f)$, which was introduced in \eqref{Psi2HDef}.
By definition, we have
\begin{align}
2\Psi_{2,H}&(f)(x)=L\Psi_H(f)(x)-B_{H'}(f,Lf)(x)\nonumber\\
& = \sum_{y\in X\setminus\{x\}} k(x,y) \Big( \Psi_H(f)(y)-\Psi_H(f)(x)-H'\big(f(y)-f(x)\big)\big(Lf(y)-Lf(x)\big)\Big)\nonumber\\
& = \sum_{y\in X\setminus\{x\}} k(x,y) \sum_{z\in X} k(y,z)\Big(H\big(f(z)-f(y)\big)-
H'\big(f(y)-f(x)\big)\big(f(z)-f(y)\big)\Big)\nonumber\\
& \quad\; +\sum_{y\in X\setminus\{x\}} k(x,y) H'\big(f(y)-f(x)\big)\sum_{z\in X} k(x,z)\big(f(z)-f(x)\big)\nonumber\\
& \quad\; - \sum_{y\in X\setminus\{x\}} k(x,y) \sum_{z\in X} k(x,z)H\big(f(z)-f(x)\big).\label{Psi2Formel}
\end{align}
For the choice of $H(r)=\frac{1}{2}r^2$, it follows easily from \eqref{Psi2Formel} that $CD(-M,\infty)$ holds provided that $\sup_{x \in X} M_1(x) \leq M < \infty$.  
Interestingly, this simple property is not satisfied in general by the $CD_\Upsilon$ condition, as will be demonstrated in Example \ref{ex:nolowerbound}.

Given $\varepsilon\in (0,\frac{1}{2})$ there exists a $\delta_\varepsilon>0$ such that
\begin{equation} \label{PalmeEstimate}
\frac{1-\varepsilon}{2}\,r^2\le \Upsilon(r)\le \frac{1+\varepsilon}{2}\,r^2,\quad r\in [-\delta_\varepsilon,\delta_\varepsilon].
\end{equation}
Suppose that $f\in \iR^X$ satisfies $2|f|_\infty\le \delta\le \delta_\varepsilon$. Applying \eqref{Psi2Formel} with $H=\Upsilon$
and using the shorter notation $f_x$ in place of $f(x)$ as well as the relation $\Upsilon'(r)=\Upsilon(r)+r$ we obtain
\begin{align*}
2\Psi_{2,\Upsilon}(f)(x) &=\!\!\!\sum_{y\in X\setminus\{x\}} k(x,y)\!\!\! \sum_{z\in X\setminus\{y\}} k(y,z)
\Big(\Upsilon(f_z-f_y)-\big[
(f_y-f_x)+\Upsilon(f_y-f_x)\big](f_z-f_y)\Big)\\
& \quad\; +\sum_{y\in X\setminus\{x\}} k(x,y) \big[(f_y-f_x)+\Upsilon(f_y-f_x)\big]
\sum_{z\in X} k(x,z)(f_z-f_x)\\
& \quad\; - \sum_{y\in X\setminus\{x\}} k(x,y) \sum_{z\in X} k(x,z)\Upsilon(f_z-f_x).
\end{align*}
Employing \eqref{PalmeEstimate}, Young's inequality, and formula \eqref{Psi2Formel} with $H(r)=\frac{1}{2}r^2$ we infer that
\begin{align*}
2\Psi_{2,\Upsilon}(f)(x) &\ge \!\!\!\sum_{y\in X\setminus\{x\}} k(x,y)\!\!\! \sum_{z\in X\setminus\{y\}} k(y,z)
\Big(\frac{1-\varepsilon}{2}(f_z-f_y)^2-(1-2\varepsilon)(f_y-f_x)(f_z-f_y)\\
& \quad\quad -2\varepsilon(f_y-f_x)(f_z-f_y)-\Upsilon(f_y-f_x) (f_z-f_y)
\Big)\\
& \quad\; +\sum_{y\in X\setminus\{x\}} k(x,y) \sum_{z\in X} k(x,z) \Big((1-2\varepsilon)(f_y-f_x)
(f_z-f_x)\\
& \quad\quad +2\varepsilon(f_y-f_x)(f_z-f_x)+\Upsilon(f_y-f_x)(f_z-f_x)\Big)\\
& \quad\; - \sum_{y\in X\setminus\{x\}} k(x,y) \sum_{z\in X} k(x,z)\Big(\frac{1-2\varepsilon}{2}(f_z-f_x)^2
+\frac{3\varepsilon}{2} (f_z-f_x)^2\Big)\\
&\ge 2 (1-2\varepsilon)\Gamma_2(f)\\
& \quad -\!\!\!\sum_{y\in X\setminus\{x\}}\! \!\!k(x,y)\!\!\! \sum_{z\in X\setminus\{y\}}\!\!\! k(y,z)
\Big(2\varepsilon(f_y-f_x)^2
+\frac{\delta(1+\varepsilon)}{2}(f_y-f_x)^2\Big)\\
& \quad -\!\!\!\sum_{y\in X\setminus\{x\}}\!\!\! k(x,y) \!\!\!\sum_{z\in X\setminus\{x\}}\!\!\! k(x,z)
\Big(\varepsilon(f_y-f_x)^2+\frac{5\varepsilon}{2}(f_z-f_x)^2
+\frac{\delta(1+\varepsilon)}{2}(f_y-f_x)^2\Big).
\end{align*}
Assuming that $\sum_{y\in X\setminus\{x\}}k(x,y)\le M$ for all $x\in X$, the last estimate gives
\begin{align*}
2\Psi_{2,\Upsilon}(f)(x) &\ge  2 (1-2\varepsilon)\Gamma_2(f)(x)
 -M\!\!\!\sum_{y\in X\setminus\{x\}}\!\!\! k(x,y)\big(3\varepsilon+\delta(1+\varepsilon)\big)(f_y-f_x)^2\\ 
& \quad -M\!\!\!\sum_{z\in X\setminus\{x\}}\!\!\! k(x,z) \frac{5\varepsilon}{2}\,(f_z-f_x)^2\\
& = 2 (1-2\varepsilon)\Gamma_2(f)(x)-M\big(11\varepsilon+2\delta(1+\varepsilon)\big)\Gamma(f)(x).
\end{align*}
In analogous fashion one can show that
\begin{align*}
2\Psi_{2,\Upsilon}(f)(x) &\le 
 2 (1+2\varepsilon)\Gamma_2(f)(x)+M\big(11\varepsilon+2\delta(1+\varepsilon)\big)\Gamma(f)(x).
\end{align*}

Furthermore,
\begin{align*}
(1-\varepsilon)\Gamma(f)(x) &=\sum_{y\in X\setminus\{x\}} k(x,y) \frac{1-\varepsilon}{2}\,(f_y-f_x)^2
\le \sum_{y\in X\setminus\{x\}} k(x,y) \Upsilon(f_y-f_x)\\
& \le \sum_{y\in X\setminus\{x\}} k(x,y) \frac{1+\varepsilon}{2}\,(f_y-f_x)^2=(1+\varepsilon)\Gamma(f)(x). 
\end{align*}

This proves the following proposition.
\begin{proposition} \label{GammaPsiCompare}
Let $\varepsilon\in(0,\frac{1}{2})$ and $\delta_\varepsilon>0$ be such that the two-sided quadratic estimate \eqref{PalmeEstimate} for the $\Upsilon$-function is satisfied. Suppose further that there is a constant $M>0$ such that
\begin{equation} \label{kernelM}
\sup_{x \in X} M_1(x) \le M .
\end{equation}
Then the following estimates hold for any $x\in X$ and $f\in \iR^X$ 
with $2|f|_\infty\le \delta\le \delta_\varepsilon$.
\begin{align} 
(1-\varepsilon)\Gamma(f)(x) & \le\Psi_\Upsilon(f)(x) \le \; (1+\varepsilon)\Gamma(f)(x), \label{PsiGamma1}
\end{align}
\begin{align}
 (1-2\varepsilon)\Gamma_2(f)(x)- & 6M(\varepsilon+\delta)\Gamma(f)(x)
 \le \Psi_{2,\Upsilon}(f)(x) \nonumber\\
 & \le  (1+2\varepsilon)\Gamma_2(f)(x)+6M(\varepsilon+\delta)\Gamma(f)(x). \label{PsiGamma2}
\end{align}
\end{proposition}
From Proposition \ref{GammaPsiCompare} we immediately deduce the following result.
\begin{corollary}
Let $\varepsilon\in(0,\frac{1}{2})$ and $\delta_\varepsilon>0$ be such that the two-sided quadratic estimate \eqref{PalmeEstimate} for the $\Upsilon$-function is satisfied. Suppose further that \eqref{kernelM} holds and that
the Markov generator $L$ satisfies the condition $CD(\kappa,\infty)$ with some $\kappa>0$.
Then for any $x\in X$ and any $f\in \iR^X$ 
with $2|f|_\infty\le \delta\le \delta_\varepsilon$ there holds
\begin{equation} \label{PsiGamma2a}
 \Big( 1-2\varepsilon-\frac{6M}{\kappa}(\varepsilon+\delta)\Big)\Gamma_2(f)(x) \le 
\Psi_{2,\Upsilon}(f)(x) \le \Big( 1+2\varepsilon+\frac{6M}{\kappa}(\varepsilon+\delta)\Big)\Gamma_2(f)(x).
\end{equation}
\end{corollary}
We come back to the important case where $X=\iZ^d$ and $k$ is a translation invariant kernel. 
\begin{proposition} \label{Psi2FormelGrid}
Let $d\in \iN$ and suppose that the kernel $k:\iZ^d\times \iZ^d\to \iR$ is of the form described in Lemma 
\ref{secondFI}, that is $k(x,y)=k_*(y-x)$, $x,y\in \iZ^d$, $x\neq y$, where $k_*\in \ell^1(\iZ^d\setminus\{0\})$
is nonnegative. Let $f\in \ell^\infty(\Z^d)$ and $x\in \iZ^d$. Then
\begin{equation} \label{Gamma2grid}
\Gamma_2(f)(x)=\frac{1}{4}
\!\sum_{h,\sigma \in \iZ^d\setminus\{0\}}\!\!\!\! k_*(h) k_*(\sigma)
\big(f(x+h+\sigma)-f(x+h)-f(x+\sigma)+f(x)\big)^2
\end{equation}
and
\begin{equation} \label{Psi2Palmegrid}
\Psi_{2,\Upsilon}(f)(x)=\frac{1}{2}
\!\sum_{h,\sigma \in \iZ^d\setminus\{0\}}\!\!\!\! k_*(h) k_*(\sigma)e^{f(x+\sigma)-f(x)}
\Upsilon\big(f(x+h+\sigma)-f(x+h)-f(x+\sigma)+f(x)\big).
\end{equation}
In particular, the CD conditions $CD(0,\infty)$ and $CD_\Upsilon(0,\infty)$ are satisfied.
\end{proposition}
\begin{proof}
The assertion on $\Gamma_2$ and the classical Bakry-\'Emery condition are well known in the literature. See also the 
discussion above after Lemma \ref{secondFI}. 

Concerning $\Psi_{2,\Upsilon}$ we apply Lemma \ref{secondFI} with $H=\Upsilon$ and use that
\begin{align*}
\Lambda_\Upsilon(w,z)\, & =\Upsilon(w)-\Upsilon(z)-\Upsilon'(z)(w-z)\\
& = \big(e^{w}-1-w\big)- \big(e^{z}-1-z\big)-\big(e^{z}-1\big)(w-z)\\
& = e^{w}-e^{z}-e^{z}(w-z)\\
& = \Lambda_{\exp}(w,z)= e^{z}\Upsilon(w-z).
\end{align*}
\end{proof}
\begin{remark}
The identity \eqref{Gamma2grid} has been used in the one-dimensional setting of $X=\Z$ in the quite recent work \cite{SWZ1}, where positive and negative results concerning the $CD(0,n)$ condition with $n<\infty$ were established. According to Remark \ref{commentsCD}(i) and Remark \ref{limiting}(ii), this is important in the context of Li-Yau inequalities since the $CD(0,n)$ condition is necessary for $CD_\Upsilon(0,n)$, which has been introduced in Remark \ref{commentsCD}(ii). However, as to curvature bounds, we observe that, formally, $CD(\kappa,\infty)$ (and thus $CD_\Upsilon(\kappa,\infty)$) can not be valid by plugging the identity function in \eqref{Gamma2grid}. See \cite[Theorem 6.1]{SWZ1} for a rigorous proof that applies to a large class of kernels using approximative arguments.
\end{remark}

\section{Entropy decay and modified log-Sobolev inequality}\label{sec:entrodecayandmLSI}
Recall that we assume throughout this section that the invariant and reversible measure $\mu$ is a probability measure on $X$ and $(P_t)_{t \geq 0}$ is a Markov semigroup. Also recall that $\pi$ denotes the (probability) density of $\mu$ with respect to the counting measure on $X$.  Let
\[
\cP(X):=\{\rho:\,X \to [0,\infty)\;\mbox{such that}\;\sum_{x\in X}\rho(x)\pi(x)=1\}
\]
denote the set of probability densities with respect to $\mu$ on $X$. Further, let $\mathcal{P}_*(X)$ denote the set of strictly positive probability densities with respect to $\mu$ and  $\mathcal{P}_*^+(X):= \mathcal{P}_*(X) \cap \ell^{\infty,+}(X)$, where
\begin{align*}
\ell^{\infty,+}(X)=\{f \in \ell^{\infty}(X): \exists c>0 \; \text{s.t.}\; f(x) \geq c > 0, \forall x \in X \}.
\end{align*}

We consider the entropy
\[
\cH (\rho)=\int_X \rho \log \rho\,d\mu= \sum_{x\in X} \rho(x) \log(\rho(x)) \pi(x),\quad \rho\in \cP(X),
\]
where we set $0 \log 0 = 0$ in this definition (and $\mathcal{H}(f)=\infty$ is allowed), and the discrete Fisher information
\begin{equation} \label{DefFisher}
\cI(\rho)=\int_X \rho \Psi_\Upsilon(\log \rho)\,d\mu=\sum_{x\in X} \rho(x) \Psi_\Upsilon(\log \rho)(x)\pi(x), \quad \rho \in \mathcal{P}_*^+(X).
\end{equation}
Using the identity
\[
a\Upsilon(\log b-\log a)=b-a-a(\log b-\log a),\quad a,b>0,
\]
and detailed balance we can rewrite the Fisher information as
\begin{align}
\cI(\rho) = & \sum_{x\in X} \rho(x) \sum_{y\in X} k(x,y) \Upsilon\big(\log \rho(y) -\log \rho(x)\big) \pi(x)\nonumber\\
= & \,\frac{1}{2}\,\sum_{x\in X} \rho(x) \sum_{y\in X} k(x,y) \Upsilon\big(\log \rho(y) -\log \rho(x)\big) \pi(x)\nonumber\\
& + \frac{1}{2}\,\sum_{y\in X} \rho(y) \sum_{x\in X} k(y,x) \Upsilon\big(\log \rho(x) -\log \rho(y)\big) \pi(y)\nonumber\\
= & \,\frac{1}{2}\,\sum_{x,y\in X}  k(x,y) \Big(\rho(y)-\rho(x)-\rho(x)\big(\log \rho(y) -\log \rho(x)\big)\Big) \pi(x)\nonumber\\
& + \frac{1}{2}\,\sum_{x,y\in X} k(y,x) \Big(\rho(x)-\rho(y)-\rho(y)\big(\log \rho(x) -\log \rho(y)\big)\Big) \pi(y)\nonumber\\
 = &  \,\frac{1}{2}\,\sum_{x,y\in X} k(x,y) \big(\rho(y)-\rho(x)\big)\big(\log \rho(y) -\log \rho(x)\big)\pi(x). \label{FisherUsual}
\end{align}
The representation \eqref{FisherUsual} for the Fisher information has been frequently used in the existing literature, see e.g.
\cite{EM12}, \cite[Section 5.1]{Jn}. Note also that $\mathcal{I}(\rho)=\mathcal{E}(\rho,\log \rho)$, where $\mathcal{E}$ is the Dirichlet form given by \eqref{Dirichletform}. In the sequel we use this relation to extend $\mathcal{I}$ to functions $\rho \in \mathcal{P}_*(X)$, where $\mathcal{I}(\rho)=\infty$ is allowed.
\begin{defi}\label{def:mLSI}
We say that $L$ satisfies the  modified logarithmic Sobolev inequality $\mathrm{MLSI}(\alpha)$ with $\alpha>0$, if
\begin{equation}\label{MLSI}
\mathcal{H}(f) \leq \frac{1}{2\alpha} \mathcal{I}(f)
\end{equation}
holds for any $f \in P_*(X)$ with $\mathcal{H}(f)<\infty$.
\end{defi}
It is a classical fact that exponential decay of the entropy
\begin{equation}\label{entropydecayequivtoMLSI}
\mathcal{H}(P_t f) \leq e^{-2\alpha t} \mathcal{H}(f),\quad t>0,
\end{equation}
is equivalent to \eqref{MLSI} on a suitable class of functions, see e.g. \cite{CaDP}. With regard to the latter we will work with the function space $\mathcal{P}_*^+(X)$.

\begin{lemma}\label{MLSIapproximation}
Assume that \eqref{MLSI} holds for $\alpha>0$ and any $f \in \mathcal{P}_*^+(X)$. Then $L$ satisfies $\mathrm{MLSI}(\alpha)$.
\end{lemma}
\begin{proof} 
Clearly, there is nothing to show in the case of $\mathcal{I}(f)=\infty$. If $\mathcal{I}(f)<\infty$ the claim can be established by a standard truncation argument. More precisely, one considers the sequence $(f_n)_{n \in \N} \subset P_*^+(X)$ given by $f_n(x)=\frac{f(x)}{C_n}$ if $\frac{1}{n}\leq f(x)\leq n$, $f_n(x)=\frac{1}{C_n n}$ if $f(x) > \frac{1}{n}$ and $f_n(x)=\frac{n}{C_n}$ if $f(x)>n$, where $C_n$ denotes the corresponding normalizing constant. Now, it is straightforward to check that $\mathcal{H}(f_n) \to \mathcal{H}(f)$ and $\mathcal{I}(f_n) \to \mathcal{I}(f)$ as $n\to \infty$.
\end{proof}
In the classical diffusive setting the modified logarithmic Sobolev inequality coincides with the logarithmic Sobolev inequality by means of classical chain rules. This is no longer true in the discrete setting, where the logarithmic Sobolev inequality is stronger than \eqref{MLSI} (cf. \cite[Lemma 2.7]{DSC96}). We refer to \cite{BoTe} and \cite{DSC96} for an extensive overview on these functional inequalities in the discrete setting of Markov chains.

As a motivation, we briefly outline the Bakry-\'Emery approach to establish the functional inequality \eqref{MLSI}. The key step is to achieve the differential inequality
\begin{equation}\label{diffineqasMotivation}
\frac{d^2}{dt^2} \mathcal{H}(P_t f) \geq - 2 \alpha \frac{d}{dt}\mathcal{H}(P_t f),
\end{equation}
which is established in the classical diffusive setting by means of the Bakry-\'Emery  condition $CD(\alpha,\infty)$. Starting from \eqref{diffineqasMotivation} the strategy is quite general. By Gronwall's lemma one deduces $-\frac{d}{dt}\mathcal{H}(P_t f)\leq - e^{-2\alpha t} \left.\frac{d}{dt}\mathcal{H}(P_t f)\right|_{t=0}$. Integrating this inequality, yields
\begin{equation*}
\mathcal{H}(f) - \mathcal{H}(P_t f) \leq - \left.\frac{d}{dt} \mathcal{H}(P_t f)\right|_{t = 0}\, \int_0^t e^{-2\alpha s} ds.
\end{equation*}
As $\frac{d}{dt}\mathcal{H}(P_t f) = - \mathcal{I}(P_t f)$ holds (see Proposition \ref{deBruijnidentity} below), \eqref{MLSI} follows by sending $t \to \infty$ if $\mathcal{H}(P_t f) \to 0$ as $t \to \infty$. If $f \in \mathcal{P}_*^+(X)$, the latter property follows from ergodicity of the Markov chain, the dominated convergence theorem and $(P_t)_{t \geq 0}$ being a Markov semigroup.

Note that this strategy not only applies to the entropy functional $\mathcal{H}$. Therefore we will use it not only in the present context to deduce Corollary \ref{mLSIoutofCDPalme}, but also in the setting of power-type entropies (Corollary \ref{BecwithCDp}) and those of hybrid processes (Corollary \ref{mLSIforhybridprocesses}).

In order to apply the pointwise condition \eqref{CDUpsDef}, we need to be able to interchange the order of differentation and integration in the strategy explained above. Concerning the  finite state-space case this is of course clear. For the general case, we will impose respective integrability conditions on the functions $M_1$ and $M_2$ (see \eqref{ASSsumfinite} and \eqref{ASSsecsumfinite}) such that the dominated convergence theorem applies, see Proposition \ref{deBruijnidentity}, Theorem \ref{ThmSecTDEnt}, Corollary \ref{mLSIoutofCDPalme}, Theorem \ref{theo:BecknerEntropiederivatives} and Corollary \ref{BecwithCDp}.

It is well known that the time-derivative of the entropy equals the
negative Fisher information along the heat flow associated with the operator $L$, see e.g.\ \cite{EM12}, \cite[Section 5.1]{Jn}.
\begin{proposition}\label{deBruijnidentity}
Let $M_1 \in \ell^1(\mu)$. Then for any $f\in \cP_*^+(X)$,
\begin{equation} \label{TimeDeriEntropy}
\frac{d}{dt}\,\cH(P_t f)=-\cI(P_t f),\quad t\ge 0
\end{equation}
holds true.
\end{proposition}
Even though this result is well known it is instructive to revisit the argument in terms of our notation. Since 
$\frac{d}{dt}P_t f=L P_t f$ we have
\begin{align*}
\frac{d}{dt}\,\cH(P_t f) & =\frac{d}{dt}\,\int_X P_t f \log\big(P_t f\big)\,d\mu\\
& = \int_X (LP_t f) \log\big(P_t f\big)\,d\mu+\int_X  LP_t f\,d\mu.
\end{align*}  
The last integral vanishes by invariance of $\mu$. Since $L=L^*$ w.r.t.\ $\mu$ (by reversibility) and applying Lemma \ref{PalmeFI}, the first integral
can be rewritten as
\begin{align}
 \int_X (LP_t f) \log\big(P_t f\big)\,d\mu &= \int_X P_t f L\log\big(P_t f\big)\,d\mu\nonumber\\
 & =  \int_X P_t f \Big(\frac{1}{P_t f}\,LP_t f-\Psi_\Upsilon\big(\log(P_t f)\big)\Big)\,d\mu\label{HI1}\\
 & = -\int_X P_t f \,\Psi_\Upsilon\big(\log(P_t f)\big)\,d\mu=-\cI(P_t f),\nonumber
\end{align}
where we use again invariance of $\mu$. This shows \eqref{TimeDeriEntropy}. 

Observe that the only difference of the previous proof to the one in the diffusion setting is the usage of the fundamental
identity from Lemma \ref{PalmeFI} in line \eqref{HI1}. In the diffusion setting, at this place, one employs the chain rule
\[
L\log\big(P_t f\big)=\frac{1}{P_t f}\,L P_t f-\Gamma\big(\log(P_t f)\big),
\] 
which then leads to the desired result as the Fisher information is defined as
\begin{equation} \label{FisherDiff}
\cI(\rho)=\int_X \rho \,\Gamma\big(\log \rho\big)\,d\mu=\int_X \frac{1}{\rho}\,\Gamma(\rho)\,d\mu.
\end{equation}

We turn now to derive an estimate for the second time derivative of the entropy, which is the key idea in the 
Bakry-\'Emery approach as described above. We will prove a representation formula for the second time derivative of the entropy which 
is a discrete version of the crucial identity
\begin{equation} \label{SecEntDiff}
\frac{d^2}{dt^2}\,\cH(P_t f) = 2\int_X P_t f \,\Gamma_2\big(\log (P_t f)\big)\,d\mu
\end{equation}
from the diffusion setting and is tailor-made for applying a $CD_\Upsilon$ condition.
We will also see that the whole line of arguments is a
natural discrete analogue of a suitably chosen proof of \eqref{SecEntDiff} in the diffusion setting. To make this explicit and to motivate our 
discrete computations we will first give such a proof in the diffusion setting and treat the discrete case afterwards.
The following argument is also interesting in its own, regardless of the question of adaptability to the discrete setting.
It does not seem to be included in standard references such as \cite{BGL} and \cite{Jn}, at least in the form we present here. 

The Fisher information being given by \eqref{FisherDiff}, we have in the diffusion setting
\begin{align*}
\frac{d^2}{dt^2}\,\cH(P_t f) & =-\frac{d}{dt}\,\cI(P_t f)=
-\frac{d}{dt}\,\int_X P_t f \,\Gamma\big(\log(P_t f)\big)\,d\mu\\
& = -\int_X  (LP_t f) \,\Gamma\big(\log(P_t f)\big)\,d\mu
-2\int_X P_t f \,\Gamma\Big(\log(P_t f),\frac{d}{dt}\,\log(P_t f)\Big)\,d\mu\\
& = -\int_X  P_t f \,L\Gamma\big(\log(P_t f)\big)\,d\mu
-2\int_X P_t f \,\Gamma\Big(\log(P_t f),\frac{1}{P_t f}LP_t f\Big)\,d\mu.
\end{align*} 
Inserting the identity
\[
\frac{1}{P_t f}\,L P_t f=L \log(P_t f)+\Gamma\big(\log(P_t f)\big),
\]    
cf. \eqref{logChain}, we obtain 
\begin{align*}
2\int_X P_t f \,\Gamma & \Big(\log(P_t f),\frac{1}{P_t f}LP_t f\Big)\,d\mu=
2\int_X P_t f \,\Gamma\big(\log(P_t f), L\log(P_t f)\big)\,d\mu\\
& +2\int_X P_t f \,\Gamma\Big(\log(P_t f), \Gamma\big(\log(P_t f)\big)\Big)\,d\mu.
\end{align*}
Applying to the last integral the identity \eqref{GammaExp}  with $g=\log(P_t f)$ and $h=\Gamma(\log(P_t f))$, 
we find that
\[
2\int_X P_t f \,\Gamma\Big(\log(P_t f), \Gamma\big(\log(P_t f)\big)\Big)\,d\mu
=-2 \int_X P_t f\, L\Gamma\big(\log(P_t f)\big)\,d\mu.
\]
Combining the previous relations yields
\begin{align*}
\frac{d^2}{dt^2}\,\cH(P_t f) & =\int_X P_t f\, L\Gamma\big(\log(P_t f)\big)\,d\mu-
2\int_X P_t f \,\Gamma\Big(\log(P_t f), L\log(P_t f)\Big)\,d\mu\\
& = 2\int_X P_t f \,\Gamma_2\big(\log (P_t f)\big)\,d\mu.
\end{align*}

In the discrete case, we have now the following fundamental result.
\begin{theorem} \label{ThmSecTDEnt}
Let $M_1 \in \ell^2(\mu)$ and $M_2 \in \ell^1(\mu)$. Then for any $f \in \mathcal{P}_*^+(X)$
\begin{equation} \label{SecondTimeDeriEntropy}
\frac{d^2}{dt^2}\,\cH(P_t f)=2\int_X P_t f \,\Psi_{2,\Upsilon}\big(\log (P_t f)\big)\,d\mu ,\quad t\ge 0
\end{equation} 
holds true.
\end{theorem}
\begin{proof}
\begin{align*}
\frac{d^2}{dt^2}\,\cH(P_t f) & =-\frac{d}{dt}\,\cI(P_t f)=-\frac{d}{dt}\,\int_X P_t f \,\Psi_\Upsilon\big(\log(P_t f)\big)\,d\mu\\
& = -\int_X  (LP_t f) \,\Psi_\Upsilon\big(\log(P_t f)\big)\,d\mu
-\int_X P_t f \,B_{\Upsilon'}\Big(\log(P_t f),\frac{d}{dt}\,\log(P_t f)\Big)\,d\mu\\
& = -\int_X  P_t f \,L\Psi_\Upsilon\big(\log(P_t f)\big)\,d\mu
-\int_X P_t f \,B_{\Upsilon'}\Big(\log(P_t f),\frac{1}{P_t f}LP_t f\Big)\,d\mu
\end{align*}
Using the fundamental identity from Lemma \ref{PalmeFI} and the linearity of the operator $B_{\Upsilon'}$ in the second argument we obtain
\begin{align*}
\int_X P_t f \,B_{\Upsilon'} & \Big(\log(P_t f),\frac{1}{P_t f}LP_t f\Big)\,d\mu=
\int_X P_t f \,B_{\Upsilon'}\big(\log(P_t f), L\log(P_t f)\big)\,d\mu\\
& +\int_X P_t f \,B_{\Upsilon'}\Big(\log(P_t f), \Psi_\Upsilon\big(\log(P_t f)\big)\Big)\,d\mu.
\end{align*}
Invoking Lemma \ref{BPalmeID} with $g=\log(P_t f)$ and $h=\Psi_\Upsilon(\log(P_t f))$, the last integral can be rewritten as
\[
\int_X P_t f \,B_{\Upsilon'}\Big(\log(P_t f), \Psi_\Upsilon\big(\log(P_t f)\big)\Big)\,d\mu
=-2 \int_X P_t f\, L\Psi_\Upsilon\big(\log(P_t f)\big)\,d\mu.
\]
Combining the previous relations yields
\begin{align*}
\frac{d^2}{dt^2}\,\cH(P_t f) & =\int_X P_t f\, L\Psi_\Upsilon\big(\log(P_t f)\big)\,d\mu-
\int_X P_t f \,B_{\Upsilon'}\Big(\log(P_t f), L\log(P_t f)\Big)\,d\mu\\
& = 2\int_X P_t f \,\Psi_{2,\Upsilon}\big(\log (P_t f)\big)\,d\mu.
\end{align*}
This proves the theorem.
\end{proof}
From Theorem \ref{ThmSecTDEnt} and Lemma \ref{MLSIapproximation} we conclude the following result.
\begin{corollary}\label{mLSIoutofCDPalme}
If the Markov generator $L$ satisfies $CD_\Upsilon(\kappa,\infty)$, $\kappa>0$, $M_1 \in \ell^2(\mu)$ and $M_2 \in \ell^1(\mu)$, then $L$ satisfies $\mathrm{MLSI}(\kappa)$. In particular, the entropy decay estimate \eqref{entropydecayequivtoMLSI} (with $\alpha = \kappa$) holds true for any $f \in \mathcal{P}_*^+(X)$.
\end{corollary}
Note that the $CD_\Upsilon(\kappa,\infty)$ condition has been applied to $\log(P_t f)$ in order to deduce Corollary \ref{mLSIoutofCDPalme}. However, due to ergodicity $\log(P_t f)$ is relatively small for large $t$, which motivates the following improvement of \eqref{entropydecayequivtoMLSI}.
\begin{theorem} \label{improvedDecay}
Let $X$ be finite and assume that the CD conditions $CD(\kappa_0,\infty)$ and $CD_\Upsilon(\kappa,\infty)$
are satisfied with $\kappa_0> \kappa>0$. Let $\pi_0=\min_{x\in X} \pi(x)$. Then for any 
$\varepsilon_0\in (0,1-\frac{\kappa}{\kappa_0})$ there exist
constants $h=h(\varepsilon_0,\kappa_0, |M_1|_\infty, \pi_0)>0$
and
$C=C(\varepsilon_0, \kappa_0, \kappa, |M_1|_\infty, \pi_0)>0$ 
such that
for any $f\in \mathcal{P}_*^+(X)$ the following holds true.

(i) If $\cH(f)\le h$ then
\begin{equation}  \label{entropyupgrade1}
\cH(P_t f)\le e^{-2(1-\varepsilon_0)\kappa_0 t} \cH(f),\quad t\ge 0.
\end{equation}

(ii) If $\cH(f)>h$ then
\begin{equation} \label{entropyupgrade2}
\cH(P_t f)\le C e^{-2(1-\varepsilon_0)\kappa_0 t}
\cH(f)^{\frac{\kappa_0}{\kappa}},\quad t\ge 0.
\end{equation}  
\end{theorem}
\begin{proof}
Recall the Csisz\'ar-Kullback-Pinsker inequality
\[
|\nu-\mu|_{TV}^2\le \frac{1}{2}H(\nu|\mu),
\]
where $H(\nu|\mu)$ denotes the relative entropy of the probability measure $\nu$ with respect to the reference probability
measure $\mu$. If $\nu\ll\mu$ and $g=\frac{d\nu}{d\mu}$, we have $H(\nu|\mu)=\int_X g\log g\,d\mu$ and
\[
|\nu-\mu|_{TV}=\frac{1}{2}\int_X|g-1|\,d\mu.
\]
For $d\nu=P_t f\,d\mu$ this gives
\[
|P_t f(x)-1| \pi(x)\le \int_X |P_t f-1|\,d\mu\le \sqrt{2\cH(P_t f)},\quad t\ge 0,\;x\in X,
\]
and thus
\begin{equation} \label{Ptfestimate}
|P_t f(x)-1|\le \frac{1}{\pi_0}\,e^{-\kappa t}\sqrt{2 \cH(f)},\quad t\ge 0,\;x\in X.
\end{equation}
Let $c\ge 1$ such that $|\log y|\le c|y-1|$ for all $y\in [\frac{1}{2},\frac{3}{2}]$. 
Given $\varepsilon_0\in (0,1-\frac{\kappa}{\kappa_0})$ we choose $\varepsilon\in(0,\frac{1}{2})$ so small that
\begin{equation} \label{choiceps}
\frac{(1-2\varepsilon) \kappa_0- 12M \varepsilon}{1+\varepsilon}\ge (1-\varepsilon_0)\kappa_0.
\end{equation}
We then fix $\delta_\varepsilon\in(0,\varepsilon]$ such that the two-sided quadratic estimate \eqref{PalmeEstimate} holds. 
Next, define $T_{\varepsilon}=0$ if $2c\sqrt{2\cH(f)}\le\delta_\varepsilon \pi_0$ and set
\begin{equation} \label{Teps}
T_{\varepsilon}=-\frac{1}{\kappa}\,\log\Big(\frac{\delta_\varepsilon \pi_0}{2c\sqrt{2\cH(f)}}\Big)
\end{equation}
otherwise. Employing \eqref{Ptfestimate}, one can easily check that
\[
|P_t f(x)-1|\le \frac{\delta_\varepsilon}{2c}\le \frac{1}{2},\quad t\ge T_{\varepsilon},\;x\in X,
\]
which further entails that
\begin{equation} \label{logPtf}
2|\log\big(P_t f(x)\big)|\le \delta_\varepsilon,\quad t\ge T_{\varepsilon},\;x\in X.
\end{equation}
Thanks to \eqref{logPtf} we may apply Proposition \ref{GammaPsiCompare} with $M=|M_1|_\infty$ to the function $\log(P_t f)$,
which together with the condition $CD(\kappa_0,\infty)$ and \eqref{choiceps} implies that 
for all $t\ge T_{\varepsilon}$,
\begin{align*}
\Psi_{2,\Upsilon}\big(\log(P_t f)\big) & \ge  (1-2\varepsilon)\Gamma_2\big(\log(P_t f)\big)- 6M(\varepsilon+
\delta_\varepsilon)\Gamma\big(\log(P_t f)\big)\\
& \ge  \big((1-2\varepsilon) \kappa_0- 12 M\varepsilon\big) \Gamma\big(\log(P_t f)\big)\\
& \ge \frac{1}{1+\varepsilon}\, \big((1-2\varepsilon) \kappa_0- 12M \varepsilon\big) 
\Psi_\Upsilon\big(\log(P_t f)\big)\\
& \ge (1-\varepsilon_0)\kappa_0\Psi_\Upsilon\big(\log(P_t f)\big).
\end{align*}
Inserting the latter estimate in the representation formula \eqref{SecondTimeDeriEntropy} from Theorem
\ref{ThmSecTDEnt} then yields that for all $t\ge  T_{\varepsilon}$,
\begin{align*}
\frac{d^2}{dt^2}\,\cH(P_t f) &=2\int_X P_t f \,\Psi_{2,\Upsilon}\big(\log (P_t f)\big)\,d\mu\\
& \ge  2(1-\varepsilon_0)\kappa_0 \int_X P_t f \,\Psi_{\Upsilon}\big(\log (P_t f)\big)\,d\mu\\
& = 2(1-\varepsilon_0)\kappa_0\,\cI(P_t f)= -2(1-\varepsilon_0)\kappa_0\,\frac{d}{dt}\,\cH(P_t f).
\end{align*}
cf.\ \eqref{DefFisher}. Setting $\lambda=(1-\varepsilon_0)\kappa_0\,(>\kappa)$ and using $\mathcal{H}(P_t f) \to 0$ as $t \to \infty$, this implies
\begin{equation} \label{betterdecay1}
\cH(P_t f)\le e^{-2\lambda(t-T_\varepsilon)}\cH(P_{T_\varepsilon}f)
,\quad t\ge T_{\varepsilon}.
\end{equation}
This shows assertion (i) with 
\[
h=\frac{(\delta_\varepsilon\pi_0 )^2}{8c^2},
\]
since $\cH(f)\le h$ is equivalent to $T_\varepsilon=0$.

By the entropy decay estimate resulting from $CD_\Upsilon(\kappa,\infty)$
(see Corollary \ref{mLSIoutofCDPalme}),
\[
\cH(P_{T_\varepsilon}f)\le e^{-2\kappa T_\varepsilon}\cH(f).
\]
Combining this with \eqref{betterdecay1} gives
\begin{align*}
\cH(P_t f)\le e^{-2\lambda t} e^{2T_\varepsilon(\lambda-\kappa)} \cH(f),\quad t\ge T_{\varepsilon}.
\end{align*}

Suppose that $2c\sqrt{2\cH(f)}>\delta_\varepsilon \pi_0$, which is equivalent to $\cH(f)>h$ and means that
$T_\varepsilon>0$ is given by \eqref{Teps}. Then
\begin{equation} \label{Tepsest}
e^{2T_\varepsilon(\lambda-\kappa)}=
\Big(\frac{2c\sqrt{2\cH(f)}}{\delta_\varepsilon \pi_0}\Big)^{\frac{2}{\kappa}(\lambda-\kappa)}
=\Big(\frac{2c\sqrt{2}}{\delta_\varepsilon \pi_0}\Big)^{\frac{2}{\kappa}(\lambda-\kappa)}
\cH(f)^{\frac{\lambda}{\kappa}-1},
\end{equation}
and thus
\[
\cH(P_t f)\le e^{-2\lambda t}\Big(\frac{2c\sqrt{2}}{\delta_\varepsilon \pi_0}\Big)^{\frac{2}{\kappa}(\lambda-\kappa)}
\cH(f)^{\frac{\lambda}{\kappa}},\quad t\ge T_{\varepsilon}.
\]
Since $\cH(f)>h$, we have $\cH(f)^{\frac{\lambda}{\kappa}}\le \cH(f)^{\frac{\kappa_0}{\kappa}}h^{-\varepsilon_0
\frac{\kappa_0}{\kappa}}$ and therefore
\[
\cH(P_t f)\le C e^{-2\lambda t}\cH(f)^{\frac{\kappa_0}{\kappa}},\quad t\ge T_\varepsilon,\quad
\mbox{with}\;\,C=\Big(\frac{2c\sqrt{2}}{\delta_\varepsilon \pi_0}\Big)^{\frac{2}{\kappa}(\lambda-\kappa)}
h^{-\varepsilon_0
\frac{\kappa_0}{\kappa}}.
\]

Finally, for $0\le t\le T_\varepsilon$, we may estimate
\[
\cH(P_t f) \le e^{-2\kappa t} \cH(f)\le e^{-2\lambda t} e^{2T_\varepsilon(\lambda-\kappa)} \cH(f)
\]
and use again \eqref{Tepsest} as before to see that $\cH(P_t f)\le C e^{-2\lambda t}\cH(f)^{\frac{\kappa_0}{\kappa}}$.
This proves assertion (ii).
\end{proof}
%
%
\section{Tensorization property}\label{tensandgrad}
In this section we investigate the tensorization property for  the $CD_\Upsilon$ condition, which is fundamental for the classical Bakry-\'Emery curvature-dimension condition and has been extended (see \cite{BGL} for the diffusive setting) to the discrete setting in \cite{LiPe}.

 Let $L_1$ and $L_2$ generate a Markov chain on state spaces $X_1$ and $X_2$ respectively. The operator $L:=L_1 \oplus L_2$  acts on functions $f:X_1 \times X_2 \to \R$ as $(L_1 \oplus L_2)  (f)(x,y) =  L_1 (f_y)(x)+L_2(f^x)(y)$, where we set $f_y(\cdot):= f(\cdot , y)$, $f^x(\cdot)=f(x,\cdot)$. The transition rates of the Markov process  generated by $L_1 \oplus L_2$ (on the state space $X_1 \times X_2$) are given by 
\begin{equation*}
k_\oplus ((x_1,y_1),(x_2,y_2))= \left\{\begin{array}{lll} k_1(x_1,x_2), & \text{ if } y_1=y_2, \\
         k_2(y_1,y_2), & \text{ if } x_1=x_2, \\
         0, & \text{ else, }\end{array}\right. 
\end{equation*} 
where $k_\nu$ denote the transition rates for the process generated by $L_\nu$, $\nu \in \{1,2\}$.
In particular, the operator in \eqref{def:Psi} and  respectively that in \eqref{def:BH}, both corresponding to the transition rates $k_\oplus$, are given by 
\begin{align*}
\Psi_H(f)(x,y) &= \Psi_H^{(1)} (f_y)(x) + \Psi_H^{(2)}(f^x)(y),\\
B_H(f,g)(x,y) &= B_H^{(1)}(f_y,g_y)(x)+B_H^{(2)}(f^x,g^x)(y),
\end{align*}
where  $H:\R \to \R$, $f,g : X_1 \times X_2 \to \R$ and upper indices were used to indicate the corresponding state space.

We are now ready to analyze the operator $\Psi_{2,H}$ corresponding to $k_\oplus$. To that aim, we first provide the following crucial estimate.
\begin{lemma}\label{lem:iteratedcdcestimate}
Let $H\in C^1(\R)$ be convex. Then we have
\begin{equation}\label{crucialestimate}
\begin{split}
L \Psi_H (f) (x,y) - B_{H'}(f,Lf)(x,y) &\geq L_1 \Psi_H^{(1)}(f_y)(x) - B_{H'}^{(1)}(f_y , L_1 f_y)(x)\\
& + L_2 \Psi_H^{(2)}(f^x)(y) - B_{H'}^{(2)}(f^x , L_1 f^x)(y)
\end{split}
\end{equation}
for any $f \in \ell^\infty(X)$.
\end{lemma}
\begin{proof}
We observe that
\begin{align*}
&L \Psi_{H}(f)(x,y) = \sum\limits_{x_i \in X_1} k_1(x,x_i) ( \Psi_{H}(f)(x_i,y)- \Psi_{H}(f)(x,y)) \\
&\qquad+ \sum\limits_{y_j \in X_2} k_2(y,y_j) ( \Psi_{H}(f)(x,y_j)- \Psi_{H}(f)(x,y))\\
&= L_1 \Psi_{H}^{(1)}(f_y)(x) + L_2 \Psi_H^{(2)}(f^x)(y) \\ &\qquad + \sum\limits_{x_i \in X_1}\sum\limits_{y_j \in X_2} k_1(x,x_i) k_2(y,y_j) \big( H(f(x_i,y_j)-f(x_i,y)) - H(f(x,y_j)-f(x,y)) \big) \\
\\ &\qquad + \sum\limits_{x_i \in X_1}\sum\limits_{y_j \in X_2} k_1(x,x_i) k_2(y,y_j) \big( H(f(x_i,y_j) - f(x,y_j)) - H(f(x_i,y) - f(x,y)) \big).
\end{align*}
Furthermore,
\begin{align*}
&B_{H'}(f,Lf)(x,y) = \sum\limits_{x_i \in X_1} k_1(x,x_i) H'( f(x_i,y) - f(x,y)) ( L f(x_i,y) - L f(x,y) ) \\
&\qquad + \sum\limits_{y_j \in X_2} k_2(y,y_j) H'(f(x,y_j)-f(x,y)) (L f(x,y_j) - L f(x,y) )\\
&= B_{H'}^{(1)}(f_y,L_1 f_y )(x)+ B_{H'}^{(2)}(f^x,L_2 f^x) (y)\\ &\; + \sum\limits_{x_i \in X_1} \sum\limits_{y_j \in X_2} k_1(x,x_i) k_2(y,y_j) H'(f(x_i,y) - f(x,y)) \big(f(x_i,y_j)-f(x_i,y) - (f(x,y_j) - f(x,y))\big) \\
\\ &\; + \sum\limits_{x_i \in X_1} \sum\limits_{y_j \in X_2} k_1(x,x_i) k_2(y,y_j) H'(f(x,y_j)-f(x,y)) \big( f(x_i,y_j) - f(x,y_j) - ( f(x_i,y) - f(x,y))\big).
\end{align*}
Hence, we deduce
\begin{align*}
&L \Psi_{H} (f)(x,y) - B_{H'}(f,Lf)(x,y)\\ 
&\quad= L_1 \Psi_H^{(1)}(f_y)(x) - B_{H'}^{(1)}(f_y , L_1 f_y)(x) + L_2 \Psi_H^{(2)}(f^x)(y) - B_{H'}^{(2)}(f^x , L_2 f^x)(x)\\
&\qquad + \sum\limits_{x_i \in X_1} \sum\limits_{y_j \in X_2} k_1(x,x_i) k_2(y,y_j) \Big[ H(f(x_i,y_j)-f(x_i,y)) - H(f(x_i,y)-f(x,y))\\ &\qquad\qquad\qquad\qquad \qquad- H'(f(x,y_j)-f(x,y)) \big( f(x_i,y_j) - f(x,y_j) - ( f(x_i,y) - f(x,y))\big) \Big] \\
&\qquad + \sum\limits_{x_i \in X_1} \sum\limits_{y_j \in X_2} k_1(x,x_i) k_2(y,y_j) \Big[ H(f(x_i,y_j)-f(x,y_j)) - H(f(x_i,y) - f(x,y))\\ &\qquad\qquad\qquad\qquad \qquad - H'(f(x_i,y) - f(x,y)) \big(f(x_i,y_j)-f(x_i,y) - (f(x,y_j) - f(x,y))\big) \Big].
\end{align*}
The convexity of $H$ implies that $H(a)-H(b) \geq H'(b)(a-b)$ for all $a,b \in \R$. Inspecting above expressions, we thus conclude the claim.
\end{proof}
For the choice of $H(r)=\frac{r^2}{2}$ the statement of Lemma \ref{lem:iteratedcdcestimate} has been obtained in a similar way in \cite{LiPe} and served as the key estimate in order to show the tensorization property for the Bakry-\'Emery curvature-dimension condition. For our purposes, we choose $H(r)=\Upsilon(r)$.
\begin{theorem}\label{theo:tensorproperty}
If $L_1$ satisfies $CD_\Upsilon(\kappa_1, \infty)$ and $L_2$ satisfies $CD_\Upsilon(\kappa_2, \infty)$ for $\kappa_1,\kappa_2 \in \R$, then $L_1 \oplus L_2$ satisfies $CD_\Upsilon(\kappa,\infty)$ with $\kappa= \min\{\kappa_1,\kappa_2 \}$.
\end{theorem}
\begin{proof}
From Lemma \ref{lem:iteratedcdcestimate}, we obtain 
\begin{align*}
\Psi_{2,\Upsilon}(f) (x,y) &\geq \Psi_{2,\Upsilon}^{(1)}(f_y)(x) + \Psi_{2,\Upsilon}^{(2)}(f^x)(y)\\ 
&\geq \kappa_1 \Psi_{\Upsilon}^{(1)} (f_y)(x) + \kappa_2 \Psi_\Upsilon^{(2)} (f^x)(y)\\
&\geq \kappa \ (\Psi_\Upsilon^{(1)}(f_y)(x) + \Psi_\Upsilon^{(2)}(f^x)(y))  \\
&= \kappa \ \Psi_\Upsilon(f)(x,y). 
\end{align*}
\end{proof}
\begin{remark} \label{RemarkTensorDim}
The above tensorization property extends to the $CD_\Upsilon(\kappa,d)$ condition as defined in Remark \ref{commentsCD}(ii). Indeed, if $L_i$ satisfies $CD_\Upsilon(\kappa_i,d_i)$ with $\kappa_i \in \R$ and $d_i \in (0,\infty]$, $i\in \{1,2\}$, then $L_1 \oplus L_2$ satisfies $CD_\Upsilon(\kappa , d)$ with $\kappa= \min\{\kappa_1,\kappa_2 \}$ and $d=d_1+d_2$. This follows analogously as in \cite{LiPe}, using Young's inequality to treat the dimension terms.
\end{remark}
\section{Examples}\label{sec:examplesection}
\begin{example}[\textit{The two-point space}] \label{TwoPoint}
We consider  $X=\{0,1\}$ with $k(0,1)=a$ and $k(1,0)=b$, where $a,b>0$. 
Here, the invariant and reversible probability measure $\mu$ is given by $d\mu=\pi d\#$ with $\pi(0)=\frac{b}{a+b}$ and $\pi(1)=\frac{a}{a+b}$.
Writing $\tilde{x}=1-x$ for $x\in X$ we have
\[
\Psi_\Upsilon(f)(x)=k(x,\tilde{x})\Upsilon\big(f(\tilde{x})-f(x)\big)
\]
and
\begin{align*}
2\Psi_{2,\Upsilon}&(f)(x) = L\Psi_\Upsilon(f)(x)-B_{\Upsilon'}(f,Lf)(x)\\
& = k(x,\tilde{x})\Big(\big(\Psi_\Upsilon(f)(\tilde{x})-\Psi_\Upsilon(f)(x)  \big)
- \Upsilon'\big(f(\tilde{x})-f(x)  \big)\big(Lf(\tilde{x})-Lf(x)\big)\Big)\\
& = k(x,\tilde{x}) \,k(\tilde{x},x)\big[\Upsilon\big(f(x)-f(\tilde{x})\big)-\Upsilon'\big(f(\tilde{x})-f(x)\big) 
\big(f(x)-f(\tilde{x})\big)\big]\\
& \;\quad +k(x,\tilde{x})^2\big[-\Upsilon\big(f(\tilde{x})-f(x)\big)+ \Upsilon'\big(f(\tilde{x})-f(x)\big) \big( f(\tilde{x})-f(x)\big) \big].
\end{align*}
Setting $t=f(\tilde{x})-f(x)$, we thus have $\Psi_{2,\Upsilon}(f)(x)\ge \kappa \Psi_\Upsilon(f)(x)$ for all $f\in \iR^X$ if and only if
\begin{equation} \label{twopoint1}
k(\tilde{x},x)\big(e^{-t}-1+te^t\big)+k(x,\tilde{x})\big(te^t-e^t+1\big)\ge 2\kappa \Upsilon(t),\quad t\in \iR.
\end{equation}
Note that
\[
\omega(t):=te^t-e^t+1=\Upsilon'(t)t-\Upsilon(t)> 0,\quad t\in \iR\setminus \{0\},
\]
since for any $t\in \iR\setminus \{0\}$ there is some $\xi$ between $0$ and $t$ such that
\[
0=\Upsilon(0)= \Upsilon(t)+\Upsilon'(t)(-t)+\frac{1}{2}e^{\xi}t^2>\Upsilon(t)+\Upsilon'(t)(-t).
\]
Further, $\omega(t)\sim \frac{1}{2}t^2$ as $t\to 0$ and thus $\omega(t)/\Upsilon(t)\to 1$ as $t\to 0$.
Note also that
\[
e^{-t}-1+te^t=\omega(t)+\big(e^t+e^{-t}-2\big)=\omega(t)+\big(e^{t/2}-e^{-t/2}\big)^2>0,\quad t\in \iR\setminus\{0\},
\]
and $e^{-t}-1+te^t\sim \frac{3}{2}t^2$ as $t\to 0$. Consequently, for $x=0$, the condition \eqref{twopoint1} is equivalent to
\begin{equation}  \label{twopoint2}
b\, \frac{\omega(t)+\big(e^{t/2}-e^{-t/2}\big)^2}{\Upsilon(t)}+a\, \frac{\omega(t)}{\Upsilon(t)}\ge 2\kappa,\quad 
t\in \iR\setminus \{0\}. 
\end{equation}
Observe that $\omega(t)/\Upsilon(t)\to 0$ as $t\to -\infty$ and $(e^{t/2}-e^{-t/2})^2/\Upsilon(t)\to 1$ as $t\to \infty$.
On the other hand, $\omega(t)/\Upsilon(t)\to \infty$ as $t\to \infty$ and $(e^{t/2}-e^{-t/2})^2/\Upsilon(t)\to \infty$ as $t\to
-\infty$. This implies that for any $a,b>0$ there exists a maximal $\kappa=\kappa_0(a,b)>0$ such that \eqref{twopoint2} holds true,
and for arbitrarily fixed $a>0$ we have $\kappa_0(a,b)\to 0$ as $b\to 0+$. Since $(e^{t/2}-e^{-t/2})^2/\Upsilon(t)>1$
for all $t\in \iR\setminus\{0\}$ (with limit $2$ at $t=0$), it follows that $\kappa_0(a,b)\ge \frac{b}{2}$.
By symmetry, one obtains analogous statements concerning the validity of \eqref{twopoint1} at $x=1$, i.e.
for any $a,b>0$ there is a maximal $\kappa=\kappa_1(a,b)>0$ such that this condition is satisfied, and $\kappa_1(a,b)\to 0$
as $a\to 0+$ as well as $\kappa_1(a,b)\ge \frac{a}{2}$. Consequently, $CD_\Upsilon(\kappa,\infty)$ is satisfied
with optimal $\kappa=\min(\kappa_0(a,b),\kappa_1(a,b))$.

As to the classical Bakry-\'Emery condition, one finds that
\[
\Gamma(f)(x)=\frac{1}{2}\,k(x,\tilde{x})\big(f(\tilde{x})-f(x)\big)^2
\]
and
\begin{align*}
2\Gamma_2 &(f)(x)=L \Gamma(f)(x)-2\Gamma(f, Lf)(x)\\
& = k(x,\tilde{x})\Big(\big(\Gamma(f)(\tilde{x})-\Gamma(f)(x)  \big)
- \big(f(\tilde{x})-f(x)  \big)\big(Lf(\tilde{x})-Lf(x)\big)\Big)\\
& = \frac{3}{2} k(x,\tilde{x}) \,k(\tilde{x},x) \big(f(x)-f(\tilde{x})\big)^2+
\frac{1}{2} k(x,\tilde{x})^2\big(f(\tilde{x})-f(x)\big)^2
\end{align*}
which shows that $\Gamma_2(f)(x)\ge \kappa \Gamma(f)(x)$ for all $f\in \iR^X$ if and only if
\[
\kappa\le \frac{1}{2}\,\big(3k(\tilde{x},x)+k(x,\tilde{x})\big).
\]
Consequently, $CD(\kappa,\infty)$ holds true with optimal 
\[
\kappa=\frac{a+b}{2}+\min(a,b).
\]

Comparing the two estimates, we see that keeping one of the two parameters $a$ and $b$ fixed and sending the other to zero,
the optimal curvature constant $\kappa$ in $CD_\Upsilon(\kappa,\infty)$ tends to zero as well whereas the optimal curvature bound
in $CD(\kappa,\infty)$ stays bounded away from zero.
\end{example}
\begin{example}[\textit{Weighted complete graphs}]\label{ex:kequall}
Here we consider an arbitrary countable set $X$ with at least two elements.
Let $l:\,X\to (0,\infty)$ be integrable on $X$ and $k(x,y)=l(y)$ for all $x,y\in X$ with $x\neq y$. Then $\mu$ given by
$d\mu=\pi d\#$ with $\pi(x)=l(x)$, $x\in X$, is an invariant and reversible measure. We have
\[
\Psi_\Upsilon(f)(x)=\sum_{y\in X}l(y)\Upsilon\big(f(y)-f(x)\big),
\]
and by \eqref{Psi2Formel}
\begin{align*}
2\Psi_{2,\Upsilon} &(f)(x)  =\sum_{y\in X\setminus\{x\}}\sum_{z\in X} l(y)l(z)\Big(\Upsilon\big(f(z)-f(y)\big)-
\Upsilon'\big(f(y)-f(x)\big)\big(f(z)-f(y)\big)\Big)\\
& \quad\; +\sum_{y\in X\setminus\{x\}}\sum_{z\in X} l(y) l(z) \Upsilon'\big(f(y)-f(x)\big)\big(f(z)-f(x)\big)\\
& \quad\; - \sum_{y\in X\setminus\{x\}}\sum_{z\in X} l(y) l(z) \Upsilon\big(f(z)-f(x)\big)\\
& = \sum_{y\in X\setminus\{x\}}\sum_{z\in X} l(y) l(z) \Big(  \Upsilon'\big(f(y)-f(x)\big)\big(f(y)-f(x)\big)
-\Upsilon\big(f(y)-f(x)\big)\Big)\\
& \quad\; +\sum_{y\in X\setminus\{x\}}\sum_{z\in X} l(y)l(z)\Upsilon\big(f(z)-f(y)\big)
+l(x)\Psi_\Upsilon(f)(x)\\
&= \sum_{y\in X\setminus\{x\}} l(y) \Big( |l|_1 \Upsilon'(f(y)-f(x)) (f(y)-f(x)) + l(x) \Upsilon(f(x)-f(y)) \\
& \qquad\qquad \quad - \big(|l|_1 - l(x)\big)\Upsilon(f(y)-f(x)) \Big) + \sum_{\substack{y,z \in X \\ y,z \neq x}}l(y)l(z) \Upsilon(f(z)-f(y)).
\end{align*}
We introduce the functions $\nu_{c,d}:\R \to \R$, given by 
\begin{equation}\label{nucddefi}
\nu_{c,d}(r)= c \Upsilon'(r)r + \Upsilon(-r) - d \Upsilon(r),
\end{equation}
with constants $c,d \in \R$. In Appendix \ref{Appendixnucd}, we collect several basic properties of these functions, which will be of great importance throughout this section.

Now, $\Psi_{2,\Upsilon}(f)(x) \geq \kappa \Psi_\Upsilon(f)(x)$ is equivalent to
\begin{equation}\label{keqlformulationCD}
l(x) \sum_{y \in X \setminus \{x\}} l(y) \nu_{\frac{|l|_1}{l(x)},\frac{|l|_1}{l(x)} + \frac{2\kappa}{l(x)}-1}(f(y)-f(x)) + \sum_{\substack{y,z \in X \\ y,z \neq x}}l(y)l(z) \Upsilon(f(z)-f(y)) \geq 0 .
\end{equation}
\begin{lemma}\label{lem:keqlCDlemma}
In the setting of the present example,
$CD_\Upsilon(\kappa,\infty)$ holds true for $\kappa \in \R$ in $x\in X$ if and only if 
\begin{equation}\label{conditionkeqlCD}
\nu_{\frac{|l|_1}{l(x)},\frac{|l|_1}{l(x)} + \frac{2\kappa}{l(x)}-1}(r) \geq 0
\end{equation}
 holds for any $r \in \R$.
\end{lemma}
\begin{proof}
If $\nu_{\frac{|l|_1}{l(x)},\frac{|l|_1}{l(x)} + \frac{2\kappa}{l(x)}-1}$ is non-negative, then $CD_\Upsilon(\kappa,\infty)$ in $x \in X$ follows by \eqref{keqlformulationCD}.
Conversely, assume that there exists some $r_x\in \R$ such that $\nu_{\frac{|l|_1}{l(x)},\frac{|l|_1}{l(x)} + \frac{2\kappa}{l(x)}-1}(r_x) < 0$. We define $f \in \R^X$ by $f(y)=r_x$ for any $y \neq x$ and $f(x)=0$. For this choice, \eqref{keqlformulationCD} fails to be true.
\end{proof}
Continuing with Example \ref{ex:kequall}, we apply Lemma \ref{App_notlin}(ii) to \eqref{conditionkeqlCD} and observe that \eqref{keqlformulationCD} holds true if $|l|_1 \geq 2 l(x)$ and
\begin{equation*}
\frac{2 \kappa}{l(x)}-1 \leq 2^{\frac{3}{2}} \sqrt{\frac{|l|_1}{l(x)}} - 1,
\end{equation*}
where the latter is equivalent to $\kappa \leq \sqrt{2 |l|_1 l(x)}$. In the case of $|l|_1 < 2 l(x)$, Lemma \ref{App_notlin}(ii) applied to \eqref{conditionkeqlCD} implies that \eqref{keqlformulationCD} is valid if
\begin{equation*}
\frac{2\kappa}{l(x)}-1  \leq 2\frac{|l|_1}{l(x)}-1,
\end{equation*}
which is equivalent to $\kappa \leq |l|_1$. Note, that $|l|_1\geq \sqrt{2 |l|_1 l(x)}$ holds true if and only if $|l|_1 \geq 2 l(x)$, which is clearly satisfied for at least one $x \in X$.

As a consequence we obtain that $CD_\Upsilon(\sqrt{2 |l|_1 l_*},\infty)$ holds true, where $l_*:=\inf_{x\in X}l(x)$. If $X$ is of finite size, we thus deduce a positive global curvature bound.
In the case of $X$ being of infinite size, we deduce $CD_\Upsilon(0,\infty)$, which is in fact best possible.
Indeed, for $\kappa > 0$ we  write   $\nu_{\frac{|l|_1}{l(x)},\frac{|l|_1}{l(x)} + \frac{2\kappa}{l(x)}-1} = \nu_{\frac{|l|_1}{l(x)},\frac{|l|_1}{l(x)}(1+\frac{2\kappa}{|l|_1})-1}$ and apply Lemma \ref{App_notlin}(i) to deduce from Lemma \ref{lem:keqlCDlemma} that $CD_\Upsilon(\kappa,\infty)$ can not hold if we choose $x \in X$ such that $l(x)$ is small enough, which is possible since $l$ is integrable on $X$.
\end{example}
\begin{example}[\textit{The complete graph}]\label{ex:Kn}
Now we examine a quite important special case of Example \ref{ex:kequall}. Let $X$ be an arbitrary set of $n$ elements, $n \geq 2$, and $k(x,y)=1$ for all $x,y \in X$ with $x \neq y$. The graph associated to the Markov chain generated by $L$ is then given by the complete graph $K_n$. The uniform measure serves as the invariant and reversible measure.
In the setting of Example \ref{ex:kequall} this translates to $l(x)=1$, $x \in X$, i.e. $L$ satisfies $CD_\Upsilon(\sqrt{2n},\infty)$ by Example \ref{ex:kequall}.

In \cite{KRT} it is shown that the optimal Bakry-\'Emery curvature constant for the Markov generator associated to the complete graph is given  by $1+\frac{n}{2}$. Due to Proposition \ref{CDPalmeandCD} this is the best possible constant one can hope  for regarding the $CD_\Upsilon$ condition. By Lemma \ref{lem:keqlCDlemma}, the constant $1 + \frac{n}{2}$ is achieved if and only if $\nu_{n,2n+1}$ is non-negative, which is only true  for the case of $n=2$ by Lemma \ref{lem:nucdlemma}.

We collect a few implications of the properties presented in Appendix \ref{Appendixnucd}:
\begin{itemize}
\item In the case of $n=2$, the optimal constant is achieved, i.e. $L$ satisfies $CD_\Upsilon(2,\infty)$ (which is best possible since it is best possible for the Bakry-\'Emery curvature-dimension condition). This follows from Lemma \ref{lem:nucdlemma}.
\item If $n>2$, the optimal constant for the $CD_\Upsilon$ condition is strictly less than the optimal Bakry-\'Emery curvature constant by Lemma \ref{lem:nucdlemma}. Moreover, the constant does not grow linearly in $n$, which is in contrast to  Bakry-\'Emery curvature. Indeed, setting $\kappa= \alpha n + \beta$ with $\alpha>0$ and $\beta \in \R$ we find by Lemma \ref{App_notlin}(i) that there exists some $r \in \R$ and $N \in \N$ such that $\nu_{n,(1+2\alpha)n + 2\beta -1}(r)< 0$ for any $n \geq N$. Thus, $CD_\Upsilon(\kappa,\infty)$ can not hold true by Lemma \ref{lem:keqlCDlemma}. 
\item Let $f \in \R^X$ and $x_* \in X$ such that $f$ achieves its global minimum at $x_* \in X$. By \eqref{keqlformulationCD} and Lemma \ref{lem:nucdforpositive}, the estimate
\begin{equation*}
\Psi_{2,\Upsilon}(f)(x_*) \geq \big(1+\frac{n}{2}\big) \Psi_\Upsilon (f)(x_*)
\end{equation*} 
holds true. This estimate is best possible, as for the function $f_*(y)= \mathrm{dist}(y,x_*)$, where $\mathrm{dist}:X \times X \to \N_0$ denotes the combinatorical graph distance on the underlying graph, 
$ \Gamma_2(f_*)(x_*) = \big(1+\frac{n}{2}\big)\Gamma(f_*)(x_*)$ is valid
(cf. the proof of Theorem 1.2 in \cite{KRT}, where, more generally, locally finite graphs are considered).
\end{itemize}
\end{example}
\begin{remark}\label{ErbarMaascompletegraph}
Concerning the link between the $CD_\Upsilon(\kappa,\infty)$ condition and the entropic Ricci curvature from \cite{EM12}, Example \ref{ex:Kn} has the interesting consequence that the curvature notion of Erbar and Maas does not imply in general the $CD_\Upsilon(\kappa,\infty)$ condition. More precisely, it is not  true in general that there exist $\alpha>0$ and $\beta > - \frac{\alpha}{2}$ such that the entropic Ricci curvature of \cite{EM12}  with constant $\kappa>0$ does imply $CD_\Upsilon(\alpha \kappa + \beta,\infty)$. Indeed, let us denote for the moment by $\Psi_{2,\Upsilon}^{(n)}$ and $\Psi_\Upsilon^{(n)}$ the respective operators corresponding to the normalized complete graph, i.e. in the setting of Example \ref{ex:kequall} we have $l(x)=\frac{1}{n}$ for any $x \in X$, where $X$ is a set of $n$ elements. Note that this corresponds to the setting which has been considered in \cite{EM12}. Due to \cite[Example 5.1]{EM12}, the normalized complete graph has entropic curvature constant $\frac{1}{2}+\frac{1}{2n}$. We assume for contradiction that we find $\alpha>0$ and $\beta > - \frac{\alpha}{2}$ such that the normalized complete graph satisfies $CD_\Upsilon(\kappa,\infty)$ with $\kappa=\alpha \big(\frac{1}{2}+\frac{1}{2n}\big)+\beta$ for any $n \geq 2$. Then we have
\begin{align*}
\Psi_{2,\Upsilon}(f) = n^2 \Psi_{2,\Upsilon}^{(n)}(f)  \geq n^2 \kappa \Psi_\Upsilon^{(n)} (f) = n \kappa\, \Psi_\Upsilon(f),
\end{align*}
where $\Psi_{2,\Upsilon}$ and $\Psi_\Upsilon$ correspond to the non-normalized setting of Example \ref{ex:Kn}. Now, we get a contradiction to the fact that the curvature constant does not grow linearly in $n$ in Example \ref{ex:Kn}.  
\end{remark}
\begin{example}[\textit{The hypercube}]\label{ex:Hn} 
We have seen in Example \ref{ex:Kn} that the Markov generator associated to $K_2$ satisfies $CD_\Upsilon(2,\infty)$. Now, since $K_2$ corresponds to the one-dimensional hypercube, we denote the respective Markov generator by $L_1$. Iteratively, we define for $n \geq 2$
\begin{equation*}
L_n := L_{n-1} \oplus L_{n-1}
\end{equation*}
in the sense of Section \ref{tensandgrad}. $L_n$ is then defined on the state space $X^n$, where $X$ is a set of two elements. The transition rates for $L_n$ are  given by $k_n(x,y)=1$, $x, y \in X^n$, if the Hamming distance between $x$ and $y$ equals $1$ and $0$ otherwise. The respective invariant and reversible measure is the $n$-fold product measure of $X$ and hence the uniform measure on $X^n$ and the underlying graph to $L_n$ is the hypercube $H_n$.

Combining Theorem \ref{theo:tensorproperty} with the fact that $L_1$ satisfies $CD_\Upsilon(2,\infty)$ we conclude that $L_n$ satisfies $CD_\Upsilon(2,\infty)$ for any $n \in \N$. Note that this constant is optimal as it is optimal with respect to Bakry-\'Emery curvature (cf. \cite{KRT}).
\end{example}
\begin{remark}
Following Example 3.7 in \cite{BoTe}, the constant $\kappa=2$ yields the optimal constant in the respective modified logarithmic Sobolev inequality in case of the hypercube (note that the Dirichlet form considered in \cite[Example 3.7]{BoTe} equals $2 \mathcal{E}(f,g)$, with $\mathcal{E}$ given by \eqref{Dirichletform}).
\end{remark}
\begin{example}[\textit{Weighted $4$-cycles}]\label{ex:weighted4cyc}
Let $X=\{x_1,x_2,x_3,x_4\}$ with  $k(x_i,x_j)=0$, if $i$ and $j$ are both even  or both odd, and $k(x_1,x_2)=k(x_4,x_3)=a_+$, $k(x_2,x_1)=k(x_3,x_4)=a_-$, $k(x_1,x_4)=k(x_2,x_3)=b_+$ and $k(x_4,x_1)=k(x_3,x_2)=b_-$, with $a_+,a_-,b_+,b_- > 0$. Hence, the underlying graph to $L$ is a weighted $4$-cycle.

Let $f \in \R^X$, then we observe from \eqref{Psi2Formel}
\begin{align*}
2 \Psi_{2,\Upsilon}(f)(x_1) &= a_+ b_+ \Big(\Upsilon(f(x_3)-f(x_2))-\Upsilon'(f(x_2)-f(x_1)) (f(x_3) - f(x_2)) \\
&\quad+ \Upsilon(f(x_3)-f(x_4)) - \Upsilon'(f(x_4)-f(x_1)) (f(x_3)-f(x_4))\Big)\\
&\quad+ a_+ a_- \big( \Upsilon(f(x_1)-f(x_2))+ \Upsilon'(f(x_2)-f(x_1))(f(x_2)-f(x_1))\big)\\
&\quad + b_+ b_- \big( \Upsilon(f(x_1)-f(x_4))+ \Upsilon'(f(x_4)-f(x_1))(f(x_4)-f(x_1))\big)\\
&\quad + a_+^2 \Upsilon'(f(x_2)-f(x_1)) (f(x_2)-f(x_1)) + b_+^2 \Upsilon'(f(x_4)-f(x_1))(f(x_4)-f(x_1))\\
&\quad + a_+ b_+\Big( \Upsilon'(f(x_2)-f(x_1))(f(x_4)-f(x_1)) + \Upsilon'(f(x_4)-f(x_1))(f(x_2)-f(x_1))\Big) \\
&\quad - ( a_+ + b_+ ) \big( a_+ \Upsilon(f(x_2)-f(x_1)) + b_+ \Upsilon(f(x_4)-f(x_1))\big). 
\end{align*}
Now, we minimize this expression with respect to the value of $f(x_3)$. For that purpose, we consider the mapping $r\mapsto \Upsilon(r -c_1) - \Upsilon'(c_1-c_0)(r-c_1) + \Upsilon(r-c_2)-\Upsilon'(c_2-c_0)(r-c_2)$, with $c_0,c_1,c_2 \in \R$, which attains its minimum at $r=c_1+c_2-c_0$. Henceforth we will use the notation $t= f(x_2)-f(x_1)$ and $s= f(x_4)-f(x_1)$. We deduce
\begin{align*}
2 \Psi_{2,\Upsilon}(f)(x_1) &\geq a_+ \Big( \Upsilon'(t)t (a_- + a_+) +  \Upsilon(-t)a_- - \Upsilon(t)a_+ \Big) \\
&\quad+ b_+\Big(\Upsilon'(s)s (b_-+b_+) + \Upsilon(-s)b_- - \Upsilon(s)b_+ \Big)\\
&= a_+ a_-\, \nu_{1+\frac{a_+}{a_-},\frac{a_+}{a_-}}(t) + b_+ b_-\, \nu_{1+\frac{b_+}{b_-},\frac{b_+}{b_-}}(s).
\end{align*}
Due to the symmetric structure, we can perform analogous computations at $x_2,x_3$ and $x_4$. Then, we deduce that $CD_\Upsilon(\kappa,\infty)$ is satisfied for some $\kappa>0$ provided that
the function
$\nu_{\lambda_1,\lambda_2}$
is non negative 
for any pair of parameters
\begin{equation*}
(\lambda_1,\lambda_2)\in \Big\{ \big(1+\frac{a_+}{a_-}, \frac{a_+ +2\kappa}{a_-}\big),\big(1+\frac{a_-}{a_+}, \frac{a_- +2\kappa}{a_+}\big),\big(1+\frac{b_+}{b_-}, \frac{b_+ +2\kappa}{b_-}\big),\big(1+\frac{b_-}{b_+}, \frac{b_- +2\kappa}{b_+}\big)\Big\}.
\end{equation*}
We employ Lemma \ref{App_notlin}(ii) to conclude after a short computation that this is satisfied if 
\begin{equation*}
\kappa \leq \min \{ \sqrt{2 a_{\min}(a_{\min} + a_{\max})}, a_{\min} + a_{\max} \}= \sqrt{2 a_{\min}(a_{\min} + a_{\max})}
\end{equation*}
and 
\begin{equation*}
\kappa \leq \min \{ \sqrt{2 b_{\min}(b_{\min} + b_{\max})}, b_{\min} + b_{\max} \}= \sqrt{2 b_{\min}(b_{\min} + b_{\max})},
\end{equation*}
where $a_{\min}=\min\{a_-,a_+\}, a_{\max}= \max \{a_-,a_+\}, b_{\min}=\min \{b_-,b_+\}$ and $b_{\max}=\max \{b_-,b_+\}$.

Consequently, $CD_\Upsilon ( \min \big\{\sqrt{2 a_{\min}(a_{\min} + a_{\max})},\sqrt{2 b_{\min}(b_{\min} + b_{\max})} \big\} ,\infty) $ holds true. Note, that this generalizes the result on the (unweighted) $2$-dimensional hypercube, which satisfies the optimal bound $CD_\Upsilon(2,\infty)$ by Example \ref{ex:Hn}.
\end{example}
A class of examples having non-negative Bakry-\'Emery curvature  is given by Ricci-flat graphs, which were originally introduced in \cite{CY96}. We recall the definition for the reader's convenience.
\begin{defi}\label{defi:Ricciflat}
Let $G=(V,E)$ be an unweighted $d$-regular graph. We call $G$ Ricci-flat at $x\in V$ if there exist maps $\eta_i: B_1(x) \to V$ (where $B_1(x)$ is the closed ball with radius $1$ and center $x$ with respect to the combinatorical graph distance) for $1 \leq i \leq d$ satisfying the following properties:
\begin{itemize}
\item[(i)] $\eta_i(u)\in B_1(u) \setminus \{u\} $ for any $u \in B_1(x)$,
\item[(ii)] $\eta_i(u) \neq \eta_j(u)$, whenever $i \neq j$,
\item[(iii)] $\bigcup_j \eta_j(\eta_i(x)) = \bigcup_j \eta_i(\eta_j(x))$ for any $i \in \{1,...,d\}$.
\end{itemize} 
We say $G$ is Ricci-flat if $G$ is Ricci-flat at each $x \in V$.
\end{defi}
In \cite{Mn1} it was shown already that Ricci-flat graphs satisfy $CD \log (d,0)$ with some finite $d>0$. As a consequence, $CD_\Upsilon(0,\infty)$ holds, cf. Remark \ref{commentsCD}(i). See also \cite{DKZ}. Nevertheless we revisit this example by giving an argument that uses the convexity of the mapping $r \mapsto \Upsilon(r)$. This argument will also be crucial below to treat the case of (R)-Ricci-flat graphs.

\begin{example}[\textit{Ricci-flat graphs}]\label{ex:Ricci}
Let the transition rates are chosen in such a way that the underlying graph to $L$ is Ricci-flat with vertex set $X$. The invariant and reversible measure is given by the uniform measure on $X$.
Then we can write
\begin{align*}
2 \Psi_{2,\Upsilon} (f) (x) &= L \Psi_\Upsilon (f)(x) - B_{\Upsilon'}(f,Lf)(x) \\
&= \sum\limits_{i=1}^d \Big(\Psi_\Upsilon (f (\eta_i (x)) - \Psi_\Upsilon (f(x))) - \Upsilon'\big(f(\eta_i (x))-f(x)\big) \big(Lf(\eta_i(x)) - Lf (x)\big)\Big).
\end{align*}
Applying property (iii) from Definition \ref{defi:Ricciflat}, yields
\begin{align*}
&\sum\limits_{i=1}^d \Upsilon'\big(f(\eta_i (x))-f(x)\big) \big(Lf(\eta_i(x)) - Lf (x)\big) \\
&\qquad = \sum\limits_{j=1}^d\sum\limits_{i=1}^d \Upsilon'\big(f(\eta_i(x)) - f(x)\big) \big(f(\eta_j(\eta_i(x))) -f(\eta_i(x)) - (f(\eta_j(x)) - f(x))\big) \\
&\qquad = \sum\limits_{j=1}^d\sum\limits_{i=1}^d \Upsilon'\big(f(\eta_i(x)) - f(x)\big) \big(f(\eta_i(\eta_j(x))) -f(\eta_i(x)) - (f(\eta_j(x)) - f(x))\big).
\end{align*}
Hence, we can write
\begin{equation}\label{ricciflatgg2}
\begin{split}
2 \Psi_{2,\Upsilon} (f) (x) &= \sum\limits_{j=1}^d\sum\limits_{i=1}^d \Big(\Upsilon\big(f( \eta_j (\eta_i (x)))-f(\eta_i(x))\big) - \Upsilon\big(f(\eta_j (x))-f(x)\big)\\
&\qquad \qquad- \Upsilon'\big(f(\eta_j(x)) - f(x)\big) \big(f(\eta_j(\eta_i(x)))-f(\eta_i(x)) - ( f(\eta_j(x)) - f(x))\big)\Big).
\end{split} 
\end{equation}
By convexity of $r \mapsto \Upsilon(r)$ we have $\Upsilon(a) - \Upsilon(b) \geq \Upsilon'(b)(a-b)$, $a,b \in \R$. Setting $a=f(\eta_j(\eta_i(x)))-f(\eta_i(x))$ and $b=f(\eta_j(x)) - f(x)$, yields that $L$ satisfies $CD_\Upsilon(0,\infty)$. This is optimal in general, due to the existence of Ricci-flat graphs which do not allow for positive Bakry-\'Emery curvature, e.g. the discrete Laplacian on the lattice $\mathbb{Z}^d$, $d \in \N$.
\end{example}
The situation of the latter example improves if we restrict on a smaller class of graphs. Very recently the notion of (R)-Ricci flatness was introduced in \cite{CKKLP}. Again, we repeat the definition.
\begin{defi}
We call a graph $G=(V,E)$  (R)-Ricci-flat in $x\in V$ if it is Ricci-flat in $x \in V$ and there exist mappings $\eta_i$, $1\leq i\leq d$, satisfying the conditions from Definition \ref{defi:Ricciflat} and additionally the reflexivity condition $\eta_i(\eta_i(x))=x$ for any $i\in \{1,...,d\}$.
We say $G$ is (R)-Ricci-flat if $G$ is (R)-Ricci-flat at each $x \in V$.
\end{defi} 
The complete bipartite graphs $K_{d,d}$ or the Shrikhande graph are important examples of (R)-Ricci-flat graphs   (cf.\ \cite{CKKLP}). Furthermore, in \cite{CKKLP} it is shown that the Bakry-\'Emery condition $CD(2,\infty)$ holds given that the underlying graph is (R)-Ricci-flat. The following example shows an analogue behaviour regarding the $CD_\Upsilon$ condition.
\begin{example}[\textit{(R)-Ricci-flat graphs}]\label{ex:RRicci}
Let the transition rates be chosen in such a way that the underlying graph to $L$ is (R)-Ricci-flat with vertex set $X$. The invariant and reversible measure is given by the uniform measure on $X$.

The convexity argument from Example \ref{ex:Ricci} yields in particular that each summand in \eqref{ricciflatgg2} is non-negative. Applying the reflexivity condition leads to
\begin{align*}
2 \Psi_{2,\Upsilon}(f)(x) &\geq \sum\limits_{i=1}^d \Big(\Upsilon\big(f(x)-f(\eta_i(x))\big) - \Upsilon \big( f(\eta_i(x)) - f(x)\big) \\
&\qquad \qquad + 2\Upsilon'\big(f(\eta_i(x)) - f(x)\big) \big(f(\eta_i(x)) - f(x)\big) \Big)\\
&= \sum\limits_{i=1}^d \nu_{2,1}(f(\eta_i(x))-f(x)) \\
&=4 \Psi_\Upsilon(f)(x) + \sum\limits_{i=1}^d \nu_{2,5}(f(\eta_i(x)) - f(x)).
\end{align*}
We conclude from Lemma \ref{lem:nucdlemma}
that $L$ satisfies $CD_\Upsilon(2,\infty)$. This is optimal in general, since $CD(2,\infty)$ is optimal for any (unweighted) triangle free graph (see Theorem 1.2 in \cite{KRT}) and each $K_{d,d}$ is both, triangle free and (R)-Ricci-flat.
\end{example}
\begin{example}[\textit{Birth-death processes}]\label{ex:birthdeath}
Here, we consider a birth-death process  with state space given by $X=\{0,...,N\}$ for some $N \in \N$ or even the infinite case of $X= \N_0$. Following the notation of \cite{CaDP}, we introduce the functions $a,b: X \to [0,\infty)$ with $a(x)=k(x,x+1)$, $b(x)=k(x,x-1)$, $b(0)=0$, $b(x)>0$ otherwise, and $a(x)>0$ for any $x\in X$ if $X$ is infinite and for any $x \in X \setminus \{N\}$ if $X$ is finite, in which case $a(N)=0$. Moreover, we set  $k(x,y)=0$ whenever $|x-y| >1$. The detailed balance condition now reads as
\begin{equation}\label{dbforbd}
a(x)\pi(x)= b(x+1)\pi(x+1)
\end{equation}
for any $x \in X$ in the infinite case respectively  for any $x \in X \setminus \{N\}$ if $X$ is of finite size. Note that in the infinite case, the measure $\mu$ given by $d\mu=\pi d\#$ is a finite measure if and only if
\begin{equation*}
\sum\limits_{x=1}^\infty \frac{a(x-1)\cdot \cdot \cdot a(0)}{b(x)\cdot \cdot \cdot b(1)} < \infty.
\end{equation*}
The method of \cite{CaDP} applies quite successfully  to birth-death processes if one assumes monotonicity of the rates in the sense that
\begin{equation}\label{monoticityofratesBD}
a(x)\leq a(x-1), \, b(x-1)\leq b(x),
\end{equation}
for any $x \in X \setminus \{0\}$ and, additionally,
the bound
\begin{equation}\label{CPPassumption}
a(x-1) - a(x) + b(x)-b(x-1) \geq \kappa, \quad \forall x \in X \setminus \{0\},
\end{equation}
for some $\kappa>0$. Due to \cite{CaDP} this suffices to deduce the modified logarithmic Sobolev inequalitiy with constant $\frac{\kappa}{2}$. For our purposes it will be crucial to assume that the condition
\begin{equation}\label{BirthDeathCDPalmecondition}
\sqrt{2 \min \{ a(x-1)-a(x),b(x)-b(x-1)\} \big(a(x-1)-a(x) +b(x)-b(x-1)\big)} \geq \kappa
\end{equation}
holds true for some $\kappa>0$ and any $x \in X \setminus \{0\}$.
Note that the condition \eqref{BirthDeathCDPalmecondition} is stronger than \eqref{CPPassumption}. This is not surprising as the $CD_\Upsilon(\kappa,\infty)$ condition (for which \eqref{BirthDeathCDPalmecondition} is suitable) yields pointwise estimates with regard to the entropy method (cf. Section \ref{sec:entrodecayandmLSI}) and hence is a stronger assumption compared to the non-local approach of \cite{CaDP}.

In the subsequent lines we will chronologically show the following assertions:
\begin{itemize}
\item[(i)] Assuming \eqref{monoticityofratesBD} and \eqref{CPPassumption}, $CD(\frac{\kappa}{2},\infty)$ holds.
\item[(ii)] The condition \eqref{monoticityofratesBD} implies $CD_\Upsilon(0,\infty)$.
\item[(iii)] $CD_\Upsilon(0,\infty)$ is best possible for the Poisson case on $\N_0$, even though \eqref{monoticityofratesBD} and \eqref{CPPassumption} are satisfied.
\item[(iv)] If strict monotonicity holds in \eqref{monoticityofratesBD}, for both $a$ and $b$, and \eqref{BirthDeathCDPalmecondition} is valid, then $CD_\Upsilon(\kappa,\infty)$ holds true. 
\end{itemize}
We emphasize that (iv) combined with Corollary \ref{mLSIoutofCDPalme} yields $\textrm{MLSI}(\kappa)$ while the method of \cite{CaDP} (assuming \eqref{monoticityofratesBD} and \eqref{CPPassumption}) implies the modified logarithmic Sobolev inequality with constant $\frac{\kappa}{2}$. For instance, if the left-hand sides of \eqref{CPPassumption} and \eqref{BirthDeathCDPalmecondition} coincide, which can only happen in the finite state space case, our method thus improves the corresponding constant by a factor $2$.

Now, we prove (i). We have by \eqref{Psi2Formel} 
\begin{align*}
&2 \Gamma_2(f)(x)=  L\Gamma(f)(x) - 2 \Gamma(f,Lf)(x)\\
 &\qquad  =\frac{1}{2} \Big[ a(x)a(x+1) \Big((f(x+2)-f(x+1))^2 - 2(f(x+1)-f(x)) (f(x+2) - f(x+1))\Big) \\
&\qquad\quad + a(x)b(x+1) \Big( (f(x)-f(x+1))^2 - 2(f(x+1)-f(x)) (f(x)-f(x+1)) \Big)\\
&\qquad\quad + b(x) a(x-1) \Big((f(x)-f(x-1))^2 - 2(f(x-1)-f(x))(f(x)-f(x-1))\Big) \\
&\qquad\quad + b(x)b(x-1) \Big( (f(x-2)-f(x-1))^2- 2(f(x-1)-f(x))(f(x-2)-f(x-1))\Big) \\
&\qquad\quad + 2 \Big( a(x) (f(x+1)-f(x)) + b(x) (f(x-1)-f(x))\Big)^2 \Big]\\
&\qquad\quad - (a(x)+b(x)) \Gamma(f)(x).
\end{align*}
Here and in the sequel, we set $a(x-1)=b(x-1)=0$ if $x=0$ and in the case of finite $X$, $a(x+1)=b(x+1)=0$ for $x=N$. Note that in those cases the respective choice of extension for $f \in \R^X$ does not play a role.

Viewing the above expression as a function of $f(x+2)$ and $f(x-2)$, it can be minimized by setting $f(x+2)-f(x+1)=f(x+1)-f(x)$ and $f(x-2)-f(x-1)=f(x-1)-f(x)$. Henceforth, we will use the notation  $t= f(x+1)-f(x)$ and $s=f(x-1)-f(x)$. We have
\begin{align*}
4 \Gamma_2(f)(x) &\geq a(x) t^2 \big( 3 b(x+1) - a(x+1)\big) + b(x) s^2 \big( 3 a(x-1) - b(x-1)\big) + 2 \big( a(x) t + b(x) s \big)^2\\
&\quad
- (a(x)+b(x)) \big( a(x)t^2 + b(x) s^2 \big) \\
&= a(x)t^2 \big( b(x+1) -b(x)+a(x)-a(x+1)\big) \\
&\quad + b(x)s^2 \big(  a(x-1) -a(x) + b(x) - b(x-1) \big) \\
&\quad + 2 \big(b(x+1)a(x)t^2 + a(x-1)b(x) s^2 + 2 a(x)b(x)ts \big)  \\
&\geq \kappa \big(a(x) t^2 +  b(x) s^2 \big) + 2 a(x)b(x) (s+t)^2 \\
&\geq 2\kappa \Gamma(f)(x),
\end{align*}
where we have applied \eqref{monoticityofratesBD} and \eqref{CPPassumption}. Hence, $CD(\frac{\kappa}{2},\infty)$ holds true.

 Regarding (ii), we will follow a similar approach. In this case, we observe by \eqref{Psi2Formel}
\begin{align*}
&2 \Psi_{2,\Upsilon}(f)(x)= L \Psi_\Upsilon (f)(x) - B_{\Upsilon'}(f,Lf)(x)\\
 &\qquad = a(x)a(x+1) \Big(\Upsilon(f(x+2)-f(x+1)) - \Upsilon'(f(x+1)-f(x)) (f(x+2) - f(x+1))\Big) \\
&\qquad\quad + a(x)b(x+1) \Big( \Upsilon (f(x)-f(x+1)) - \Upsilon'(f(x+1)-f(x)) (f(x)-f(x+1)) \Big)\\
&\qquad\quad + b(x) a(x-1) \Big( \Upsilon(f(x)-f(x-1)) - \Upsilon'(f(x-1)-f(x))(f(x)-f(x-1))\Big) \\
&\qquad\quad + b(x)b(x-1) \Big( \Upsilon(f(x-2)-f(x-1))- \Upsilon'(f(x-1)-f(x))(f(x-2)-f(x-1))\Big) \\
&\qquad\quad + \big( a(x) \Upsilon'(f(x+1)-f(x)) + b(x) \Upsilon'(f(x-1)-f(x)\big) L f(x) \\
&\qquad\quad - (a(x)+b(x)) \Psi_\Upsilon(f)(x).
\end{align*}
Again, we  minimize the above expression by choosing the values of $f(x+2)$ and $f(x-2)$ accordingly. To that aim, it can be easily checked that the mapping $r \mapsto \Upsilon(r-c_1)- \Upsilon'(c_1-c_2)(r-c_1) $ with parameters $c_1,c_2 \in \R$ achieves the global minimum for $r-c_1=c_1-c_2$. Recalling the notation $t= f(x+1)-f(x)$ and $s=f(x-1)-f(x)$, we obtain
\begin{align*}
2 \Psi_{2,\Upsilon}(f)(x) &\geq a(x)a(x+1)\big( \Upsilon(t) - \Upsilon'(t)t \big) + a(x)b(x+1) \big( \Upsilon(-t) + \Upsilon'(t)t \big)\\
&\; +b(x)a(x-1) \big( \Upsilon(-s) + \Upsilon'(s)s \big) + b(x)b(x-1) \big(\Upsilon(s)- \Upsilon'(s)s\big) \\
&\; + \big( (a(x) \Upsilon'(t)+b(x)\Upsilon'(s) \big)\big( a(x)t+b(x)s\big) - (a(x)+b(x))\big(a(x)\Upsilon(t)+b(x)\Upsilon(s)\big)\\
&= a(x) \big[\Upsilon(t) (a(x+1)-a(x)-b(x)) + \Upsilon(-t) b(x+1) \\
&\qquad\qquad+ \Upsilon'(t)t(b(x+1)+a(x)-a(x+1))\big]\\
&\; + b(x) \big[ \Upsilon(s) (b(x-1)-b(x)-a(x)) + \Upsilon(-s)a(x-1) \\
&\qquad\qquad + \Upsilon'(s)s ( a(x-1)+b(x)-b(x-1))\big]\\
&\; + a(x)b(x) \big( \Upsilon'(t)s + \Upsilon'(s)t \big). 
\end{align*}
The last summand, which is of mixed type with respect to $s$ and $t$, can be treated  by the inequality $e^\beta \beta - e^\beta \alpha - e^\beta + e^\alpha \geq 0$, $\alpha,\beta \in \R$, which follows from convexity of the exponential function. In fact, we conclude
\begin{equation}\label{mixedterm_birthdeath}
\begin{split}
\Upsilon'(t)s + \Upsilon'(s)t = e^t s - s + e^s t -t &\geq e^t-e^t t -e^{-t} - t + e^s - e^s s - e^{-s} - s \\
&= \Upsilon(t) - \Upsilon(-t) - \Upsilon'(t)t + \Upsilon(s) - \Upsilon(-s) - \Upsilon'(s)s.
\end{split}
\end{equation}
The major advantage of the right-hand side of \eqref{mixedterm_birthdeath} is that no mixed terms occur anymore, which allows to consider the terms involving $s$ resp. $t$ separately.
We have 
\begin{align*}
&\Upsilon(t) \big(a(x+1)-a(x)-b(x)\big) + \Upsilon(-t) b(x+1) + \Upsilon'(t)t \big(b(x+1)+a(x)-a(x+1)\big)\\
&\qquad\quad + b(x) \big( \Upsilon(t) - \Upsilon(-t) - \Upsilon'(t)t \big)\\
&\qquad = b_{x,+} \big( \Upsilon'(t)t + \Upsilon(-t) \big) + a_{x,+}  \big( \Upsilon'(t)t - \Upsilon(t) \big),
\end{align*}
where $a_{x,+}= a(x)-a(x+1)$ and $b_{x,+}=b(x+1)-b(x)$ and analogously 
\begin{align*}
&\Upsilon(s) \big(b(x-1)-b(x)-a(x)\big) + \Upsilon(-s)a(x-1) + \Upsilon'(s)s \big( a(x-1)+b(x)-b(x-1)\big) \\
&\qquad\quad + a(x) \big( \Upsilon(s) - \Upsilon(-s) - \Upsilon'(s)s \big) \\
&\qquad =  b_{x,-} \big( \Upsilon'(s)s - \Upsilon(s)\big) + a_{x,-} \big( \Upsilon'(s)s + \Upsilon(-s)\big),
\end{align*} 
where $a_{x,-}=a(x-1)-a(x)$ and $b_{x,-}=b(x)-b(x-1)$.

In total, we thus have
\begin{equation}\label{FormulaBirthDeathPsi2}
\begin{split}
2 \Psi_{2,\Upsilon}(f)(x) &\geq a(x) \big( b_{x,+} ( \Upsilon'(t)t + \Upsilon(-t)) + a_{x,+} (\Upsilon'(t)t - \Upsilon(t)) \big) \\
&\;+ b(x) \big(b_{x,-}(\Upsilon'(s)s - \Upsilon(s)) + a_{x,-} ( \Upsilon'(s)s + \Upsilon(-s)) \big).
\end{split}
\end{equation}
Whenever $b_{x,-},a_{x,-},b_{x,+},a_{x,+} \geq 0$ we deduce that $CD_\Upsilon(0,\infty)$ is valid in $x $. Hence, the global monotonicity condition \eqref{monoticityofratesBD}  implies  $CD_\Upsilon(0,\infty)$, which shows (ii).

The Poisson case on $\N_0$ shows that this bound is in general sharp, even if the condition \eqref{CPPassumption} is additionally assumed. Indeed, let 
$a(x)=\lambda$, $\lambda>0$, and $b(x)=x$ for any $x \in \N_0$. Then the corresponding Markov chain coincides with the Poisson case of the birth-death process on $\N_0$ with invariant probability measure given by $d\mu=\pi_\lambda d\#$ and $\pi_\lambda(x) = \frac{\lambda^x}{x!}e^{-\lambda}$, $x \in \N_0$. The convexity inequality \eqref{mixedterm_birthdeath} becomes an equality if $s=-t$. For $n \geq 2$ and $\tau>0$, we define the functions  $f_{n,\tau}: \N_0 \to \R$ by $f_{n,\tau}(x)=\tau (x-n)$ for $x \in \{n-2,n-1,n+1,n+2\}$ and $0$ otherwise.
For this particular choice we have equality in \eqref{FormulaBirthDeathPsi2} at $x=n$. More precisely, 
\begin{equation*}
2 \Psi_{2,\Upsilon}(f_{n,\tau})(n) = \lambda \big( \Upsilon'(\tau)\tau + \Upsilon(-\tau)\big) + n \big( \Upsilon'(-\tau) (-\tau)- \Upsilon(-\tau)\big).
\end{equation*}
Consequently, for $\kappa>0$ the estimate $2\Psi_{2,\Upsilon}(f_{n,\tau})(n) \geq \kappa \Psi_{\Upsilon}(f_{n,\tau})(n)$ is equivalent to
\begin{equation}\label{PoissonPsi2}
\lambda \big( \Upsilon'(\tau)\tau + \Upsilon(-\tau) - \kappa \Upsilon(\tau)\big) + n \big( \Upsilon'(-\tau)(-\tau) - (1+\kappa) \Upsilon(-\tau) \big) \geq 0.
\end{equation}
Choosing $\tau>0$ large enough such that $\Upsilon'(-\tau)(-\tau) - (1+\kappa) \Upsilon(-\tau)<0$ and sending $n\to \infty$ yields a contradiction to \eqref{PoissonPsi2}. This shows assertion (iii). Interestingly, since here the assumptions  \eqref{monoticityofratesBD} and \eqref{CPPassumption} are satisfied, this gives an example of a Markov chain satisfying $CD(\frac{1}{2},\infty)$ by assertion (i), where a   positive $CD_\Upsilon$ curvature condition is impossible to hold.

Next we aim for positive $CD_\Upsilon$ curvature bounds. We assume that $b_{x,+},a_{x,-} > 0$ for any $x \in X$. Note that this does not violate the conditions on the boundary, as we have $b_{0,-}=0$ and, in the finite case, $a_{N,+}=0$. Recall the inequality \eqref{FormulaBirthDeathPsi2}.  The estimate
\begin{align*}
a(x) \big( b_{x,+} ( \Upsilon'(t)t + \Upsilon(-t)) + a_{x,+} (\Upsilon'(t)t - \Upsilon(t)) \big) \geq 2 \kappa a(x) \Upsilon(t)  , \quad
 t \in \R,
\end{align*}
(which is clear for $x=N$ in the finite case) is equivalent for $x \in X$ resp. $x \in X \setminus \{N\}$ in the finite case to 
\begin{equation}\label{xpluscondition}
b_{x,+} \nu_{1+\frac{a_{x,+}}{b_{x,+}},\frac{a_{x,+}}{b_{x,+}}+\frac{2\kappa}{b_{x,+}}}(t) \geq 0 , \quad
 t \in \R,
\end{equation} 
and analogously
\begin{align*}
b(x) \big(b_{x,-}(\Upsilon'(s)s - \Upsilon(s)) + a_{x,-} ( \Upsilon'(s)s + \Upsilon(-s)) \big) \geq 2\kappa b(x) \Upsilon(s) , \quad s \in \R,
\end{align*} 
(which is clear for $x=0$) is equivalent for $x \in X \setminus \{0\}$ to 
\begin{equation}\label{xminuscondition}
a_{x,-} \nu_{1+\frac{b_{x,-}}{a_{x,-}},\frac{b_{x,-}}{a_{x,-}}+\frac{2\kappa}{a_{x,-}}}(s) \geq 0 , \quad
 s \in \R.
\end{equation} 
We want to employ Lemma \ref{App_notlin}(ii) (with $\lambda=1 + \frac{a_{x,+}}{b_{x,+}}$ for \eqref{xpluscondition} and $\lambda= 1+\frac{b_{x,-}}{a_{x,-}}$ for \eqref{xminuscondition}). A short computation (as in Example \ref{ex:weighted4cyc}) shows that if
\begin{equation}\label{BirthDeathconditionwithxminus}
\kappa \leq \sqrt{2 \min \{ a_{x,-},b_{x,-}\} (a_{x,-}+b_{x,-})}
\end{equation}
holds for $x \in X \setminus \{0\}$ then \eqref{xminuscondition} is satisfied. But since $a_{x-1,+}=a_{x,-}$ and $b_{x-1,+}=b_{x,-}$, \eqref{xpluscondition} also follows from \eqref{BirthDeathconditionwithxminus}. Clearly, \eqref{BirthDeathCDPalmecondition} and \eqref{BirthDeathconditionwithxminus} are the same and hence (iv) is established.

We conclude this example with a remark on the infinite state space case. Note that \eqref{monoticityofratesBD} implies $\inf_{x \in \N} a_{x,-}=0$. Hence, in order to achieve a global positive curvature bound by \eqref{BirthDeathCDPalmecondition}, we need to choose the sequence $(b_{x,-})_{x \in \N}$ such that $\inf_{x \in \N} a_{x,-}b_{x,-} \geq c$ for some  constant $c>0$. This implies in particular  that $b(x)\to \infty$, and hence $M_1(x) \to \infty$,  as $x \to \infty$. Concerning the modified logarithmic Sobolev inequality  the assumptions $M_1 \in \ell^2(\mu)$ and $M_2 \in \ell^1(\mu)$, which were imposed in Theorem \ref{ThmSecTDEnt} and Corollary \ref{mLSIoutofCDPalme},  boil down to assume that $b \in \ell^2(\mu)$. This is immediate for $M_1 \in \ell^2(\mu)$ and can be easily deduced from the detailed balance condition \eqref{dbforbd} for $M_2 \in \ell^1(\mu)$. The ratio test implies that $b \in \ell^2(\mu)$ is satisfied if there exists some constant $q \in (0,1)$ such that
\begin{equation}\label{ratiotestbirthdeath}
\frac{a(x)b(x+1)}{b(x)^2}\leq q < 1,
\end{equation}
for all sufficiently large $x$, where again \eqref{dbforbd} has been used. Finally, we illustrate the infinite state space case by the following concrete example. We choose the transition rates as
\begin{equation*}
a(x)= \sum_{n=x+1}^\infty n^{-\alpha} , \, b(x) = \sum_{n=0}^x n^\beta
\end{equation*} 
for any $x \in \N_0$ and parameters $\beta \geq \alpha >1$. Then \eqref{ratiotestbirthdeath} is apparently satisfied. Moreover, we have $a(x-1)-a(x)= x^{-\alpha}$ and $b(x)-b(x-1)= x^\beta$ for any $x \in \N$ and thus \eqref{BirthDeathCDPalmecondition} turns to
\begin{equation*}
\sqrt{2 (x^{-2\alpha} + x^{\beta - \alpha})}\geq \kappa
\end{equation*} 
for any $x \in \N$. Clearly, $\kappa$ can at least be chosen as $\sqrt{2}$ and hence $\mathrm{MLSI}(\sqrt{2})$ holds. On the other hand, the method of \cite{CaDP} yields at most a modified logarithmic Sobolev inequality with constant $1$.

\end{example}
\begin{example}[\textit{Weighted graphs with large girth}]\label{ex:nolowerbound}
Now we consider the situation, where the underlying graph has girth at least $5$, i.e. if there exist a subset $\{x_1,...,x_g\}\subset X$, of pairwise distinct elements, with $k(x_i,x_{i+1})>0$ for any $i \in \{1,...,g-1\}$ and $k(x_g,x_1)>0$, then $g \geq 5$ must hold. Further,  we assume the existence of $x,y \in X$ with $k(x,y)>0$ (and hence $k(y,x)>0$) and such that 
\begin{equation}\label{minusinftycondition}
M_1(x)+M_1(y)-2 \big( k(x,y)+k(y,x)\big) > 0
\end{equation}
holds true. We claim that this implies that $CD_\Upsilon(\kappa,\infty)$ fails in $x$ (and by symmetry also in $y$) for all $\kappa \in \R$ .

We remark that any (connected) unweighted graph (here $M_1(x)$ and $M_1(y)$  equal the degree of $x$ and $y$ respectively)  with $|X|\geq 5$ and girth at least $5$ (in particular trees) that is neither a path nor a cycle satisfies \eqref{minusinftycondition} for some $x,y \in X$. Indeed, choose $x \in X$ such that $x$ has at least $3$ neighbors (which is possible since the underlying graph is neither a path nor a cycle) and $y$ a neighbor of $x$ that has at least one neighbor distinct from $x$ (which exists as $|X| \geq 5$). Then \eqref{minusinftycondition} holds for $x$ and $y$.

For the discrete Laplacian on $\Z$ we have equality in \eqref{minusinftycondition} at any $x \in \Z$ and $y= x \pm 1$, as well as $CD_\Upsilon(0,\infty)$ holds true by Example \ref{ex:Ricci}. In this sense, the condition \eqref{minusinftycondition} is sharp.

Note that \eqref{minusinftycondition} also applies to birth-death processes, which have been discussed in Example \ref{ex:birthdeath}. In fact, using the notation of Example \ref{ex:birthdeath}, \eqref{minusinftycondition} translates to 
\begin{equation*}
a(x)+b(x) + a(x \pm 1) + b(x \pm 1) - 2 \big( k(x, x \pm 1) + k(x \pm 1, x) \big) > 0
\end{equation*}
and hence to the condition that $ 0 > a_{x,+}+ b_{x,+}$ resp. $0 > a_{x,-}+b_{x,-}$ holds true for some $x \in X$.

Let us now prove the claim. For $z \in  X$, we denote by $S_1(z)= \{ v \in X \setminus \{z\} : \, k(z,v)>0\}$ the neighborhood of $z$. Let $x,y \in X$ be given as in \eqref{minusinftycondition}, $S_1(x)=\{y_1,...,y_m\}$, $m \geq 1$, and let w.l.o.g. $y_1=y$. We remark that the arguments below also apply to the case that $S_1(x)$ consists of infinitely many elements, as well as to the special case of $m=1$ (where the situation is easier as the corresponding sums over the index set $\{2,...,m \}$ are empty). Since the girth of the graph is at least $5$, we have $y_j \notin S_1(y_i)$ and $S_1(y_i) \cap S_1(y_j) = \{x\}$ if $i \neq j$. Thus, we can choose for some $\tau \in \R$ a function $f_\tau \in \R^X$ such that $\tau= f_\tau(y_i) - f_\tau(x)$ for any $i \in \{2,...,m\}$, $f_\tau(y_1)-f_\tau(x)=-\tau$ and  $f_\tau(z) - f_\tau(y_j)=f_\tau(y_j)-f_\tau(x)$ for any $z \in S_1(y_j)\setminus \{x\}$, $j \in \{1,...,m\}$. For this particular choice, \eqref{Psi2Formel} turns to
\begin{align*}
2 \Psi_{2,\Upsilon}&(f_\tau)(x) \\&= \sum_{i=1}^m k(x,y_i) \big( M_1(y_i) - k(y_i,x)\big) \big(\Upsilon(f_\tau(y_i)-f_\tau(x)) - \Upsilon'(f_\tau(y_i) - f_\tau(x))(f_\tau(y_i)-f_\tau(x)) \big) \\
&\quad + \sum_{i=1}^m k(x,y_i)k(y_i,x) \big( \Upsilon(f_\tau(x)-f_\tau(y_i)) +\Upsilon'(f_\tau(y_i)-f_\tau(x))(f_\tau(y_i)-f_\tau(x))\big) \\
&\quad + \sum_{i=1}^m k(x,y_i)^2 \Upsilon'(f_\tau(y_i)-f_\tau(x))(f_\tau(y_i)-f_\tau(x)) - M_1(x) \sum_{i=1}^m k(x,y_i)\Upsilon(f_\tau(y_i)-f_\tau(x))\\
&\quad + \sum_{\substack{i,j \in \{1,...,m\} \\ i \neq j}}k(x,y_i)k(x,y_j) \big(\Upsilon'(f_\tau(y_i)-f_\tau(x))(f_\tau(y_j)-f_\tau(x))\big).
\end{align*}
Further, we have
\begin{align*}
&\sum_{\substack{i,j \in \{1,...,m\} \\ i \neq j}}k(x,y_i)k(x,y_j) \big(\Upsilon'(f_\tau(y_i)-f_\tau(x))(f_\tau(y_j)-f_\tau(x))\big) \\
&\qquad= \sum_{j=2}^m k(x,y_1)k(x,y_j)\big( \Upsilon'(-\tau)\tau- \Upsilon'(\tau)\tau \big) +  \sum_{\substack{i,j \in \{2,...,m\} \\ i \neq j}}k(x,y_i)k(x,y_j)\Upsilon'(\tau)\tau  \\
&\qquad= k(x,y_1)\Upsilon'(-\tau)\tau \big( M_1(x)- k(x,y_1)\big)+ \sum_{i=2}^m k(x,y_i) \Upsilon'(\tau)\tau \big(M_1(x) - k(x,y_i) - 2k(x,y_1)\big).
\end{align*}
Combininig the previous identities yields
\begin{align*}
&2\Psi_{2,\Upsilon}(f_\tau)(x) \\&= \sum_{i=2}^m k(x,y_i)\Big[ \Upsilon'(\tau)\tau \big( M_1(x)-M_1(y_i) + 2 ( k(y_i,x)-k(x,y_1))\big)\\
&\qquad\qquad\qquad  + \Upsilon(-\tau)k(y_i,x) + \Upsilon(\tau) \big(M_1(y_i)-M_1(x)- k(y_i,x)\big)\Big]\\
&\quad + \Upsilon'(-\tau)\tau  k(x,y_1) \big( M_1(x)+M_1(y_1) - 2(k(x,y_1)+k(y_1,x))\big) + \Upsilon(\tau)k(x,y_1)k(y_1,x)\\
&\quad+ k(x,y_1) \Upsilon(-\tau)\big( M_1(y_1)-M_1(x) - k(y_1,x)\big)\\
&= C_{1,x}\Upsilon'(\tau)\tau + C_{2,x}\Upsilon(-\tau) + C_{3,x} \Upsilon(\tau)  \\
&\quad\quad + k(x,y_1)\Upsilon'(-\tau)\tau \big( M_1(x)+M_1(y_1) - 2(k(x,y_1)+k(y_1,x))\big),
\end{align*}
where $C_{1,x},C_{2,x},C_{3,x} \in \R$ are constants that depend on $x$ and the transition rates, but not  on $\tau$. Now, we observe
\begin{equation}\label{noCDforweightedgirth}
\begin{split}
&2 \big(\Psi_{2,\Upsilon}(f_\tau)(x) - \kappa \Psi_\Upsilon(f_\tau)(x)\big) \\&\quad=  C_{1,x}\Upsilon'(\tau)\tau + \big(C_{2,x}-2\kappa k(x,y_1)\big)\Upsilon(-\tau) + \big(C_{3,x}-2\kappa(M_1(x)-k(x,y_1))\big) \Upsilon(\tau) \\& \qquad+ k(x,y_1)\Upsilon'(-\tau)\tau \big( M_1(x)+M_1(y_1) - 2(k(x,y_1)+k(y_1,x))\big),
\end{split}
\end{equation} 
where $\kappa \in \R$. Sending $\tau \to -\infty$, the dominating term in the right hand side of \eqref{noCDforweightedgirth} is given by $\Upsilon'(-\tau)\tau$, which tends to $-\infty$. By \eqref{minusinftycondition}, we thus can find for any $\kappa \in \R$ some $\tau<0$ such that \eqref{noCDforweightedgirth} is negative. Hence, $CD_\Upsilon(\kappa,\infty)$ is impossible to hold in $x$ for any $\kappa \in \R$.
\end{example}
The latter example has the very interesting consequence that the $CD_\Upsilon$ condition is not robust against small perturbations, which will be demonstrated in the following example.
\begin{example}[\textit{Perturbed birth-death process}]\label{ex:pertubedbirthdeath}
Let $L$ be the generator of a birth-death process on  the state space $X=\{0,...,N\}$ with $N\geq 4$, resp. $X=\N_0$, with invariant and reversible measure $\mu$ given by $d\mu= \pi d \#$. The transitions rates $\big(k(x,y)\big)_{x,y \in X}$ are given as explained in Example \ref{ex:birthdeath} and the notation is taken from there. Now, we add an edge to the underlying graph. Precisely, let $x_0,y_0 \in X$ with $y_0 \geq x_0+4$ be fixed and let $\lambda_{x_0,y_0} = \frac{\pi(x_0)}{\pi(y_0)}$. Let $L_\varepsilon$ denote the Markov generator whose transition rates are given by $k_\varepsilon(x,y)=k(x,y)$, $x \neq y$, if $(x,y) \notin \{(x_0,y_0),(y_0,x_0) \}$, $k_\varepsilon(x_0,y_0)=\varepsilon$ and $k_\varepsilon(y_0,x_0)=\lambda_{x_0,y_0}\varepsilon$. Note that $\mu$ is also reversible (and hence invariant) for $L_\varepsilon$ since
\begin{equation*}
k_\varepsilon(y_0,x_0)\pi(y_0)= \varepsilon \lambda_{x_0,y_0} \pi(y_0) = k_\varepsilon(x_0,y_0) \pi(x_0)
\end{equation*}
holds. Besides that we obtain
\begin{align*}
M_1(x_0) + M_1(y_0) - 2\big(k_\varepsilon(x_0,y_0&) + k_\varepsilon(y_0,x_0)\big) \\
&= a(x_0)+b(x_0) + \varepsilon + a(y_0) + b(y_0) + \lambda_{x_0,y_0} \varepsilon - 2\big(\varepsilon + \varepsilon \lambda_{x_0,y_0} \big)\\
&= a(x_0)+b(x_0)+a(y_0)+b(y_0) - \varepsilon \big( 1 + \lambda_{x_0,y_0}\big)
\end{align*}
and consequently there exists some $\varepsilon_0=\varepsilon_0(x_0,y_0)>0$ such that condition \eqref{minusinftycondition} is satisfied for $x_0$ and $y_0$ and any $\varepsilon \in (0,\varepsilon_0]$. In particular, we can choose for given $\kappa_0>0$ the transition rates $a$ and $b$ such that $L$ satisfies $CD_\Upsilon(\kappa_0,\infty)$ (see Example \ref{ex:birthdeath}) while $L_\varepsilon$ fails to satisfy $CD_\Upsilon(\kappa,\infty)$ for any $\kappa \in \R$ and $\varepsilon>0$ small enough by Example \ref{ex:nolowerbound}.
\end{example}
In view of Example \ref{ex:nolowerbound} the natural question arises, whether there exists a Markov generator $L$ that does not satisfy $CD_\Upsilon(0,\infty)$ but $CD_\Upsilon(-\kappa,\infty)$ for some $\kappa>0$. We answer this question positively by the following example.
\begin{example}[\textit{The unweighted 3-star}]\label{ex:3star}
We consider $X=\{x_*,a_1,a_2,a_3\}$ with $k(x_*,a_i)=k(a_i,x_*)=1$ for any $i\in \{1,2,3\}$ and $k(a_i,a_j)=0$ if $i \neq j$. The underlying graph is given by the unweighted $3$-star. Note that this is the only unweighted graph with girth at least $5$ that is neither a path nor a cycle to which \eqref{minusinftycondition} does not apply. Let $f \in \R^X$ be arbitrary chosen. We set $r= f(a_1)-f(x_*)$, $s=f(a_2)-f(x_*)$ and $t= f(a_3)-f(x_*)$. By \eqref{Psi2Formel}, we have for $\kappa \geq 0$
\begin{align}
\nonumber 2 &\big( \Psi_{2,\Upsilon}(f)(x_*) + \kappa \Psi_\Upsilon(f)(x_*)\big) \\ \nonumber
&= \sum_{i=1}^3 \Upsilon'(f(a_i)-f(x_*))(f(a_i)-f(x_*)) + \Upsilon(f(x_*)-f(a_i)) - (3-2\kappa) \Upsilon(f(a_i)-f(x_*))\\
\nonumber & \quad + \sum_{i=1}^3 \Upsilon'(f(a_i)-f(x_*)) \sum_{j=1}^3 (f(a_j)-f(x_*))\\
&=  \nu_{1,3-2\kappa}(r) + \nu_{1,3-2\kappa}(s) + \nu_{1,3-2\kappa}(t) + (r+s+t) (e^r + e^s + e^t -3)\label{3starnuforumla}.
\end{align}
First we aim to show that $CD_\Upsilon(0,\infty)$ is not satisfied in $x_*$. Having the approach of Example \ref{ex:nolowerbound} in mind, we choose $s=t$ and $r=-t$. Then \eqref{3starnuforumla} (with $\kappa=0$) reads as
\begin{align*}
2 \Psi_{2,\Upsilon}(f)(x_*) &= \nu_{1,3}(-t) + 2\nu_{1,3}(t) + t (e^{-t} + 2 e^t - 3) \\
&= 4 \Upsilon'(t)t - 5 \Upsilon(t) - \Upsilon(-t).
\end{align*}
Thus, choosing $t<0$ with $|t|$ sufficiently large shows that the latter becomes negative and consequently $CD_\Upsilon(0,\infty)$ fails in $x_*$.

Next, we aim to show the existence of $\kappa>0$ such that $CD_\Upsilon(-\kappa,\infty)$ holds. We assume in the sequel that $\kappa \geq \frac{3}{2}$. In particular,  $\nu_{1,3-2\kappa}$ is non-negative. By \eqref{3starnuforumla} there is nothing to show if  $r,s$ and $t$ all have the same sign. In case that $r+s+t \geq 0$ it follows from Jensen's inequality that
\begin{align*}
e^r + e^s + e^t \geq 3e^{\frac{1}{3}(r+s+t)} \geq 3.
\end{align*}
Hence, we are left with the case of $r+s+t<0$, which will be assumed from now on. We write
\begin{equation}\label{formula3starPsi2}
\begin{split}
2 \big( \Psi_{2,\Upsilon}(f)(x_*) + \kappa \Psi_\Upsilon(f)(x_*)\big) &= (2\kappa - 3) \big( \Upsilon(r) + \Upsilon(s) + \Upsilon(t) \big) + \Upsilon(-r) + \Upsilon(-s) + \Upsilon(-t)\\
&\quad + \Upsilon'(r)(s+t+2r) + \Upsilon'(s)(r+t+2s) + \Upsilon'(t)(r+s+2t).
\end{split} 
\end{equation}
We infer from \eqref{formula3starPsi2} that for any $R>0$ there exists some $\kappa_R \geq \frac{3}{2}$ such that the expression $\Psi_{2,\Upsilon}(f)(x_*)+\kappa \Psi_\Upsilon(f)(x_*)$ is non-negative for any $\kappa\geq \kappa_R$ provided that $(r,s,t) \in B_R(0)$ (where $B_R(0)$ denotes the  ball with center $(0,0,0)$ and radius $R$ with respect to the euclidean distance in $\R^3$).

\textit{1. case:} One element of $r,s,t$ is positive while the other two are non-positive (w.l.o.g. $r>0$, $s,t \leq 0$).

Clearly, $r+s+2t, r+t+2s <0$ and $\Upsilon'(s),\Upsilon'(t)\leq 0$. Thus, the only summand in the right-hand side of \eqref{formula3starPsi2} that may be negative is given by $\Upsilon'(r)(s+t+2r)$. By convexity we deduce
\begin{align*}
\Upsilon(-t)+\Upsilon(-s) + \Upsilon'(r)(s+t+2r) &\geq 2 \Big( \Upsilon\big(-\frac{t+s}{2}\big) + \Upsilon'(r) \big( \frac{s+t}{2}+r\big)\Big)\\
&\geq 2 \Big( \Upsilon(r) - \Upsilon'(r) \big( \frac{s+t}{2}+r\big)+ \Upsilon'(r) \big( \frac{s+t}{2}+r\big)\Big)\\ 
&= 2\Upsilon(r) \geq 0
\end{align*}
and hence the term $\Psi_{2,\Upsilon}(f)(x_*)+\kappa \Psi_\Upsilon(f)(x_*)$ is non-negative by \eqref{formula3starPsi2}.

\textit{2. case:} One element of $r,s,t$ is negative while the other two are non-negative (w.l.o.g. $r<0$, $s,t \geq 0$).

Clearly, $s+t+2r<0$ and $\Upsilon'(r)<0$. Thus, the only summands that may be negative in the right-hand side of \eqref{formula3starPsi2} are given by $\Upsilon'(t)(r+s+2t)$ and $\Upsilon'(s)(r+t+2s)$. As $\Upsilon'(t),\Upsilon'(s) \geq 0$, we are left with $-r>s+2t$ resp. $-r>t+2s$. Due to the asymptotic behavior of the following summands there exists a radius $R_0>0$ such that
\begin{equation*}
\Upsilon(-r) + \Upsilon'(t)(r+s+2t) + \Upsilon'(s)(r+t+2s) \geq 0
\end{equation*}
for any $(r,s,t) \in B_{R_0}(0)^C$ with $r<0$, $s,t\geq 0$ and $-r>t+2s$ or $-r>s+2t$. Consequently, $\Psi_{2,\Upsilon}(f)(x_*)+\kappa \Psi_\Upsilon(f)(x_*)$ is non-negative by \eqref{formula3starPsi2} for any $\kappa \geq \kappa_{R_0}$ in the present case.

Overall, we deduce that $CD_\Upsilon(-\kappa_{R_0},\infty)$ holds at $x_*$. In $a_j$, $j \in \{1,2,3\}$ the situation is better. W.l.o.g. let $j=1$. After minimizing $\Psi_{2,\Upsilon}(f)(a_1)$ with respect to   $f(a_2)$ and $f(a_3)$ (analogously as explained in Example \ref{ex:birthdeath}), we read from \eqref{Psi2Formel}
\begin{align*}
2 \Psi_{2,\Upsilon}(f)(a_1) = \Upsilon(f(x_*)-f(a_1)) + \Upsilon(f(a_1)-f(x_*)) \geq 0
\end{align*} 
and hence $CD_\Upsilon(0,\infty)$ holds in $a_1$.
\end{example}
In \cite{HL16} it was shown that an unweighted graph with girth at least $5$ satisfies the Bakry-\'Emery condition $CD(0,\infty)$ if and only if it is either a (finite or infinite) path or a cycle of length at least $5$ or the $3$-star, which has been considered in Example \ref{ex:3star}. Combining Example \ref{ex:nolowerbound} with Example \ref{ex:3star} and Example \ref{ex:Ricci} (paths and cycles are Ricci-flat), the analogue of the mentioned result regarding the $CD_\Upsilon$ condition is the following.
\begin{corollary}
Let the underlying graph to $L$ be unweighted with girth at least $5$. Then $L$ satisfies $CD_\Upsilon(0,\infty)$ if and only if the underlying graph  is either a (finite or infinite) path or a cycle of length at least $5$.
\end{corollary}
\begin{example}[\textit{Weighted stars}]\label{ex:weightedstars}
This example is concerned with a weighted star. Precisely, let $X=\{x_*,a_1,...,a_m\}$, $m \geq 3$, such that $k(x_*,a_i),k(a_i,x_*)>0$ for any $i\in \{1,...,m\}$ and $k(a_i,a_j)=0$ whenever $i \neq j$. 
We assume that for some $\kappa>0$ 
\begin{equation}\label{curvaturestarcondition}
k(a_i,x_*) - (M_1(x_*) - k(x_*,a_i)) \geq \kappa
\end{equation}
and \begin{equation}\label{curvleafstarcondition}
k(x_*,a_i) \geq \frac{\kappa}{1+\sqrt{3}}
\end{equation}
hold respectively for any $i\in \{1,...,m\}$.

First, we show $\Psi_{2,\Upsilon}(f)(x_*) \geq \kappa \Psi_\Upsilon(f)(x_*)$ for any $f \in \R^X$. We use the notation $z_i= f(a_i)-f(x)$ for any $i \in \{1,...,m\}$. From \eqref{Psi2Formel} we observe
\begin{align*}
2 \Psi_{2,\Upsilon}(f)(x_*)&= \sum_{i=1}^m k(x_*,a_i) k(a_i,x_*) \big( \Upsilon(-z_i) + \Upsilon'(z_i)z_i \big) \\
&\quad + \sum_{i=1}^m k(x_*,a_i) \Upsilon'(z_i)  \sum_{j=1}^m k(x_*,a_j) z_j - M_1(x_*) \sum_{i=1}^m k(x_*,a_i) \Upsilon(z_i).
\end{align*}
Applying the convexity inequality \eqref{mixedterm_birthdeath}, yields
\begin{align*}
\sum_{i=1}^m\sum_{\substack{j=1 \\ j \neq i}}^m k(x_*,a_i)k(x_*,a_j)\Upsilon'(z_i)z_j \geq \sum_{i=1}^m k(x_*,a_i) (M_1(x_*) - k(x_*,a_i)) \Big( \Upsilon(z_i) - \Upsilon(-z_i) - \Upsilon'(z_i)z_i \Big).
\end{align*}
Consequently, we observe
\begin{align*}
2 \Psi_{2,\Upsilon}(f)(x_*)&\geq \sum_{i=1}^m k(x_*,a_i) \Big[ \Upsilon'(z_i)z_i \big( k(a_i,x_*)- (M_1(x_*) - k(x_*,a_i))+ k(x_*,a_i)\big)\\ 
&\quad\qquad\qquad\quad + \Upsilon(-z_i) \big( k(a_i,x_*)- (M_1(x_*) - k(x_*,a_i))\big)- \Upsilon(z_i)k(x_*,a_i)\Big] \\ 
&\geq \kappa \sum_{i=1}^m k(x_*,a_i) \big( \Upsilon'(z_i)z_i + \Upsilon(-z_i)\big)
\end{align*}
by \eqref{curvaturestarcondition} and the convexity of $r \mapsto \Upsilon(r)$. Now, the last term is greater than or equal to $2\kappa \Psi_\Upsilon(f)(x_*)$ by Lemma \ref{App_notlin}(ii).

Next, we aim to show $CD_\Upsilon(\kappa,\infty)$ at $a_j$ for  $j \in \{1,...,m\}$. Let $z_*= f(x_*)-f(a_j)$. After minimization over the second neighborhood (cf. Example \ref{ex:birthdeath}), we observe
\begin{align*}
2\Psi_{2,\Upsilon}(f)(a_j) &\geq \sum_{\substack{i=1 \\ i \neq j}}^m k(a_j,x_*)k(x_*,a_i) \big( \Upsilon(z_*) - \Upsilon'(z_*)z_*\big) \nonumber \\&\quad + k(a_j,x_*)k(x_*,a_j) \big( \Upsilon(-z_*) + \Upsilon'(z_*)z_*\big)\nonumber \\
&\quad + k(a_j,x_*)^2 \big(\Upsilon'(z_*)z_* - \Upsilon(z_*)\big)\nonumber \\
&= k(a_j,x_*) \Big[ \Upsilon(z_*) \big( M_1(x_*) - k(x_*,a_j) - k(a_j,x_*)\big) + \Upsilon(-z_*) k(x_*,a_j)\nonumber \\&\quad+ \Upsilon'(z_*)z_* \big(k(a_j,x_*) - (M_1(x_*) - k(x_*,a_j)) + k(x_*,a_j)\big) \Big]\nonumber\\
&\geq k(a_j,x_*) \big( \Upsilon'(z_*) z_* (k(x_*,a_j)+\kappa) + k(x_*,a_j) \Upsilon(-z_*) - \kappa \Upsilon(z_*)\big)
\end{align*} 
by \eqref{curvaturestarcondition} and the convexity of  $r \mapsto \Upsilon(r)$. Now, the last term is greater than or equal to $2\kappa  \Psi_\Upsilon(f)(a_j)$ if and only if 
\begin{equation}\label{nucdleafstar}
\nu_{1+\frac{\kappa}{k(x_*,a_j)},3 \frac{\kappa}{k(x_*,a_j)}} (z_*) \geq 0.
\end{equation}
It is straightforward to deduce from Lemma \ref{App_notlin}(ii) (by setting $\lambda = 1+ \frac{\kappa}{k(x_*,a_j)}$), that \eqref{nucdleafstar} is satisfied by the condition \eqref{curvleafstarcondition}.
\end{example}

\section{Power type entropies and Beckner inequalities} \label{BecknerSec}
Recall that we assume throughout this section that the invariant and reversible measure $\mu$ is a probability measure on $X$ and $(P_t)_{t \geq 0}$ is a Markov semigroup. Recall also the notation which have been introduced in the beginning of Section \ref{sec:entrodecayandmLSI}.

Beckner inequalities, originally introduced in \cite{Bec}, interpolate between logarithmic Sobolev inequalities and Poincar\'e inequalities. In the diffusive situation of \cite{BGL} they can be derived from the classical Bakry-\'Emery condition with a positive curvature constant, see for instance \cite{AMTU}. As in the case of the Boltzmann entropy, which we have investigated in Section \ref{sec:entrodecayandmLSI}, the Bakry-\'Emery curvature-dimension condition is not suitable in the discrete setting due to the lack of chain rule. Following the approach of \cite{CaDP} and \cite{FaMa}, in \cite{JuWe} discrete Beckner inequalities were derived in the setting of Markov chains with finite state space.

In this section we identify the corresponding operators and curvature-dimension conditions which enable us to argue similarly as in Section \ref{sec:entrodecayandmLSI} to obtain Beckner-type inequalities. 
To that aim, we introduce the entropy functionals
\begin{equation}\label{pEntropy}
\mathcal{H}_p(\rho)= \int_X  \Phi_p(\rho) d\mu  =  \sum\limits_{x \in X} \Phi_p(\rho)(x) \pi(x),\quad \rho \in \mathcal{P}(X),
\end{equation}
where $\Phi_p : [0,\infty) \to \R$ is given by
\begin{align*}
\Phi_p (r) = \frac{r^p - r}{p(p-1)} ,\quad p \in (1,2).
\end{align*}
Those functions are included in the class of so called admissible functions in the context of $\Phi$-entropy inequalities, cf.\ \cite{BoGe}.
\begin{defi}\label{p-operators}
For any $f \in \ell^{\infty,+}(X)$, we define the operators
\begin{equation*}
\begin{split}
\Psi_{\Upsilon}^{(p)}(f) &= -  \overline{\Lambda}_{\Phi_p'}(f),\\
\Psi_{2,\Upsilon}^{(p)}(f) &= \frac{1}{2}\Big( L \Psi_{\Upsilon}^{(p)}(f) - B_{\Upsilon'}(\log f , L \Phi_p ' (f)) \Big),
\end{split}
\end{equation*}
where $\overline{\Lambda}_{\Phi_p'}(f):= \sum\limits_{y \in X}k(x,y) \Lambda_{\Phi_p'}(f(y),f(x))$.
\end{defi}
\begin{remark} Note that the structural similarity of $\Psi_\Upsilon$ and $\Psi_{\Upsilon}^{(p)}$ is hidden in Lemma \ref{PalmeFI} since $\Psi_\Upsilon(\log f) = - \overline{\Lambda}_{\log}(f)$ holds for $f \in \ell^{\infty,+}(X)$. Further,
as
 $\Phi_p \to \Phi$ resp. $\Phi_p'\to \Phi'$ as $p \to 1$ in the pointwise sense,  where  $\Phi(r)=r \log r$, $r>0$, is the function that generates the Boltzmann entropy, the operators $\Psi_\Upsilon$ and $\Psi_{\Upsilon}^{(p)}$ resp. $\Psi_{2,\Upsilon}$ and $\Psi_{2,\Upsilon}^{(p)}$ are related in the following sense. For any  $f \in \ell^{\infty,+}(X)$ we deduce by means of the dominated convergence theorem (recall \eqref{ASSsumfinite} and \eqref{ASSsecsumfinite}) that $L \Phi_p'(f) \to L \log(f)$, $\Psi_{\Upsilon}^{(p)}(f) \to \Psi_\Upsilon (\log f)$ and $\Psi_{2,\Upsilon}^{(p)}(f) \to \Psi_{2,\Upsilon}(\log f) $ as $p \to 1$ in the pointwise sense, respectively.
\end{remark}
The functional 
\begin{equation*}
\mathcal{I}_p(\rho)= \frac{1}{2-p}\int_X \rho\, \Psi_{\Upsilon}^{(p)}(\rho) d\mu, \quad \rho \in \mathcal{P}_*^+(X) 
\end{equation*}
will play the role of the Fisher information as in Section \ref{sec:entrodecayandmLSI}. We have the relation
\begin{equation}\label{pFisherInfoandDirichlet}
\mathcal{E}(\rho, \Phi_p'(\rho))= \mathcal{I}_p(\rho).
\end{equation}
This will be shown in the proof of Theorem \ref{theo:BecknerEntropiederivatives}. As in Section \ref{sec:entrodecayandmLSI}, we will use \eqref{pFisherInfoandDirichlet} to extend $\mathcal{I}_p$ to functions $\rho \in \mathcal{P}_*(X)$,  where $\mathcal{I}_p(\rho)=\infty$ is allowed.
\begin{defi}
We say that $L$ satisfies the Beckner inequality $\mathrm{Bec}(\alpha)$ with $\alpha>0$, if 
\begin{equation}\label{Beckner}
\mathcal{H}_p(f) \leq \frac{1}{2\alpha}\mathcal{I}_p(f)
\end{equation}
holds for any $f \in \mathcal{P}_*(X)$ with $\mathcal{H}_p(f)<\infty$.
\end{defi}
\begin{remark}
Note that \eqref{Beckner} is in fact equivalent to the discrete Beckner inequality considered in \cite{JuWe}. This follows from the relation \eqref{pFisherInfoandDirichlet}.
\end{remark}
\begin{lemma}\label{Becapproximation}
Assume that \eqref{Beckner} holds for $\alpha>0$ and any $f \in P_*^+(X)$. Then $L$ satisfies $\mathrm{Bec}(\alpha)$.
\end{lemma}
\begin{proof}
This follows by the analogous truncation argument as it has been described in the proof of Lemma \ref{MLSIapproximation}.
\end{proof}
%
%
%
\begin{defi}\label{def:p-cdcondition}
The Markov generator $L$ is said to satisfy $CD_\Upsilon^{(p)}(\kappa,\infty)$ at $x \in X$ for $\kappa \in \R$, if
\begin{equation}\label{CDpequation}
\Psi_{2,\Upsilon}^{(p)}(f)(x) \geq \frac{\kappa}{2-p} \Psi_{\Upsilon}^{(p)}(f)(x)
\end{equation}
holds for any $f \in \ell^{\infty,+}(X)$. If $L$ satisfies $CD_\Upsilon^{(p)}(\kappa,\infty)$ at any $x \in X$, then we say that $L$ satisfies $CD_\Upsilon^{(p)}(\kappa,\infty)$.
\end{defi}
\begin{remark}
(i) The assumptions \eqref{ASSsumfinite} and \eqref{ASSsecsumfinite} guarantee that $\Psi_{\Upsilon}^{(p)}$ and $\Psi_{2,\Upsilon}^{(p)}$ are both well defined. We remark, as in Section \ref{sec:basicsandCD}, that one may choose different function spaces that possibly depend on the kernel. In particular, see Remark \ref{CDponsmallerclass}.

(ii) In case that the underlying graph to $L$ is locally finite, \eqref{CDpequation} holds true on all positive functions $f \in \R^X$ if and only if it is satisfied on $\ell^{\infty,+}(X)$. This follows by an analogous reason as explained in Remark \ref{CDPalmefunctionsspace}(ii).
\end{remark}
Analogously to Theorem \ref{ThmSecTDEnt}, the curvature-dimension condition of Definition \ref{def:p-cdcondition} yields a differential inequality regarding the heat flow along the entropy functional defined in \eqref{pEntropy}.
\begin{theorem}\label{theo:BecknerEntropiederivatives}
Let $M_1 \in \ell^2(\mu)$ and $M_2 \in l^1(\mu)$. Then for any $f \in \mathcal{P}_*^+(X)$, we have
\begin{equation*}
\frac{d}{d t} \mathcal{H}_p(P_t f) = - \mathcal{I}_p(P_t f)
\end{equation*}
and 
\begin{equation*}
\frac{d^2}{d t^2} \mathcal{H}_p(P_t f) = 2 \int_X P_t f \Psi_{2,\Upsilon}^{(p)}(P_t f) d\mu.
\end{equation*}
In particular, if the Markov generator
$L$ satisfies $CD_\Upsilon^{(p)}(\kappa,\infty)$, then the differential inequality
\begin{equation*}
\frac{d^2}{d t^2}\mathcal{H}_p (P_t f) \geq - 2 \kappa  \frac{d}{d t}\mathcal{H}_p (P_t f)
\end{equation*}
holds for any $t>0$ and $f \in \mathcal{P}_*^+(X)$.
\end{theorem}
\begin{proof}
As $L$ is symmetric with respect to the $L^2(\mu)$-inner product,
\begin{equation*}
\frac{d}{d t}\mathcal{H}_p(P_t f)= \int_X \Phi_p'(P_t f) L P_t f d\mu = \int_X P_t f L \Phi_p'(P_t f) d\mu
\end{equation*}
holds true. 
By Lemma \ref{lem:firstFI}, we have
\begin{equation}\label{FIforphip}
L \Phi_p'(P_t f) = \Phi_p''(P_t f) L P_t f + \overline{\Lambda}_{\Phi_p'}(P_t f).
\end{equation}
With regard to the identity \eqref{pFisherInfoandDirichlet}, note that the last term equals $-\mathcal{E}(P_t f, \Phi_p'(P_tf))$.
Besides that, we observe due to $\mu$ being invariant
\begin{align*}
\int_X P_t f \Phi_p''(P_t f) L P_t f d\mu  = (p-1)\int_X \Phi_p'(P_t f) L P_t f d\mu
\end{align*}
and consequently
\begin{align*}
(1-(p-1)) \int_X \Phi_p'(P_t f) L P_t f d\mu = \int_X P_t f \overline{\Lambda}_{\Phi_p'}(P_t f) d\mu,
\end{align*}
which yields
\begin{equation*}
\frac{d}{d t}\mathcal{H}_p(P_t f) = \frac{1}{2-p} \int_X P_t f \overline{\Lambda}_{\Phi_p'}(P_t f) d\mu = -\frac{1}{2-p}\int_X P_t f \; \Psi_{\Upsilon}^{(p)}(P_t f) d\mu.
\end{equation*}
Regarding the second derivative, the following representation formula will be crucial
\begin{align*}
\frac{d}{d t} \overline{\Lambda}_{\Phi_p'}(P_t f) &= \sum\limits_{y \in X} k(x,y) \frac{d}{d t} \Big( \Phi_p'(P_t f)(y) - \Phi_p'(P_t f)(x) - \Phi_p''(P_t f) (x) (P_t f(y) - P_t f(x))\Big) \\
&= \sum\limits_{y \in X} k(x,y) \big( \Phi_p''(P_t f)(y)L P_t f (y) - \Phi_p''(P_t f)(x)L P_t f(x)\big)  \\
&\qquad - \Phi_p'''(P_t f)(x) (L P_t f)(x) \sum\limits_{y \in X} k(x,y) (P_t f(y) - P_t f(x))  \\
&\qquad  - \Phi_p''(P_t f)(x) \sum\limits_X k(x,y) (L P_t f(y) - L P_t f(x)) \\
&= L \big(\Phi_p''(P_t f) L P_t f \big)(x) - \Phi_p'''(P_t f)(x) (L P_t f(x))^2 - \Phi_p''(P_t f)(x) L(L P_t f)(x).
\end{align*}
Then we have
\begin{align*}
&(2-p) \frac{d^2}{d t^2}\mathcal{H}_p(P_t f) = \int_X L P_t f \; \overline{\Lambda}_{\Phi_p'}(P_t f) d\mu + \int_X P_t f \frac{d}{d t} \overline{\Lambda}_{\Phi_p'}(P_t f) d\mu \\
&\quad = \int_X P_t f L (\overline{\Lambda}_{\Phi_p'}(P_t f)) d\mu + \int_X \Phi_p''(P_t f) (L P_t f)^2 d\mu -  \int_X P_t f\; \Phi_p'''(P_t f) (L P_t f)^2 d\mu \\
&\qquad\quad -  \int_X P_t f \; \Phi_p''(P_t f) L (L P_t f) d\mu \\
&\quad = \int_X P_t f L (\overline{\Lambda}_{\Phi_p'}(P_t f)) d\mu + (3-p) \int_X \Phi_p''(P_t f) (L P_t f)^2 d\mu  -  (p-1) \int_X \Phi_p'(P_t f) L (L P_t f) d\mu.
\end{align*}
Invoking \eqref{FIforphip}, we obtain
\begin{align*}
\int_X \Phi_p''(P_t f) (L P_t f)^2 d\mu &= \int_X L \Phi_p'(P_t f) L P_t f d\mu  - \int_X L P_t f \;  \overline{\Lambda}_{\Phi_p'}(P_t f) d\mu \\
&= \int_X P_t f \; L \big(L \Phi_p'(P_t f)\big) d\mu - \int_X P_t f \;L \overline{\Lambda}_{\Phi_p'}(P_t f) d\mu
\end{align*}
and hence
\begin{equation}\label{secondpEntrovorDB}
\frac{d^2}{d t^2}\mathcal{H}_p(P_t f) = \frac{1}{2-p} \Big( (p-2)\int_X P_t f \; L (\overline{\Lambda}_{\Phi_p'}(P_t f)) d\mu  + 2( 2-p) \int_X P_t f \; L(L \Phi_p'(P_t f)) d\mu \Big).
\end{equation}
Further, the detailed balance property yields
\begin{align*}
&2\int_X P_t f \; L(L \Phi_p'(P_t f)) d\mu =2 \sum\limits_{x \in X} P_t f(x) \sum\limits_{y \in X} k(x,y) \big(L \Phi_p'(P_t f)(y) - L \Phi_p'(P_t f) (x)\big) \pi(x) \\
&\qquad\qquad= - \sum\limits_{y \in X} \sum\limits_{x \in X} k(y,x) P_t f(x) \big(L \Phi_p'(P_t f)(x) - L \Phi_p'(P_t f)(y)\big) \pi(y) \\
&\quad\qquad\qquad + \sum\limits_{x \in X} P_t f(x) \sum\limits_{y \in X} k(x,y) \big(L \Phi_p'(P_t f)(y) - L \Phi_p'(P_t f) (x)\big) \pi(x)\\
&\qquad\qquad= - \sum\limits_{y \in X}P_t f(y) \sum\limits_{x \in X} k(y,x) e^{\log P_t f(x)-\log P_t f(y))} 
 \big(L \Phi_p'(P_t f)(x) - L \Phi_p'(P_t f)(y)\big) \pi(y) \\
&\qquad\qquad\quad + \sum\limits_{x \in X} P_t f(x) \sum\limits_{y \in X} k(x,y) \big(L \Phi_p'(P_t f)(y) - L \Phi_p'(P_t f) (x)\big) \pi(x)\\
&\qquad\qquad= - \sum\limits_{y \in X} P_t f(y) \sum\limits_{x \in X} k(y,x) (e^{\log P_t f(x) - \log P_t f(y))} - 1 ) \big(L \Phi_p'(P_t f)(x) - L \Phi_p' (y) \big)\pi(y) \\
&\qquad\qquad= - \int_X P_t f B_{\Upsilon'}\big(\log P_t f, L \Phi_p'(P_t f) \big)d\mu.
\end{align*}
Inserting this into \eqref{secondpEntrovorDB}, we end up with
\begin{align*}
\frac{d^2}{d t^2}\mathcal{H}_p(P_t f) =  \int_ X P_t f \big( L \Psi_{\Upsilon}^{(p)}(P_t f) - B_{\Upsilon'}\big(\log P_t f, L \Phi_p'(P_t f) \big)d\mu = 2 \int_X P_t f \; \Psi_{2,\Upsilon}^{(p)}(P_t f)) d\mu.
\end{align*}
The additional claim now follows by applying $CD_\Upsilon^{(p)}(\kappa,\infty)$.
\end{proof}
Consequently, using the entropy method, as explained in Section \ref{sec:entrodecayandmLSI}, and Lemma \ref{Becapproximation},  we deduce the following result.
\begin{corollary}\label{BecwithCDp}
If the Markov generator $L$ satisfies $CD_\Upsilon^{(p)}(\kappa,\infty)$, $\kappa > 0$, $M_1 \in \ell^2(\mu)$ and $M_2 \in l^1(\mu)$, then $L$ satisfies $\mathrm{Bec}(\kappa)$.
\end{corollary}
\begin{remark}\label{CDponsmallerclass}
Note that we only applied the $CD_\Upsilon^{(p)}$ condition on functions belonging to $\mathcal{P}_*^+(X)$. Thus, in order to get the statements of Theorem \ref{theo:BecknerEntropiederivatives} and Corollary \ref{BecwithCDp} respectively, we could formulate Definition \ref{def:p-cdcondition} on the smaller class of $\mathcal{P}_*^+(X)$.
\end{remark}
\section{Tensorization of the discrete and diffusion setting} \label{tensorhybrid}
Here, we revisit the tensorization procedure from Section \ref{tensandgrad} in a different spirit. More generally speaking, given two independent Markov processes $(Z_t)_{t \geq 0}$ on a state space $E$ and $(Y_t)_{t \geq 0}$ on a state space $F$, respectively, each of which with a given unique invariant measure, the product $(Z_t,Y_t)_{t \geq 0}$ defines a Markov process on the state space $E \times F$. Further, the generator of this process is given by $L_Z \oplus L_Y$, where $L_Z $ denotes the generator of $(Z_t)_{t \geq 0}$ and $L_Y$ denotes the generator of $(Y_t)_{t \geq 0}$. The invariant measure is given by the product measure of the respective invariant measures and reversibility is valid provided that it is valid for both Markov processes respectively.

The interest in this section lies in the product of two Markov processes where one fits into the classical diffusive setting of \cite{BGL} and the other one to the discrete Markov chain setting, which has been studied so far in this paper.
In this case we will also speak of {\em hybrid processes}.

As a motivating example let us consider a linear reaction-drift-diffusion system of the form
\begin{equation} \label{RDsys}
\partial_t \rho_i-\Delta \rho_i-\mbox{div}\,(\rho_i \nabla V)=\sum_{j=1}^n \alpha_{i,j} \rho_j\quad \mbox{in}\;
(0,\infty)\times \iR^d,\;i=1,2,\ldots,n,
\end{equation}
for the chemical concentrations $\rho_i(t,x)$. Here $V\in C^2(\iR^d)$ denotes the confining potential. We assume that 
$e^{-V}\in L^1(\iR^d)$. The reaction rates are such that the total mass 
$\int_{\iR^d}\sum_{i=1}^n\rho_i(t,x)\,dx$ 
is conserved, that is we ask for $\sum_{i=1} \alpha_{i,j}=0$ for all $j\in \{1,\ldots,n\}$. Furthermore, we assume that
  $\alpha_{i,i}<0$ for all $i$, and $\alpha_{i,j}\ge 0$ whenever $i\neq j$. 
  
A key idea, which has also been used in \cite{Frei} and \cite{Soner}, is now to introduce a new variable for the coordinates in \eqref{RDsys}, i.e.\ to interpret solutions to \eqref{RDsys} as functions on the product space $[0,\infty)\times \iR^d \times X$,
where $X=\{1,2,\ldots,n\}$. Then we can reformulate \eqref{RDsys} into
\begin{equation} \label{FPsum}
\partial_t \varrho (t,x,i) = (L^*_c + L^*_d ) \varrho(t,x,i),
\end{equation}
where $\varrho: [0,\infty) \times \iR^d \times X \to \R$ is given by $ \varrho (t,x,i)=\rho_i(t,x)$. Here the operator $L^*_c$
acts w.r.t.\ the space variable $x\in \iR^d$ and takes the form
$L^*_c v=\Delta v+\mbox{div}\,(v \nabla V)$, whereas $L^*_d$ acts w.r.t.\ the component variable $i\in X$ and reads as
$(L^*_d v)(i)=\sum_{j=1}^n \alpha_{i,j} v(j)$. Equation \eqref{FPsum} can be viewed as the Fokker-Planck equation
associated with a stochastic process which is the tensor product of a generalized Ornstein-Uhlenbeck process with
generator $L_c v=\Delta v-\nabla V\cdot \nabla v$ and a finite Markov chain with generator
\begin{equation*}
L_d v (i) =\sum\limits_{j \in X} k(i,j) v(j)=
 \sum\limits_{j \in X} k(i,j) \big( v(j) - v(i) \big), \quad i \in X,
\end{equation*}
with transition rates $k(i,j)=\alpha_{j,i}$. Note that $L^*_c$ is the adjoint of $L_c$ w.r.t.\ the Lebesgue measure in $\iR^d$
(see e.g.\ \cite{Jn})
and $L^*_d$ is the adjoint of $L_d$ w.r.t.\ the counting measure on $X$. 

If $V$ is uniformly convex with Hessian $\nabla ^2 V\ge \lambda>0$ (in the sense of positive definite matrices) then
the generalized Ornstein-Uhlenbeck process generated by $L_c$ satisfies $CD(\lambda,\infty)$, where
the unique invariant measure has the density $Ce^{-V(x)}$ w.r.t.\ the Lebesgue measure with a normalizing constant $C>0$.
Let us suppose that the Markov chain satisfies the assumptions of Section \ref{sec:basicsandCD} and has positive curvature
in the $CD_\Upsilon$ sense. Is there a notion of curvature (bounds) for such hybrid processes which imply modified
logarithmic Sobolev inequalities and exponential decay of the entropy and enjoys a natural tensorization principle similar
to Section \ref{tensandgrad}. In particular, can we expect decay of the (relative) entropy for the  reaction-drift-diffusion system
\eqref{RDsys} under the described assumptions? We will answer both questions in the affirmative.

In what follows, let $L_c$ be the generator of a Markov process on the state space $E$ with invariant and reversible probability measure $\mu_c$. We assume that $L_c$ satisfies the diffusion property, cf.\ \cite{BGL}. Let further $\Gamma$ and $\Gamma_2$ denote the corresponding carr\'e du champ and iterated 
carr\'e du champ operator, respectively. Besides that, let $L_d$ denote  the generator of a positive recurrent Markov chain on the  state space $X$ as in \eqref{def:generator}, with invariant and reversible probability measure $\mu_d$.

In the following lines we will remain on a formal level with respect to the class of (admissible) functions. Note however, that in case of finite $X$ we could choose $u:E\times X \to \R$ such that $u(\cdot,i)$ is in the function space $\mathcal{A}$ resp. $\mathcal{A}_0^{const+}$, which are defined in \cite{BGL}, for any $i \in X$.

We now introduce operators which are natural substitutes for the carr\'e du champ and iterated 
carr\'e du champ operator in the hybrid case. Here we use the notation $A \oplus B$ analogously to Section \ref{tensandgrad},
where now $A$ acts on the continuous variable and $B$ on the discrete one. It turns out that the natural replacement of the carr\'e du champ operator in the hybrid case is just given by $\Gamma \oplus \Psi_\Upsilon$. The analogue of
the iterated 
carr\'e du champ operator is introduced in the following definition.
\begin{defi}
For sufficiently regular functions $u: E \times X \to \R $ we define
\begin{equation*}
\big(\Gamma \oplus \Psi_\Upsilon\big)_2 (u)= \frac{1}{2}\Big((L_c \oplus L_d) (\Gamma \oplus \Psi_\Upsilon)(u) - 2 \Gamma \big(u, (L_c \oplus L_d) (u) \big) - B_{\Upsilon'}\big(u, (L_c \oplus L_d) (u) \big)\Big).
\end{equation*}
\end{defi}
The next Lemma plays the role of Lemma \ref{lem:iteratedcdcestimate} in Section \ref{tensandgrad}.
\begin{lemma}\label{contdiscestimate}
We have
\begin{equation*}
\big(\Gamma \oplus \Psi_\Upsilon \big)_2 (u) \geq \big(\Gamma_2 \oplus \Psi_{2,\Upsilon}\big) (u)
\end{equation*}
for sufficiently regular functions $u: E \times X \to \R$.
\end{lemma}
\begin{proof}
Throughout this proof, the symbols $L_c$, $\Gamma$ and $\Gamma_2$ will always refer to the continuous variable,
whereas $L_d$, $\Psi_\Upsilon$, $\Psi_{2,\Upsilon}$ and $B_{\Upsilon'}$ will refer to the discrete variable. 

We have
 \begin{align*}
 &2\big( \Gamma \oplus \Psi_\Upsilon \big)_2 (u) \\
 &\qquad\quad= (L_c (\Gamma(u) + \Psi_\Upsilon(u)) + L_d(\Gamma(u) + \Psi_\Upsilon(u)) - 2 \Gamma( u , L_c u + L_d u ) - B_{\Upsilon'} (u , L_c u + L_d u) \\
 &\qquad\quad= 2\Gamma_2 (u) + 2\Psi_{2,\Upsilon} (u) + L_c(\Psi_\Upsilon (u)) + L_d (\Gamma(u)) - 2 \Gamma (u,L_d u ) - B_{\Upsilon'} (u,L_c u). 
 \end{align*}
The diffusion property \eqref{chainL} yields
 \begin{align*}
 L_c (\Psi_\Upsilon(u))(\cdot,i) &= \sum\limits_{j \in X\setminus \{i\}} k(i,j) L_c \big(\Upsilon( u(\cdot,j)-u(\cdot,i))\big) \\
 &= \sum\limits_{j \in X\setminus \{i\}} k(i,j) \big[\Upsilon'(u(\cdot,j)-u(\cdot,i)) L_c (u(\cdot,j)-u(\cdot,i)) \\
 &\qquad\qquad\qquad+ \Upsilon''(u(\cdot,j)-u(\cdot,i)) \Gamma (u(\cdot,j)-u(\cdot,i)) \big] \\
 &= B_{\Upsilon'}( u, L_c u)(\cdot,i) + \sum\limits_{j \in X\setminus \{i\}}k(i,j) \Upsilon''(u(\cdot,j)-u(\cdot,i)) \Gamma (u(\cdot,j)-u(\cdot,i)).
 \end{align*}
Further, bilinearity yields
 \begin{align*}
 \Gamma(u,L_d u)(\cdot,i) &= \sum\limits_{j \in X\setminus \{i\}} k(i,j) \Gamma (u(\cdot,i),u(\cdot,j)-u(\cdot,i))\\ &= \sum\limits_{j \in X\setminus \{i\}} k(i,j) \big(\Gamma(u(\cdot,i),u(\cdot,j)) - \Gamma(u(\cdot,i)) \big) .
 \end{align*}
 Consequently, we deduce
 \begin{align*}
2\big( \Gamma \oplus \Psi_\Upsilon \big)_2 (u)(\cdot,i)&= 2(\Gamma_2 \oplus \Psi_{2,\Upsilon})(u)(\cdot,i)  \\
&\quad + \sum\limits_{j \in X\setminus \{i\}} k(i,j) \big[ e^{u(\cdot,j) - u(\cdot,i)} \Gamma (u(\cdot,j)-u(\cdot,i)) + \Gamma (u(\cdot,j)) + \Gamma (u(\cdot,i)) \\
&\qquad\qquad\qquad\ - 2 \Gamma(u(\cdot,j),u(\cdot,i)) \big]\\
 &= 2(\Gamma_2 \oplus \Psi_{2,\Upsilon})(u)(\cdot,i) + \sum\limits_{j \in X\setminus \{i\}} k(i,j) \Gamma(u(\cdot,j)-u(\cdot,i)) \big(e^{u(\cdot,j)-u(\cdot,i)} + 1\big)\\
 &\geq 2(\Gamma_2 \oplus \Psi_{2,\Upsilon})(u)(\cdot,i).
\end{align*}
\end{proof}
With Lemma \ref{contdiscestimate} at hand, we are now able to prove the corresponding analogue to Theorem \ref{ThmSecTDEnt}. The entropy associated to the invariant probability  measure $\mu=\mu_c \otimes \mu_d$ is given by
\begin{equation}\label{contdiscentro}
\mathcal{H}(\rho) = \int_E \int_X \rho \log (\rho) d\mu, \quad \rho \in \mathcal{P}(E \times X),
\end{equation}
where $\mathcal{P}(E \times X)$ now denotes all probability densities with respect to $\mu$ (and the convention $0 \log 0 = 0$ being assumed). Further, the Fisher information is defined on sufficiently regular positive functions  $\rho \in \mathcal{P}(E \times X)$ as
\begin{equation*}
\mathcal{I}(\rho)= \int_E \int_X \rho \,(\Gamma \oplus \Psi_\Upsilon)(\log \rho) d\mu. 
\end{equation*}
\begin{theorem}\label{theo:hybridsecderivative}
We have for sufficiently regular positive functions  $f \in \mathcal{P}(E \times X)$
\begin{equation*}
\frac{d}{d t}\mathcal{H}(P_t f) = - \mathcal{I}(P_t f)
\end{equation*}
and
\begin{equation}\label{secondDerivativeconttensordiscrete}
\frac{d^2}{d t^2}\mathcal{H}(P_t f) = 2\int_E \int_X P_t f (\Gamma \oplus \Psi_\Upsilon)_2(\log P_t f) d\mu.
\end{equation}
Moreover, if $L_c $ satisfies $CD(\kappa_c,\infty)$ and $L_d$ satisfies $CD_\Upsilon(\kappa_d,\infty)$, then the following differential inequality holds true
\begin{equation*}
\frac{d^2}{d t^2}\mathcal{H}(P_t f) \geq - 2 \kappa \frac{d}{d t} \mathcal{H}(P_t f),
\end{equation*}
with $\kappa=\min\{\kappa_c,\kappa_d \}$.
\end{theorem}
\begin{proof}
As in the proof of Lemma \ref{contdiscestimate}, the symbols $L_c$, $\Gamma$ and $\Gamma_2$ will refer to the continuous variable,
while $L_d$, $\Psi_\Upsilon$, $\Psi_{2,\Upsilon}$ and $B_{\Upsilon'}$ will refer to the discrete variable. 

Recall, that $\mu= \mu_d \otimes \mu_c$ is invariant for $(P_t )_{t \geq 0}$ and hence
\begin{align*}
\frac{d}{d t}\mathcal{H}(P_t f) &= \int_E \int_X P_t f (L_c \oplus L_d) (\log P_t f) d\mu \\
&= \int_E \int_X P_t f (L_c \log P_t f + L_d \log P_t f )d\mu \\
&=  \int_E \int_X P_t f \Big(\frac{(L_c \oplus L_d)(P_t f)}{P_t f} - (\Gamma \oplus \Psi_\Upsilon)(\log P_t f)\Big)d\mu \\
&= - \int_E \int_X P_t f (\Gamma \oplus \Psi_\Upsilon)(\log P_t f) d\mu ,
\end{align*}
where in the second to last equality the diffusion Property \eqref{logChain} for the term $L_c \log P_t f$ and Lemma \ref{PalmeFI} for the term $L_d \log P_t f$ were used.

Now,  $\frac{d}{d t} \Gamma(\log P_t f)= 2 \Gamma(\log P_t f, \frac{d}{dt}\log P_t f)$ and $\frac{d}{dt}\Psi_\Upsilon(\log P_t f) = B_{\Upsilon'}(\log P_t f , \frac{d}{d t}\log P_t f ) $ together imply
\begin{align*}
\frac{d^2}{d t^2}\mathcal{H}(P_t f) &= - \int_E \int_X  (L_c \oplus L_d) (P_t f) (\Gamma \oplus \Psi_\Upsilon)(\log P_t f) d\mu \\
&\qquad - \int_E \int_X P_t f \big(2 \Gamma(\log P_t f , \frac{d}{d t} \log P_t f ) + B_{\Upsilon'}(\log P_t f , \frac{d}{d t}\log P_t f ) \big)d\mu \\
&= - \int_E \int_X P_t f (L_c \oplus L_d) ( \Gamma \oplus \Psi_\Upsilon) (\log P_t f) d\mu \\
& - \int_E \int_X P_t f \big( 2 \Gamma (\log P_t f , (L_c \oplus L_d) (\log P_t f) ) + B_{\Upsilon'}(\log P_t f, (L_c \oplus L_d) (\log P_t f) )\big) d\mu \\
& - \int_E \int_X P_t f \big( 2 \Gamma(\log P_t f , (\Gamma \oplus \Psi_\Upsilon)(\log P_t f) ) + B_{\Upsilon'}(\log P_t f,  (\Gamma \oplus \Psi_\Upsilon)(\log P_t f)) \big)d\mu,
\end{align*}
where in the latter step \eqref{logChain} and Lemma \ref{PalmeFI} were used.
Applying \eqref{GammaExp} and \eqref{BPalmeIDFormel} as explained in Section \ref{sec:entrodecayandmLSI}, we observe \eqref{secondDerivativeconttensordiscrete}.

If additionally $L_c$ satisfies $CD(\kappa_c,\infty)$ and $L_d$ satisfies $CD_\Upsilon(\kappa_d,\infty)$, with $\kappa_c,\kappa_d > 0$, then we conclude by the aid of Lemma \ref{contdiscestimate} and \eqref{secondDerivativeconttensordiscrete}
\begin{align*}
\frac{d^2}{d t^2}\mathcal{H}(P_t f) &\geq \int_E \int_X P_t f \big( \Gamma_2 (\log P_t f)  + \Psi_{2,\Upsilon}(\log P_t f ) \big) d\mu \\
&\geq 2 \kappa \int_E \int_X P_t f \big(\Gamma \oplus \Psi_\Upsilon \big) (\log P_t f) d\mu \\
&= -2 \kappa \frac{d}{dt}\mathcal{H}(P_t f).
\end{align*}
\end{proof}
Recalling the entropy method, which has been outlined in Section \ref{sec:entrodecayandmLSI}, Theorem \ref{theo:hybridsecderivative} yields the following result.
\begin{corollary}\label{mLSIforhybridprocesses}
Let $L_c $ satisfy $CD(\kappa_c,\infty)$ and $L_d$ satisfy $CD_\Upsilon(\kappa_d,\infty)$, $\kappa_c,\kappa_d > 0$. Then the modified logarithmic Sobolev inequality
\begin{equation*}
\mathcal{H}(f) \leq \frac{1}{2 \kappa}\mathcal{I}(f),
\end{equation*}
with $\kappa=\min\{\kappa_c,\kappa_d\}$, holds true  for sufficiently regular positive functions $f \in \mathcal{P}(E\times X)$.
\end{corollary}

\section{Extensions of the calculus to non-local operators} \label{NonlocalSec}
The aim of this section is to highlight that the calculus developed in Section \ref{sec:basicsandCD} and \ref{sec:entrodecayandmLSI} not only applies to the discrete setting of Markov chains. In fact, we want to point out on a formal level at least, that the ideas can be transferred to non-local operators of the form 
\begin{equation}\label{nonlocaloperator}
\mathcal{L}f(x)=  \int_{\Omega} (f(y)-f(x)) k(x,dy),
\end{equation}
where the integral may have to be understood in the principal value sense.
Here $\Omega \subset \R^d$ is a domain, $x \in \Omega$ and the kernel $k: \Omega \times \mathcal{B}(\Omega)\to [0,\infty]$, where $\mathcal{B}(\Omega)$ denotes the Borel $\sigma$-algebra on $\Omega$, is such that $k(x,\cdot)$ defines a Borel measure on $\Omega $ with $k(x , \{x\})=0$ for any $x \in \Omega$ and $k(\cdot,B)$ is a Borel measurable mapping for any $B\in \mathcal{B}(\Omega)$. 
We impose the reversibility condition $k(x,dy)\mu(dx)= k(y,dx) \mu(dy)$ for a fixed measure $\mu: \mathcal{B}(\Omega) \to [0,\infty]$ in the sense that 
\begin{equation}\label{reversibiltyRd}
\int_{\Omega}\int_{\Omega} f(x,y) k(x,dy)\mu(dx) = \int_{\Omega}\int_{\Omega} f(x,y) k(y,dx)\mu(dy)
\end{equation}
is valid for a sufficiently large class of measurable functions $f:\Omega\times \Omega \to \R$.

We define 
\begin{equation*}
\Psi_\Upsilon (f)(x) = \int_\Omega \Upsilon(f(y)-f(x)) k(x , dy)
\end{equation*} 
and 
\begin{equation*}
\Psi_{2,\Upsilon}(f)(x) = \frac{1}{2} \big( L \Psi_{\Upsilon}(f)(x)- B_{\Upsilon'}(f,Lf)(x) \big),
\end{equation*}
with 
\begin{equation*}
B_{\Upsilon'}(f,g)(x)=  \int_{\Omega} \Upsilon'(f(y)-f(x)) (g(y)-g(x))k(x,dy),
\end{equation*}
where again the respective integrals may be given in the principal value sense.
\begin{example}
One of the most prominent non-local operators is given by the fractional Laplace operator. Here we choose $\Omega = \R^d$, $d \in \N$,  and
\begin{equation*}
k(x,dy)= c_{\beta,d} \frac{dy}{|x-y|^{d+\beta}},\quad x \in \R^d,
\end{equation*}
 with $\beta \in (0,2)$ and  normalizing constant
\begin{equation*}
c_{\beta, d}= \frac{2^\beta \Gamma(\frac{d+\beta}{2})}{\pi^{\frac{d}{2}}\big| \Gamma(-\frac{\beta}{2})\big|},
\end{equation*}
where $\Gamma$ now denotes the Gamma function. Hence, in this case \eqref{nonlocaloperator} reads as
\begin{equation*}
\mathcal{L}(f)(x)= -\big(-\Delta \big)^{\frac{\beta}{2}}(f)(x) = c_{\beta,d}\lim\limits_{\varepsilon \to 0} \int_{\R^d \setminus B_\varepsilon(x)} \frac{f(y)-f(x)}{|x-y|^{d+\beta}}dy,
\end{equation*}
where $f \in \mathcal{S}(\R^d)$, the space of  Schwartz functions.

Since $\Upsilon$ has a quadratic behaviour near $0$, we can write
\begin{equation}\label{fracLapPalme}
\Psi_{\Upsilon}(f)(x) = c_{\beta,d} \int_{\R^d} \frac{\Upsilon(f(y)-f(x))}{|x-y|^{d+\beta}}dy,
\end{equation}
with $f \in \mathcal{S}(\R^d)$, without need to take the principal value.
Further, analogously  to  Proposition \ref{Psi2FormelGrid} in the discrete setting, we have
\begin{equation}\label{fracLapPalme2}
\Psi_{2,\Upsilon}(f)(x)  =   c_{\beta,d}^2\int_{\R^d}\int_{\R^d} \frac{e^{f(x+\sigma)-f(x)}}{2}\frac{\Upsilon \big(f(x+\sigma + h)-f(x+h)-f(x+\sigma) + f(x)\big)}{|h|^{d+\beta}|\sigma|^{d+\beta}}dhd\sigma,
\end{equation}
for $f \in \mathcal{S}(\R^d)$.
Thus, we observe that $- (-\Delta)^{\frac{\beta}{2}}$ satisfies $CD_\Upsilon(0,\infty)$ on $\mathcal{S}(\R^d)$.

In the quite recent work \cite{SWZ2}, it was shown that the fractional Laplacian fails to satisfy $CD(\kappa,N)$ for any $\kappa\in \R$ and  finite $N >0$. By \cite[Lemma 2.1]{SWZ2}, the corresponding operators read as
\begin{equation} \label{eq:fracLapGamma}
 \Gamma(f)(x) = {c_{\beta,d}}\int_{\R^d}\frac{(f(x+h)-f(x))^2}{|h|^{d+\beta}}d h
\end{equation}
and 
 \begin{equation}
  \label{eq:fracLapGamma2}
  \Gamma_2(f)(x) =  {c_{\beta,d}^2}\int_{\R^d}\int_{\R^d}\frac{[f(x+h+\sigma)-f(x+h)-f(x+\sigma)+f(x)]^2}{|h|^{d+\beta}|\sigma|^{d+\beta}}d hd \sigma,
 \end{equation}
 which can also be obtained from the formulas \eqref{fracLapPalme} and \eqref{fracLapPalme2} with the scaling argument that has been used in the proof of Proposition \ref{CDPalmeandCD}.
 
 It is not hard to see that the fractional Laplacian does not satisfy $CD(\kappa,\infty)$ if $\kappa>0$.
Indeed, a scaling argument shows that $\Gamma(v)(x) = \lambda^\beta \Gamma(u)(x)$ and $\Gamma_2(v)(x) = \lambda^{2\beta} \Gamma_2(u)(x)$, where $v(x)=u(\lambda x)$ and $u \in \mathcal{S}(\R^d)$ (cf.\ the proof of Proposition 2.3 in \cite{SWZ2}). Suppose that $CD(\kappa,\infty)$ holds true. Then
\begin{equation*}
\lambda^\beta \Gamma_2(u)(x) \geq \kappa \Gamma(u)(x)
\end{equation*}
for any $\lambda>0$, which implies that $\kappa \leq 0$. This shows in particular, that $CD_\Upsilon(0,\infty)$ is best possible for the fractional Laplacian.

Formally the latter follows as well by applying $\Psi_{2,\Upsilon}$ and $\Psi_\Upsilon$ respectively to a non-trivial linear function $f:\R^d \to \R$. Then $\Psi_{2,\Upsilon}(f)=0$ while $\Psi_\Upsilon(f)>0$. This formal argument applies of course to  a larger class of kernels that allow for the continuous version of the representation of $\Psi_{2,\Upsilon}$ as in Proposition \ref{Psi2FormelGrid}.
\end{example}
Now, we aim for the corresponding analogue to Theorem \ref{ThmSecTDEnt}. For this purpose, we assume that $\mathcal{L}$ is symmetric w.r.t.\ the $L^2(\mu)$-inner product for a sufficiently large class of functions, where $\mu$ is the unique invariant measure on $\big(\Omega, \mathcal{B}(\Omega)\big)$ satisfying condition \eqref{reversibiltyRd} for the Markov semigroup $\left( P_t \right)_{t \geq 0}$ generated by $\mathcal{L}$. Then we follow exactly the same strategy as in Section \ref{sec:entrodecayandmLSI}, i.e. we differentiate 
\begin{equation*}
\mathcal{H}(P_t f) = \int_{\Omega} P_t f \log(P_t f) d\mu,
\end{equation*}
with respect to time, where $f$ is a sufficiently regular probability density with respect to $\mu$. As indicated before, we will remain on a formal level in the subsequent lines and do not consider regularity issues that may particularly come from a possible singularity of the kernel.
 
Since $\mu$ is invariant and $\mathcal{L}$ is symmetric with respect to the $L^2(\mu)$-inner product, we obtain similarly as in Section \ref{sec:entrodecayandmLSI} 
\begin{equation*}
\frac{d}{d t}\mathcal{H}(P_t f) = - \int_{\Omega} P_t f \Psi_\Upsilon(\log(P_t f)) d\mu
\end{equation*}
for sufficiently regular $f$.

As to the second derivative, we aim for an analogue of Lemma \ref{BPalmeID}. We have 
\begin{align*}
&\int_{\Omega} e^{g(x)}  \int_{\Omega} \big( e^{g(y)-g(x)}-1\big) \big( h(y)-h(x)\big) k(x,dy) \mu(dx) \\
&\qquad = \int_{\Omega}  \int_{\Omega} \big( e^{g(y)} (h(y)-h(x)) - e^{g(x)} (h(y)-h(x)) \big) k(x,dy)\mu(dx)\\
&\qquad = \int_{\Omega}  \int_{\Omega}  e^{g(y)} (h(y)-h(x)) k(x,dy)\mu(dx) - \int_{\Omega} e^{g(x)} \mathcal{L}h(x) \mu(dx),
\end{align*}
where the latter step is valid if $\mathcal{L}(h)(x)$ exists for any $x \in \Omega$ in the principal value sense. Note, that this translates to assume that $\mathcal{L}\Psi_\Upsilon(\log(P_t f))$ exists in the calculation of $\frac{d^2}{d t^2} \mathcal{H}(P_t f)$.
Now, by \eqref{reversibiltyRd} and Fubini's theorem and for $\varepsilon>0$ small enough
\begin{align*}
&\int_{\Omega}\int_{\Omega\setminus B_\varepsilon (x)}  e^{g(y)} (h(y)-h(x)) k(x,dy)\mu(dx) \\
&\qquad = \int_{\Omega}\int_{\Omega} \mathbbm{1}_{B_\varepsilon(x)^C}(y)\;  e^{g(y)} (h(y)-h(x)) k(x,dy)\mu(dx)\\
&\qquad = \int_{\Omega}\int_{\Omega} \mathbbm{1}_{B_\varepsilon(y)^C}(x)\;  e^{g(y)} (h(y)-h(x)) k(y,dx)\mu(dy)\\
&\qquad = \int_{\Omega} e^{g(y)} \int_{\Omega \setminus B_\varepsilon (y)} (h(y)-h(x)) k(y,dx)\mu(dy).
\end{align*}

For a kernel without or with a weak singularity we can repeat the above calculation with $\varepsilon=0$. If the kernel has a non-integrable singularity instead, taking the limit can be quite delicate. We will not discuss this here. On a formal level at least we observe as an continuous analogue to Lemma \ref{BPalmeID} the identity
\begin{equation*}
\frac{1}{2}\int_{\Omega} e^g B_{\Upsilon'}(g,h) d\mu  = -  \int_{\Omega} e^g \mathcal{L} h d\mu,
\end{equation*}
for sufficiently regular functions $g,h$.

With this at hand, we can copy the discrete proof verbatim to the result
\begin{equation*}
\frac{d^2}{d t^2}\mathcal{H}(P_t f) = 2 \int_{\Omega} P_t f \Psi_{2,\Upsilon}(\log P_t f) d\mu .
\end{equation*}

Having Section \ref{sec:examplesection} in mind, a continuous version of the operator considered in Example \ref{ex:kequall} on a bounded domain seems like a natural candidate for an operator of the form \eqref{nonlocaloperator} that satisfies $CD_\Upsilon(\kappa,\infty)$ for some $\kappa>0$. Surprisingly, the following example shows that this is not case. 
This emphasizes the difference between the discrete and continuous setting.
\begin{example}
Let $\Omega \subset \R^d$ be a domain and $k(x,dy)= l(y) dy$, where $l:\Omega \to (0,\infty)$ with $l \in L^1(\Omega)$. We consider the invariant  and reversible measure given by $\mu(dx)=l(x) dx$, \eqref{reversibiltyRd} following from a simple application of Fubini's theorem. We have
\begin{align*}
\mathcal{L}f(x)&= \int_\Omega (f(y)-f(x))l(y)dy, \\
\Psi_\Upsilon(f)(x) &= \int_\Omega \Upsilon(f(y)-f(x)) l(y) dy,\\
B_{\Upsilon'}(f,Lf)(x) &= \int_\Omega \Upsilon'(f(y)-f(x)) \big( \mathcal{L}f(y) - \mathcal{L}f(x) \big) l(y) dy, 
\end{align*}
and using the continuous version of the representation formula \eqref{Psi2Formel} we can calculate
\begin{align*}
2 \Psi_{2,\Upsilon}(f)(x)  &= \int_\Omega \int_\Omega \Big( \Upsilon(f(z)-f(y)) - \Upsilon'(f(y)-f(x)) (f(z)-f(y)) \Big) l(z)l(y)\, dz \,dy \\
&\qquad+ \int_\Omega \int_\Omega \Upsilon'(f(y)-f(x))(f(z)-f(x)) l(y)l(z)\,dy\,
dz\\
&\qquad - \int_\Omega\int_\Omega \Upsilon(f(y)-f(x)) l(y)l(z)\, dy\, dz \\
&= \Vert l \Vert_{1} \int_\Omega \big( \Upsilon'(f(y)-f(x))(f(y)-f(x)) - \Upsilon(f(y)-f(x))\big) l(y)\, dy \\
&\qquad+ \int_\Omega \int_\Omega  \Upsilon(f(z)-f(y))l(y)l(z)\, dz\, dy.
\end{align*}
In particular, $CD_\Upsilon(0,\infty)$ is valid since the mappings $r \mapsto \Upsilon'(r)r-\Upsilon(r)$ and $r \mapsto \Upsilon(r)$ are both non-negative. 

Regarding positive curvature bounds, we note that $\Psi_{2,\Upsilon}(f)(x) \geq \kappa \Psi_\Upsilon(f)(x) $ with  $\kappa > 0$ is equivalent to
\begin{equation}\label{continuouskeqlCDformulation}
\begin{split}
0 &\leq \Vert l \Vert_1 \int_\Omega \Big( \Upsilon'(f(y)-f(x)) (f(y)-f(x))  - \big( 1+ \frac{2\kappa}{\Vert l  \Vert_1}\big) \Upsilon(f(y)-f(x))\Big) l(y) dy \\
&\qquad\qquad+ \int_\Omega \int_\Omega  \Upsilon(f(z)-f(y))l(y)l(z) dz dy.
\end{split}
\end{equation}
Fix some $\alpha<0$. Then we find some $r_\alpha<0$ such that $\Upsilon'(r_\alpha)r_\alpha - (1+\frac{2\kappa}{\Vert l \Vert_1})\Upsilon(r_\alpha) \leq \alpha < 0$.

We define for $\varepsilon>0$ with $B_\varepsilon(x)\subset \Omega$ the mapping $f_\varepsilon: \Omega \to \R$ as $f_\varepsilon(x)= -r_\alpha$ and $f_\varepsilon(y) = 0$ for any $y \in \Omega \setminus B_\varepsilon(x)$. Further, let $0\leq f_\varepsilon(y) \leq - r_\alpha $ for any $y \in \Omega$. The latter ensures that we can find a constant $M_\alpha>0$, which is independent of $\varepsilon$, such that 
\begin{equation*}
\Upsilon'(f_\varepsilon(y)-f_\varepsilon(x)) (f_\varepsilon(y)-f_\varepsilon(x))  - \big( 1+ \frac{2\kappa}{\Vert l  \Vert_1}\big) \Upsilon(f_\varepsilon(y)-f_\varepsilon(x)) \leq M_\alpha
\end{equation*}
holds for any $y \in \Omega$. Note that we can even choose $f_\varepsilon \in C_c^\infty(\Omega)$.
We observe 
\begin{align*}
&\Vert l \Vert_1 \int_\Omega \Big( \Upsilon'(f_\varepsilon(y)-f_\varepsilon(x)) (f_\varepsilon(y)-f_\varepsilon(x))  - \big( 1+ \frac{2\kappa}{\Vert l  \Vert_1}\big) \Upsilon(f_\varepsilon(y)-f_\varepsilon(x))\Big) l(y) dy \\
&\quad= \Vert l \Vert_1 \Big( \Upsilon'(r_\alpha)r_\alpha - (1+\frac{2\kappa}{\Vert l \Vert_1})\Upsilon(r_\alpha)\Big) \int_{\Omega \setminus B_\varepsilon(x)} l(y) dy \\ &\quad\quad+ \Vert l \Vert_1 \int_{B_\varepsilon(x)} \Big( \Upsilon'(f_\varepsilon(y)-f_\varepsilon(x)) (f_\varepsilon(y)-f_\varepsilon(x))  - \big( 1+ \frac{2\kappa}{\Vert l  \Vert_1}\big) \Upsilon(f_\varepsilon(y)-f_\varepsilon(x))\Big) l(y) dy\\
&\quad\leq \Vert l \Vert_{1}  \alpha \int_{\Omega \setminus B_\varepsilon(x)} l(y) dy + \Vert l \Vert_1 M_\alpha \int_{B_\varepsilon (x)}l(y) dy.
\end{align*}
Besides that, we have
\begin{equation}\label{keqldoubleintegral}
\begin{split}
&\int_\Omega \int_\Omega \big( \Upsilon(f_\varepsilon(z)-f_\varepsilon(y))\big)l(y)l(z) dz dy \\
&= \int_{B_\varepsilon(x)}\int_\Omega \Upsilon(f_\varepsilon(z)-f_\varepsilon(y))l(y)l(z)dzdy + \int_{\Omega \setminus B_\varepsilon(x)} \int_{B_\varepsilon(x)}\Upsilon(f_\varepsilon(z)-f_\varepsilon(y))l(y)l(z)dzdy.
\end{split}
\end{equation}
By dominated convergence (note that $f_\varepsilon$ is uniformly bounded with respect to $\varepsilon$) the right-hand side
in \eqref{keqldoubleintegral} tends to $0$ as $\varepsilon \to 0$. Further, $\Vert l \Vert_1 M_\alpha \int_{B_\varepsilon(x)}l(y)dy \to 0$ as $\varepsilon \to 0$. Since $\alpha<0$, we conclude that there exists some $\varepsilon_0 > 0$ such that for any $\varepsilon \in (0,\varepsilon_0]$ \eqref{continuouskeqlCDformulation} with $f=f_\varepsilon$ is not valid.

As $x \in \Omega$ was chosen arbitrary, we deduce that $CD_\Upsilon(\kappa,\infty)$ does not hold at any $x \in \Omega$ with $\kappa>0$ arbitrary. 
\end{example}


\section{Miscellanea} \label{gemischtes}
\subsection{Link to M\"unch's $\Gamma^\psi$-calculus}
Here, we demonstrate the relation of the $CD_\Upsilon$ condition with M\"unch's $\Gamma^{\psi}$-calculus of \cite{Mn1}, which we already  described in the introduction and Remark \ref{commentsCD}(i).

To that aim, we first recall the notation of \cite{Mn1}. The graph Laplacian for an unweighted finite graph is given by 
\begin{equation}\label{MünchLaplacian}
\Delta f(x)=\sum_{y\sim x}\big(f(y)-f(x)\big).
\end{equation}
Here, the notation $y \sim x$, $x,y \in X$, means that $x$ and $y$ are neighbours. Clearly, \eqref{MünchLaplacian} coincides with the Markov generator $L$ when the transition rates $k(x,y)$ for $x\neq y$ are given by $k(x,y)=1$ whenever $x \sim y$ and $0$ otherwise.
Let $\psi\in C^1((0,\infty))$ be a concave function. The $\psi$-Laplacian is defined for positive functions as
\[
\Delta^\psi(f)(x):=\Big(\Delta\Big[\psi\Big(\frac{f}{f(x)}\Big)\Big]\Big)\big(x\big).
\]
With
\[
\bar{\psi}(y):=\psi'(1)(y-1)-\big(\psi(y)-\psi(1)\big)
\]
M\"unch defines
\[
\Gamma^\psi(f):=\Delta^{\bar{\psi}}(f).
\]
It is shown that (cf. \cite[Lemma 3.13]{Mn1})
\[
\Delta^\psi f=\psi'(1)\frac{\Delta f}{f}-\Gamma^\psi(f).
\]
Setting
\[
\Omega^\psi(f)(x):=\Big(\Delta \Big[\psi'\Big(\frac{f}{f(x)}\Big)\cdot \frac{f}{f(x)}
\Big[\frac{\Delta f}{f}-\frac{(\Delta f)(x)}{f(x)}\Big]\Big]\Big)\big(x\big) 
\]
M\"unch then defines
\[
\Gamma_2^\psi(f)=\,\frac{1}{2}\,\Big(\Omega^\psi(f)+\frac{(\Delta f) \Delta^\psi(f)}{f}-\frac{\Delta(f\Delta^\psi (f))}{f}\Big)
\]
and introduces the so-called $CD\psi(d,K)$ condition via the inequality
\begin{equation} \label{MnCDpsi}
\Gamma_2^\psi(f)\ge \frac{1}{d} \big( \Delta^\psi f\big)^2+K \Gamma^\psi(f),
\end{equation}
which has to hold for all positive functions $f$ on $X$. The focus in \cite{Mn1} is on the case $K=0$.

Let us now consider the special case $\psi=\log$. Then we have $\Delta^{\log} (f)=\Delta(\log f)$,
\[
\bar{\psi}(y)=y-1-\log y=\Upsilon(\log y),
\]
in particular $\bar{\psi}(1)=0$, and thus
\begin{align*}
\Gamma^{\log}(f)(x) & =\Delta^{\overline{\psi}}(f)(x)=\Big(\Delta\Big[\overline{\psi}\Big(\frac{f}{f(x)}\Big)\Big]\Big)\big(x\big)\\
& = \sum_{y\sim x} \overline{\psi}\Big(\frac{f(y)}{f(x)}\Big)=\sum_{y\sim x} \Upsilon\big(\log(f(y))-\log(f(x))\big) 
=\Psi_\Upsilon(\log f)(x).
\end{align*}
Moreover, inserting $\psi'(y)=\frac{1}{y}$ into the definition of $\Omega^{\log}$ we find that
\begin{align*}
\Omega^{\log}(f)(x)=\Delta \Big(\frac{\Delta f}{f}-\frac{(\Delta f)(x)}{f(x)}\Big)\big(x\big)=\Delta \Big(\frac{\Delta f}{f}\Big)
\big(x\big).
\end{align*}
Employing Lemma \ref{PalmeFI} and the identity
\[
\Delta(gh)=(\Delta g) h+g \Delta h +2\Gamma(g,h),
\]
we therefore obtain
\begin{align*}
2\Gamma_2^{\log}(f) &=\Delta \Big(\frac{\Delta f}{f}\Big)+\frac{(\Delta f) \Delta (\log f)}{f}-\frac{\Delta(f\Delta(\log f))}{f}\\
& = \Delta \big(\Psi_{\Upsilon}(\log f)+\Delta (\log f)\big)+\frac{(\Delta f) \Delta (\log f)}{f}\\
& \quad\;-\frac{1}{f}(\Delta f) \Delta(\log f)-\Delta \Delta (\log f)-\frac{2}{f}\Gamma\big(f,\Delta(\log f)\big)\\
& =  \Delta \Psi_{\Upsilon}(\log f)-\frac{2}{f}\,\Gamma\big(f,\Delta(\log f)\big).
\end{align*}
Finally, 
\begin{align*}
\frac{2}{f}\,\Gamma\big(f,\Delta(\log f)\big)& =\sum_{y\sim x} \frac{1}{f(x)}\big((f(y)-f(x)\big)\big((\Delta(\log f))(y)
-(\Delta(\log f))(x)\big)\\
& =\sum_{y\sim x} \big(e^{(\log f)(y)-(\log f)(x)}-1\big)\big((\Delta(\log f))(y)
-(\Delta(\log f))(x)\big)\\
& = B_{\Upsilon'}\big(\log f, \Delta(\log f)\big),
\end{align*}
and so we conclude that
\[
\Gamma_2^{\log}(f)=\Psi_{2,\Upsilon}(\log f),
\]
which together with $\Gamma^{\log}(f)=\Psi_{\Upsilon}(\log f)$ shows that $CD\psi(\infty,\kappa)$ with $\psi=\log$ is equivalent to $CD_\Upsilon(\kappa,\infty)$.

\medskip
\subsection{Link to entropic Ricci curvature}
For $\rho \in \mathcal{P}(X)$ and $\psi \in \R^X$, as we have indicated in the introduction and Remark \ref{commentsCD}(v), the following objects play a fundamental role in \cite{EM12}:
\begin{equation}\label{AErbarMaasA}
\mathcal{A}(\rho,\psi)= \frac{1}{2}\sum_{x,y \in X} \big( \psi(x)-\psi(y)\big)^2 \theta(\rho(x),\rho(y)) k(x,y) \pi(x)
\end{equation}
and
\begin{equation}\label{BErbarMaasB}
\begin{split}
\mathcal{B}(\rho,\psi) = &\frac{1}{4}\sum_{x,y \in X}\big( \psi(x)-\psi(y)\big)^2 \hat{L} \rho (x,y) k(x,y) \pi(x) \\&- \frac{1}{2}\sum_{x,y \in X} \big( L \psi (x) - L \psi(y)\big) (\psi(x)-\psi(y))\, \theta(\rho(x),\rho(y)) k(x,y) \pi(x),
\end{split}
\end{equation}
where $\hat{L}\rho (x,y) = \partial_1 \theta (\rho(x),\rho(y)) L \rho(x) + \partial_2 \theta(\rho(x),\rho(y)) L \rho(y)$
and  for $s,t>0$ the quantity $\theta(s,t):= \int_0^1 s^{1-p}t^p dp= \frac{s-t}{\log s - \log t}$ denotes the logarithmic mean with the latter identity being valid if $s \neq t$.
Note, that by the detailed balance condition and the symmetry of the logarithmic mean, we can rewrite
\begin{align*}
&\frac{1}{4}\sum_{x,y \in X}\big( \psi(x)-\psi(y)\big)^2 \hat{L} \rho (x,y) k(x,y) \pi(x) \\
&\qquad = \frac{1}{2}\sum_{x,y,z \in X} \big( \psi(x)-\psi(y)\big)^2 \partial_1\theta(\rho(x),\rho(y)) (\rho(z)-\rho(x))k(x,y)k(x,z) \pi(x),
\end{align*}
cf. the proof of Proposition 4.3 in \cite{EM12}.

We  aim now to show the identities \eqref{AFormel} and \eqref{BFormel}. Choosing $\psi=\log \rho$ in \eqref{AErbarMaasA}, yields
\begin{align*}
\mathcal{A}(\rho,\log \rho) &= \,\frac{1}{2}\,\sum_{x,y\in X} k(x,y) \big(\rho(y)-\rho(x)\big)\big(\log \rho(y) -\log \rho(x)\big)\pi(x) \\
&= \sum_{x \in X} \rho(x) \Psi_{\Upsilon}(\log \rho)(x) \pi(x)\\
&= \int_X \rho \Psi_\Upsilon(\log (\rho)) d\mu,
\end{align*}
as it has been shown in Section \ref{sec:entrodecayandmLSI}.
Further, we have
\begin{align*}
&-\frac{1}{2}\sum_{x,y \in X} \big( L (\log \rho)(x) - L (\log \rho)(y)\big) (\log \rho(x)-\log \rho(y)) \theta(\rho(x),\rho(y)) k(x,y) \pi(x)\\ &\qquad = - \sum_{x \in X} \Gamma (\rho,L (\log \rho ))(x) \pi(x) = \sum_{x \in X} L(\log \rho)(x) L \rho(x) \pi(x).
\end{align*}
Since $\partial_1\theta(s,t)=\frac{\log s - \log t - \frac{s-t}{s}}{(\log s - \log t)^2 }$, we observe
\begin{align*}
&\frac{1}{2}\sum_{x,y,z \in X} \big( \log \rho(x)-\log \rho(y)\big)^2 \partial_1\theta(\rho(x),\rho(y)) (\rho(z)-\rho(x))k(x,y)k(x,z) \pi(x)\\
&\qquad = \frac{1}{2} \sum_{x,y,z \in X} \Big( \log \rho (x) - \log \rho (y) - \frac{\rho(x) - \rho(y)}{\rho(x)} \Big) (\rho(z) - \rho(x) ) k(x,y)k(x,z) \pi(x)\\
&\qquad = \frac{1}{2}\sum_{x \in X} \frac{(L \rho (x))^2}{\rho(x)} \pi(x) - \frac{1}{2} \sum_{x \in X} L (\log \rho)(x) L \rho (x) \pi (x).
\end{align*}
Consequently, we get
\begin{equation}\label{BequalsRHSsdentropy}
\mathcal{B}(\rho,\log \rho) = \frac{1}{2}\sum_{x \in X} L (\log \rho)(x) L \rho (x) \pi (x) + \frac{1}{2}\sum_{x \in X} \frac{(L \rho (x))^2}{\rho(x)} \pi(x).
\end{equation}
As it is well known (see e.g. \cite{CaDP}), the right hand side of \eqref{BequalsRHSsdentropy} equals $\frac{1}{2}\left.\frac{d^2}{dt^2}\mathcal{H}(P_t \rho)\right|_{t = 0}$ and hence, by using Theorem \ref{ThmSecTDEnt},
it follows that  $\mathcal{B}(\rho,\log \rho) = \int_X \rho\, \Psi_{2,\Upsilon}(\log \rho) d\mu$.

\appendix
\section{Properties of $\nu_{c,d}$}\label{Appendixnucd}
In this section, we collect some properties of the function $\nu_{c,d}:\R \to \R$, given by
\begin{equation*}
\nu_{c,d}(r)= c \Upsilon'(r)r + \Upsilon(-r) - d \Upsilon(r), \quad c,d \in \R,
\end{equation*}
which was of fundamental importance in Section \ref{sec:examplesection}. We have
\begin{align*}
\nu_{c,d}'(r) &= c \Upsilon''(r)r + \Upsilon'(r) (c-d) - \Upsilon'(-r), \\
\nu_{c,d}''(r) &= e^r ( cr + 2c - d) + e^{-r}, \\
\nu_{c,d}'''(r)&= e^r ( cr + 3c  -d ) -e^{-r}, \\
\nu_{c,d}''''(r) &= e^r ( cr + 4c - d) +e^{-r}. 
\end{align*}
Note that $\nu_{c,d}(0)=0$ and $r=0$ is a critical point for $\nu_{c,d}$.
\begin{lemma}\label{lem:verybasicsnucd}
The following assertions hold:
\begin{itemize}
\item[(i)] If $c \geq d$, then $\nu_{c,d}(r) \geq 0$ for any $r \in \R$,
\item[(ii)] Let $c,d,r \in \R$ such that $\nu_{c,d}(r)\geq 0$. Then we have $\nu_{c+h,d+h}(r) \geq 0$ for any $h \geq 0$.
\item[(iii)] We have $\nu_{c',d'}(r) \geq \nu_{c,d}(r)$ for any $r \in \R$ if $c'\geq c$ and $d'\leq d$.
\end{itemize}
\end{lemma}
\begin{proof}
The first two properties follow immediately from $\Upsilon'(r)r \geq \Upsilon(r)$, $r \in \R$. The third property then follows from the fact that $r\mapsto \Upsilon(r)$ is non-negative.
\end{proof}
\begin{lemma}\label{lem:nucdforpositive}
Let $\lambda \geq 1$. Then we have
\begin{equation*}
\nu_{\lambda,h(\lambda)} (r) \geq 0 
\end{equation*}
for any $r \geq 0$, where 
\begin{equation*}
h(\lambda) = \left\{\begin{array}{ll} 2\lambda +1, & \lambda \geq 2, \\
     3\lambda - 1 , & 1\leq \lambda < 2. \end{array}\right. 
\end{equation*}
\end{lemma}
\begin{proof}
We have
\begin{align*}
\nu_{\lambda,h(\lambda)}''(0)= \left\{\begin{array}{ll} 0, & \lambda \geq 2, \\
     2 - \lambda  & 1\leq \lambda < 2 \end{array}\right.
\end{align*}
and 
\begin{align*}
\nu_{\lambda,h(\lambda)}'''(r)= \left\{\begin{array}{ll} e^r \big( \lambda r + \lambda -1\big) - e^{-r}, & \lambda \geq 2, \\
     e^r \big( \lambda r  + 1\big) - e^{-r}  & 1\leq \lambda < 2. \end{array}\right.
\end{align*}
Note that $\nu_{\lambda,h(\lambda)}'''(r)\geq 0$ is equivalent to
$e^{2r}\big( \lambda r + \lambda - 1\big)\geq 1$ in the case of $\lambda \geq 2$ and, for  $1 \leq \lambda <2$, to $e^{2r} \big( \lambda r + 1 \big)\geq 1$. Both conditions are apparently satisfied for any $r\geq 0$. Hence, we observe $\nu_{\lambda,h(\lambda)}'''(r)\geq 0$ for any $\lambda \geq 1$ and any $r \geq 0$.
This suffices to establish the claim, since we can deduce $\nu_{\lambda,h(\lambda)}''(r) \geq 0$ from $\nu_{\lambda,h(\lambda)}''(0) \geq 0$ and, by the same argument, $\nu_{\lambda,h(\lambda)}'(r)\geq 0$ and $\nu_{\lambda,h(\lambda)}(r) \geq 0$.
\end{proof}
\begin{lemma}\label{lem:nucdlemma}
Let $\lambda \geq 2$. Then $\nu_{\lambda,2\lambda+1}(r) \geq 0$ for any $r \in \R$ if and only if $\lambda=2$.
\end{lemma}
\begin{proof}
We have
\begin{align*}
\nu_{\lambda,2\lambda+1}'(r) &= e^r (\lambda r - (\lambda+1)) +\lambda+2 - e^{-r},  \\
\nu_{\lambda,2\lambda+1}''(r) &= e^r (\lambda r-1)+e^{-r}, \\
\nu_{\lambda,2\lambda+1}'''(r) &= e^r(\lambda r + \lambda -1)-e^{-r}.
\end{align*}
In particular, $\nu_{\lambda,2\lambda+1}(0)=\nu_{\lambda,2\lambda+1}'(0)=\nu_{\lambda,2\lambda+1}''(0)=0$.
Moreover,  $\nu_{\lambda,2\lambda+1}'''(0)= \lambda-2$ and thus $x=0$ is a saddle point whenever $\lambda>2$. This implies that there exist some $\delta_\lambda >0$ such that $\nu_{\lambda,2\lambda+1}(r)<0$ for any $r\in (-\delta_\lambda,0)$.

On the other hand, we have $\nu_{2,5}'''(r)=0$ if and only if $e^{2r}(2r+1)=1$ and a straightforward calculation shows that this only holds for $r=0$, i.e. $\nu_{2,5}'''(0)=0$ and $0$ is the only critical point for $\nu_{2,5}''$. As $\nu_{2,5}''''(0)=4$, $\nu_{2,5}$ has a local minimum at $0$ and is convex on $\R$, which yields the  claim.
\end{proof}
\begin{lemma}\label{App_notlin}
The following assertions hold true:
\begin{itemize}
\item[(i)] Let $\alpha>1$ and $\beta \in \R$. Then there exists some $r_\alpha < 0$ and $\overline{\lambda}>0$ such that
\begin{equation*}
\nu_{\lambda,\alpha \lambda+ \beta}(r_\alpha) < 0
\end{equation*}
for any $\lambda \geq \overline{\lambda}$.
\item[(ii)] We have for any $r \in \R$  and $\lambda \geq 2$
\begin{equation*}
\nu_{\lambda,\lambda+\tau(\lambda)}(r) \geq 0,
\end{equation*}
with $\tau(\lambda)= 2^\frac{3}{2} \sqrt{\lambda}-1$.\\
Moreover, for $\lambda \in [1,2)$, we have $\nu_{\lambda,3\lambda - 1}(r)\geq 0$ for any $r \in \R$.
\end{itemize}
\end{lemma}
\begin{proof}
(i): We have
\begin{equation}\label{Aeq:nolinear}
\nu_{\lambda,\alpha \lambda+ \beta}(r) = \lambda \big( \Upsilon'(r)r - \alpha \Upsilon(r) \big) + \Upsilon(-r) - \beta \Upsilon(r). 
\end{equation}
Since the mappings $r \mapsto \Upsilon'(r) r$ and $r \mapsto \Upsilon(r)$ have the same asymptotic behaviour as $r \to - \infty$, we find some $r_\alpha<0$, such that  $\Upsilon'(r_\alpha)r_\alpha - \alpha \Upsilon(r_\alpha)<0$. Hence, the claim follows from \eqref{Aeq:nolinear}.

(ii): We have $\nu_{\lambda,\lambda+\tau(\lambda)}(r)= \lambda \big(\Upsilon'(r)r- \Upsilon(r)\big) + \Upsilon(-r) - \tau(\lambda) \Upsilon(r)$ and hence
\begin{equation}\label{derivativenulambdatau}
\nu_{\lambda,\lambda + \tau(\lambda)}'(r) = \lambda \Upsilon''(r) r - \Upsilon'(-r) - \tau(\lambda) \Upsilon'(r).
\end{equation}
We aim to show that $\nu_{\lambda, \lambda + \tau(\lambda)}'(r) \leq 0$ for any $r \leq 0$ and any $\lambda \geq 1$. This suffices to deduce the claim, since Lemma \ref{lem:nucdlemma} reduces the claim in the case of $\lambda \geq 2$ to negative arguments because of Lemma \ref{lem:verybasicsnucd}(iii) and  the relation
\begin{equation}\label{lambdatauestimate}
\lambda + \tau(\lambda) \leq 2 \lambda + 1.
\end{equation}
To see the latter, we define $p(\lambda)=\lambda -2^{\frac{3}{2}}\sqrt{\lambda}$, $\lambda \in (0,\infty)$.
We have $p'(\lambda)=1-\frac{\sqrt{2}}{\sqrt{\lambda}}$ and we obtain, that $p$ attains the global minimum, when $\sqrt{\lambda} = \sqrt{2}$. Further, $p(2) = -2$ and hence
\begin{align*}
\lambda - \tau(\lambda) = p(\lambda) +1 \geq -1,
\end{align*}
which implies \eqref{lambdatauestimate}. Further, one checks that $\lambda + \tau(\lambda) \geq 3 \lambda - 1 $ holds true if $1 \leq \lambda < 2 $. Consequently, if $\nu_{\lambda, \lambda + \tau(\lambda)}'(r) \leq 0$ for any $r \leq 0$ is valid, one deduces the claim from Lemma \ref{lem:nucdforpositive} combined with Lemma \ref{lem:verybasicsnucd}(iii).

Thus, replacing $r$ by $-y$ in \eqref{derivativenulambdatau}, it remains to show  
\begin{equation}\label{lambdataudesiredest}
\varphi(y):= \lambda e^{-y} y  + e^y - 1 + \tau(\lambda)(e^{-y}-1) \geq 0
\end{equation}
for any $y\geq 0$. Due to the asymptotic behaviour of $\varphi$ and the fact that $\varphi(0)=0$, it suffices to investigate the critical points of $\varphi$ in $(0,\infty)$. We have 
\begin{align*}
\varphi'(y)= e^{-y} \big(\lambda - \lambda y - \tau(\lambda)\big) + e^y
\end{align*}
and hence, as a necessary condition, $e^{2y} = \lambda y - \lambda + \tau(\lambda)$, i.e.
\begin{equation}\label{necformini}
y_*= \frac{e^{2y_*}+\lambda - \tau(\lambda)}{\lambda}
\end{equation}
must hold at a critical point of $\varphi$. Plugging \eqref{necformini} into \eqref{lambdataudesiredest} yields
\begin{align*}
0 \leq e^{-y_*} \big( e^{2y_*} + \lambda - \tau(\lambda) \big) + e^{y_*} - 1 + \tau(\lambda) ( e^{-y_*}-1) = 2e^{y_*} + \lambda e^{-y_*} - (\tau(\lambda) + 1).
\end{align*}
We define $g(y)=2 e^y + \lambda e^{-y}$. If $\lambda \in [1,2)$ we have $g'(y)>0$ for any $y \geq 0$ and the relation $g(0)\geq \tau(\lambda)+1$ can be readily checked. In case of $\lambda \geq 2$, $g'(y)= 0$ is equivalent to $y= \log\big(\sqrt{\frac{\lambda}{2}}\big)$. Clearly, this is the global minimum point of $g$ in $\R$. We have
\begin{align*}
g\big(\log\big( \sqrt{\frac{\lambda}{2}}\big) \big) =2 \sqrt{\frac{\lambda}{2}} + \lambda \sqrt{\frac{2}{\lambda}} = 2^{\frac{3}{2}} \sqrt{\lambda}.
\end{align*}
Hence, by definition of $\tau(\lambda)$, 
we deduce
\begin{align*}
2e^{y_*} + \lambda e^{-y_*}  \geq \tau(\lambda) + 1,
\end{align*}
which establishes the claim.
\end{proof}

\end{document}